\documentclass[letterpaper,article,11pt]{memoir}
\usepackage{CCLAuthor}
\usepackage[reqno]{amsmath}
\usepackage{lmodern}
\usepackage{amssymb}
\usepackage{amsthm}  
\theoremstyle{plain}
\newtheorem{theorem}{Theorem}

\theoremstyle{definition}
\newtheorem{definition}[theorem]{Definition}

\newtheorem{example}[theorem]{Example}
\numberwithin{theorem}{chapter}
\usepackage[round,authoryear]{natbib} 
\usepackage[]{graphicx}               
\usepackage{multicol}
\usepackage{overpic,cancel}

\usepackage[normalem]{ulem}

\usepackage{color}

\definecolor{orange}{rgb}{0.8,0.3,0.0}
\definecolor{dgreen}{rgb}{0.3,0.9,0.3}

\usepackage[colorlinks=true]{hyperref}
\hypersetup{urlcolor=blue,linkcolor=blue, citecolor=dgreen,pdfstartview=FitH, bookmarksopen=false}

\newcommand{\be}{\begin{equation}}
\newcommand{\ee}{\end{equation}}

\newcommand\udot{\dot{u}}

\newcommand\wdot{\dot{w}}

\newcommand\uddot{\ddot{u}}

\newcommand{\bgamma}{\bar{\gamma}}

\newcommand{\rme}{\textrm{e}} 

\newcommand{\R}{\mathbb{R}}
\newcommand{\N}{\mathbb{N}}

\newcommand{\W}{W^{\!}}

\renewcommand{\epsilon}{\varepsilon}
\renewcommand{\phi}{\varphi}

\newcommand{\cE}{\mathcal{E}}
\newcommand{\cO}{\mathcal{O}}
\newcommand{\ftrans}{\phi_{N_{\!R\!}}}

\newcommand{\etaNP}{\eta_{N_{\!P}}}
\newcommand{\tauNM}{\tau_{N_{\!M}\!}}
\newcommand{\gammaNM}{\gamma_{N_{\!M}}}
\newcommand{\gammaN}{\gamma_{N}}

\newcommand{\tauNP}{\tau_{N_{\!P}}}
\newcommand{\tauN}{\tau_N}

\newcommand{\VN}{V_{\!N_{\!M\!}}}

\newcommand{\uh}{u^{h^{\!}}}
\newcommand{\uhm}{u^{h}_{-\!}}
\newcommand{\uhp}{u^{h^{\!}}_{+\!}}

\numberwithin{equation}{chapter}
\numberwithin{figure}{chapter}
\numberwithin{table}{chapter}

\renewcommand{\geq}{\geqslant}
\renewcommand{\leq}{\leqslant}

\definecolor{dgreen}{rgb}{0.3,0.8,0.3}

\newcommand{\ssout}[1]{\ifmmode\text{\color{red}\sout{\color{red} \ensuremath{#1}}}\else{\color{red}\sout{{\color{red} #1}}}\fi}

\allowdisplaybreaks

\begin{document}
%
%
%
\title{Practicalities of State-Dependent and Threshold Delay Differential Equations}
%
%
\author{%
    A.R. Humphries\textsuperscript{*\dag}
    and 
    A.S. Eremin\textsuperscript{\ddag}
    and 
    Z. Wang\textsuperscript{*}
    \\ \smallskip\small
    \textsuperscript{*}%
    Department of Mathematics and Statistics, McGill University, Montreal, QC, Canada H3A 0B9                     
    \\
    \textsuperscript{\dag}%
    Department of Physiology, McGill University, Montreal, QC, Canada H3G 1Y6
    \\
    \textsuperscript{\ddag}%
    St.~Petersburg State University, St.~Petersburg, 199034, Russian Federation  
}
    \maketitle
%
%
%
%

\begin{abstract}
	Delays are ubiquitous in applied problems, but often do not arise as the simple constant discrete delays that analysts and numerical analysts like to treat. In this chapter we show how state-dependent delays arise naturally when modeling and the consequences that follow. We treat discrete state-dependent delays, and delays implicitly defined by threshold conditions. We will consider modeling, formulation as dynamical systems, linearization, and numerical techniques. For discrete state-dependent delays we show how breaking points can be tracked efficiently to preserve the order of numerical methods for simulating solutions. For threshold conditions we will discuss how a velocity ratio term arises in models, and present a heuristic linearization method that avoids Banach spaces and sun-star calculus, making the method accessible to a wider audience. We will also discuss numerical implementations of threshold and distributed delay problems which allows them to be treated numerically with standard software. 
\end{abstract}
%
%
%
%
%
%
%

\CCLsection*{Acknowledgments}

We are grateful to Dimitri Breda, Rossana Vermiglio and Jianhong Wu for organizing the Advanced School \emph{Delays and Structures
in Dynamical Systems:
Modeling, Analysis and
Numerical Methods} at the International Centre for Mechanical Sciences
(CISM) in November 2023 in Udine, and we thank CISM for their support. 
The work presented here has grown out of collaborations with colleagues and students over many years. We would particularly like to acknowledge the contributions of 
Finn Upham, Orianna DeMasi, Felicia Magpantay, Renato Calleja, Morgan Craig, Tyler Cassidy, Shaza Alsibaai, Bernd Krauskopf, Tomas Gedeon,  
Hans-Otto Walther and Michael Mackey. ARH is supported by a National Science and Engineering Research Council (NSERC) of Canada Discovery Grant.

\CCLsection{Delay Differential Equations}
\label{sec:ddes}

A delay differential equation (DDE) is a differential equation where the evolution of the state depends on past value(s) of the state. A very simple example would be to find $u(t)\in\R$ for $t\geq t_0$ such that
\begin{equation}  \label{eq:dde1}
    \udot(t)=\lambda u(t-\tau),
\end{equation}
where $\lambda\in\R$ is a parameter. If $\tau=0$ this is a linear scalar ordinary differential equation (ODE). Adding an 
initial condition $u(t_0)=u_0$ to the ODE creates an initial value problem (IVP) whose solution (with $\tau=0$) is
\begin{equation}  \label{eq:ode1sol}
    u(t)=e^{\lambda(t-t_0)}u_0.
\end{equation}

When the delay $\tau$ is a positive constant the DDE \eqref{eq:dde1} is still linear,
but no longer trivial to solve. To begin with, because of the delay $\tau$, it is no longer sufficient to specify initial data at a single point. To solve \eqref{eq:dde1} as an IVP for $\tau>0$ it will be necessary to specify an initial function 
\begin{equation} \label{eq:phi}
u(t_0+\theta)=\phi(\theta), \quad \theta\in[-\tau,0],   
\end{equation}
just to be able to evaluate the right hand-side of \eqref{eq:dde1} for each $t\in[0,\tau]$. But since in dynamical systems theory, the phase space of the dynamical system is essentially the space in which the initial conditions live, 
equation \eqref{eq:phi} leads us to a dynamical system defined on a space of functions, and hence to infinite dimensional dynamical systems. 

There is a well developed theory of DDEs as retarded functional differential equations (RFDEs); see \cite{BellmanCooke63,Hale77,HaleLunel93,DGVLW95}. Although this is a wonderful sequence of books, there is progressively more functional analysis to digest as the theory gets deeper, and the later material 
(particularly the sun-star calculus of \cite{DGVLW95}) 
will not be accessible to many applied scientists.  This is problematic as biologists, physiologists and pharmaceutical scientists will not readily embrace models and mathematics that they do not understand. One can apply all the mathematical tools and techniques at one's disposal to solve the problem at hand, but to have
impact as an applied mathematician and to reach the widest possible audience, the solution should be distilled down to the lowest level of mathematics needed to get to and explain the solution. So,
in our view it is best to avoid Banach spaces when talking to physiologists, and it is okay to trade some rigour for traction and progress on applied problems. So, in the current work we will largely eschew Banach spaces and associated analytical technicalities. That said, some notation is very useful.

An RFDE can be stated as 
\begin{equation} \label{eq:rfde}
    \udot(t) = F(t,u_t), \quad t\geq t_0
\end{equation}
where $u_t$ defines the function segment
\begin{equation} \label{eq:ut}
    u_t(\theta)=u(t+\theta), \quad \theta\in[-\tau,0].
\end{equation}
With the 
notation \eqref{eq:ut} the initial condition \eqref{eq:phi} becomes
\begin{equation} \label{eq:rfdeic}
    u_{t_0}=\phi\in C,
\end{equation}
as illustrated in Figure~\ref{fig:rfdeic}.

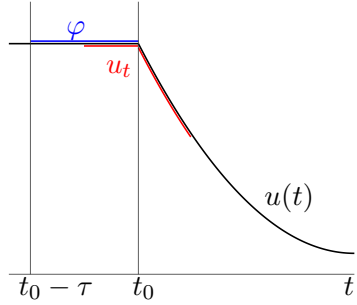
\begin{figure}[t!] 
    \centering	
    \includegraphics[scale=0.4]{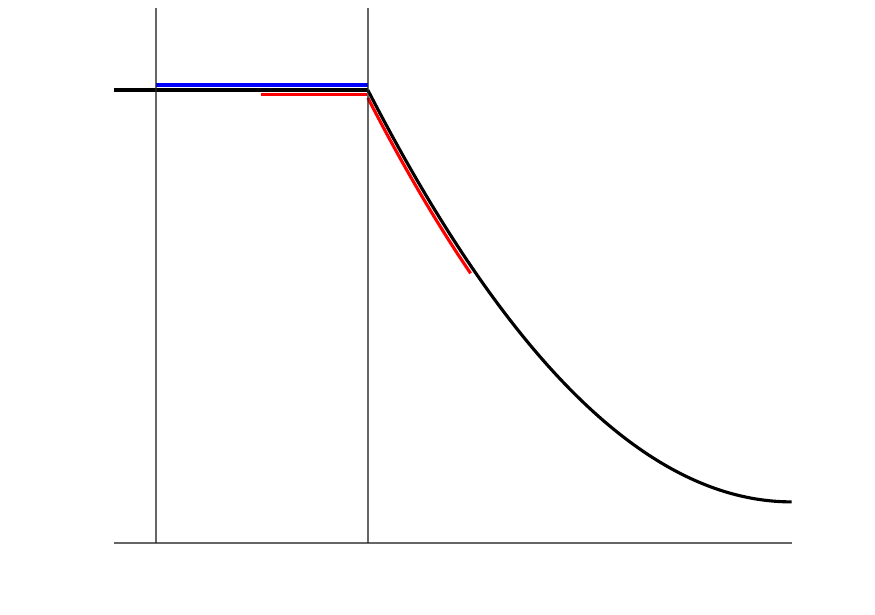}
    \put(-125,105){$\color{blue}\phi$}
    \put(-110,90){$\color{red}u_t$}
    \put(-143,6){$t_0-\tau$}
    \put(-100,6){$t_0$}
    \put(-20,6){$t$}
    \put(-50,40){$u(t)$}
    \caption{Illustration of \eqref{eq:rfdeic}, the initial condition for an RFDE, and a function segment $u_t$ for $t=t_0+\tau/2$ illustrating that $u_t$ is not necessarily continuously differentiable, even when $\phi$ is smooth.} \label{fig:rfdeic}
\end{figure}

Here $u(t)\in\R^d$ and $\udot(t)\in\R^d$ while $u_t$ is an element of $C$. But what is the appropriate function space $C$? The best that we could hope for is that the solution $u(t)$ is smooth for $t\geq t_0$, as is usually the case for ODEs. However, even if this were so, there is no reason to suppose that $\udot(t_0)=\dot\phi(0)$, since the initial function can be chosen arbitrarily. Consequently for $t\in(0,\tau)$ the function segments $u_t(\theta)$ will in general have discontinuous derivative when $\theta=-t$, as seen in Figure~\ref{fig:rfdeic}.  Consequently, $C$ is usually taken to be the space of continuous functions $C=C\bigl([-\tau,0],\mathbb{R}^d\bigr)$.

The RFDE notation \eqref{eq:rfde} is useful when the dynamics depends 
on a distribution of previous state values, but in many models
only one or a few previous state values contribute to the dynamics leading to discrete delay DDEs such as
\begin{equation} \label{eq:ddde}
\udot(t)=f(t,u(t),u(t-\tau)).
\end{equation}

So far, we have only considered constant delays, but in many applications the delay can depend on the state of the system,
$\tau(u(t))$, or in an RFDE $\tau(u_t)$. The analysis of RFDEs with
state-dependent delays becomes very complicated because
$C$ being a space of only continuous functions then results in $F$ not being a Lipschitz function, and standard results for smooth functions cannot be applied directly. See the review \cite{HKWW06} for a description of the theoretical issues and approaches for their resolution.

In the current work, we will consider why and how delays arise, 
why constant discrete delays may not always be appropriate, and what are the consequences of modelling with 
distributed or state-dependent delays. 
We will consider how to treat these problems numerically and
demonstrate heuristic methods for their analysis
without recourse to deep functional analysis.

\CCLsubsection{DDE Model Examples}
\label{sec:delphysio}

Delays are ubiquitous in the physical sciences and engineering and may arise, for example, from transport, communication, and processing time. In the life sciences delays often arise from a blend of all three of these. For example,  
a hormone or antigen must first be produced and then transported to a receptor before a signal is received and processed. Maturation and incubation delays are also often significant in biological processes, where delays may occur at many scales from subcellular to population level processes.

Delays are also often incorporated into biological models
to avoid modelling the details of intricate processes, instead 
representing the process by a function that includes a delay.
This way, introducing delays allows us to keep models 
simpler than they might otherwise be. Of course, it is a modelling decision on when 
to incorporate a delay rather than model the entire process
leading to that delay, and when to ignore a delay because it is too small to be significant.

\CCLsubsubsection{Subcellular Delays: The Goodwin Operon Model} 
Within cells, the molecular machinery to
transcribe DNA to produce mRNA and then translate 
mRNA to produce an effector protein
is known as an \emph{operon}. 
\cite{goodwin1963, goodwin1965} first modeled operon dynamics mathematically using a system of ODEs, now known as the Goodwin model.
This model tracks the concentrations of mRNA ($M$), intermediate protein ($I$), and final effector protein ($E$) over time through the major stages of gene expression: DNA transcription, mRNA translation, and post-translational modification. 

Goodwin acknowledged the presence of delays in the processes, but neglected delays from his model. Later, \citet{banks77} extended the model by incorporating constant time delays. Subsequent studies of Goodwin-like models with constant delays have shown that their dynamics are usually similar to those of the models without delays, see \cite{mackey2016simple}.
Similar to the statement there,
\begin{align} \label{eq:GwM}
	\dfrac{dM\!}{dt}(t) & = \beta_M e^{-\mu \tau_{M}}f(E(t-\tau_M)) -\gamma_M M(t), \\ \label{eq:GwI}
	\dfrac{dI}{dt}(t) & = \beta_I e^{-\mu\tau_I} M(t-\tau_I) -\gamma_I I(t), \\ \label{eq:GwE}
	\dfrac{dE}{dt}(t) & = \beta_E I(t) -\gamma_E E(t),
\end{align}
is one possible form of such a model. With $\mu=0$, this models a cell of constant size, while for $\mu>0$ it accounts for dilution due to cellular growth.

\CCLsubsubsection{Burns and Tannock \texorpdfstring {$\mathbf{G_0}$}{G0} Cell Cycle Model}

A small pool of hematopoietic stem cells (HSCs) residing in the bone marrow, undergo cell division and differentiation to give rise to all of the blood cells in the body, with the 
human hematopoietic system producing about $10^{11}$ blood cells of various types per day~\cite{Kaushansky_2016}. That's about the same as the number of stars in our galaxy!
Each cell division requires many active proteins produced by operons, but it would be completely infeasible to model a population of dividing cells from the level of operons. A tractable model of 
HSC dynamics is described by the classic $G_{0}$ cell-cycle model of \cite{Burns1970}, illustrated in Figure~\ref{fig_SchematicHSC}.

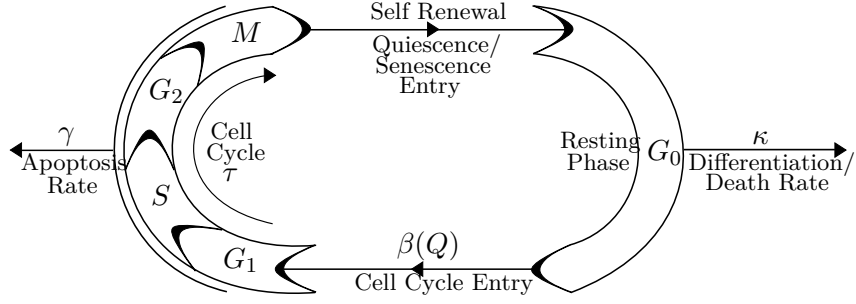
\begin{figure}[t!]
\centering
\begin{overpic}[scale=0.25]{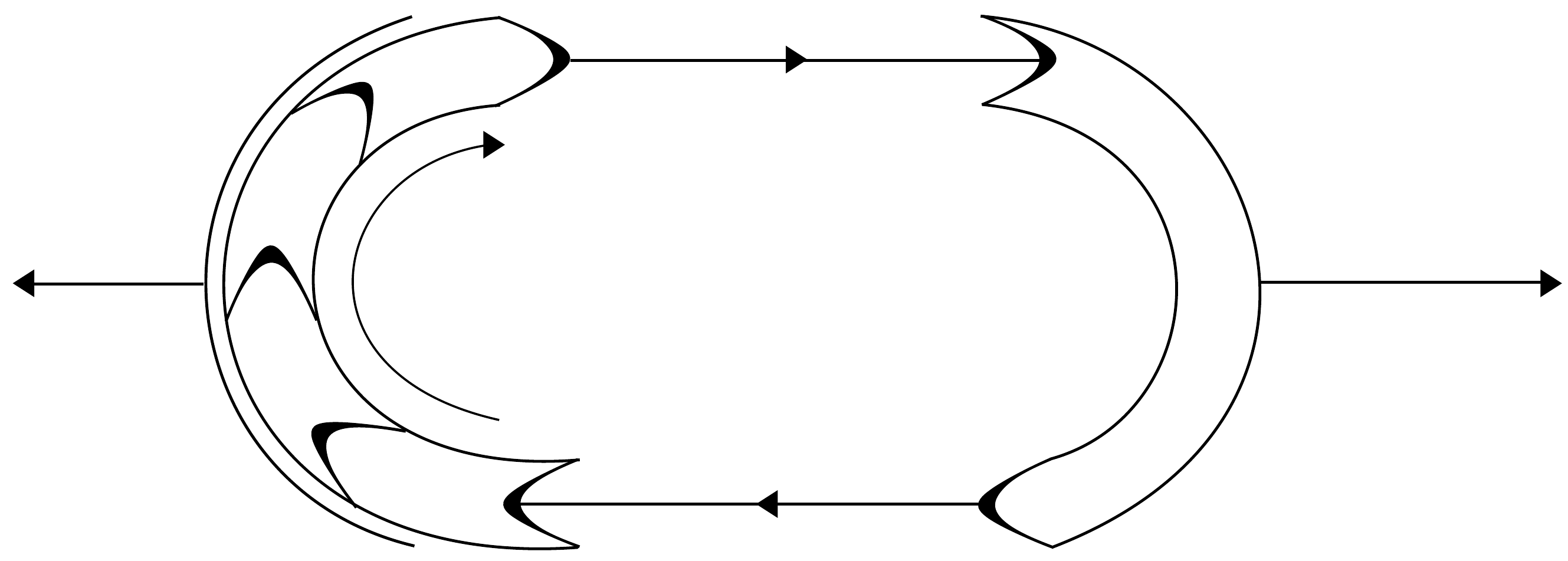}
\put(6.5,18.9){$\gamma$}
\put(2.1,15.4){\footnotesize{Apoptosis}}
\put(5.2,12.6){\footnotesize{Rate}}
\put(17.2,23.8){$G_{2}$}
\put(26.9,30.5){$M$}
\put(24.5,19){\footnotesize{Cell}}
\put(24,16.5){\footnotesize{Cycle}}
\put(25.9,14.0){$\tau$}
\put(17.5,11.2){$S$}
\put(26.0,3.8){$G_1$}
\put(43.4,33.2){\footnotesize{Self Renewal}}
\put(43.6,29.3){\footnotesize{Quiescence/}}
\put(44.0,26.8){\footnotesize{Senescence}}
\put(46.8,24.5){\footnotesize{Entry}}
\put(46.5,5.5){$\beta(Q)$}
\put(41.5,1.1){\footnotesize{Cell Cycle Entry}}
\put(65.8,17.9){\footnotesize{Resting}}
\put(66.6,15.4){\footnotesize{Phase}}
\put(76.1,16.5){$G_{0}$}
\put(88.5,18.4){$\kappa$}
\put(81.0,15.4){\footnotesize{Differentiation/}}
\put(83.0,13){\footnotesize{Death Rate}}
\end{overpic}
\vspace*{-1mm}
\caption{Schematic representation of the classical $G_{0}$ model for HSCs. The proliferating phase of the cell cycle is divided between 4 subphases:  gap one $G_{1}$, synthesis $S$, gap two $G_{2}$, and mitosis $M$. Cells in the resting phase (gap zero $G_{0}$), may differentiate with rate $\kappa$ or entry the cell cycle with rate $\beta(Q)$. Cells in the proliferation phase may be lost by apoptosis with rate $\gamma$, otherwise they re-enter the resting phase after mitosis, $\tau$ time units after they left the resting phase. (\textcopyright~Society for Industrial and Applied Mathematics (SIAM); Reproduced from \cite{SIADS19} with permission.)}
\label{fig_SchematicHSC}
\end{figure}

Of the five phases of the cell cycle, $S$ corresponds to DNA synthesis (which requires proteins produced by the operons), and $M$ is the actual mitosis step, when the cell divides in two. 
\cite{Mackey1978} formulated the $G_{0}$ cell-cycle model as
a DDE
\begin{gather} \label{eq:Qprime}
Q^{\prime}(t) = -(\kappa+\beta(Q(t)))Q(t)+A\beta(Q(t-\tau))Q(t-\tau),\\
\beta(Q) = f\frac{\theta^{s}}{\theta^{s}+Q^{s}},\qquad A = 2\mathnormal{e}^{-\gamma\tau} \notag
\end{gather}
that models the number of cells $Q$
in $G_0$ (the quiescent rest phase).
See \cite{Mackey_1994} for a full derivation of this DDE.

From the resting phase HSCs may enter the proliferating phase at a
rate $\beta(Q)$, or differentiate at a constant rate $\kappa$, or remain in the resting phase.
Once HSCs enter the proliferating phase they are lost by apoptosis with a constant rate $\gamma$ or undergo mitosis. 
After mitosis cells return to the $G_{0}$ resting phase, from whence the cycle may begin again. The model \eqref{eq:Qprime} represents
the 4 subphases of the cell cycle, $G_1$, $S$, $G_2$, $M$, very simply
by the time $\tau$ required for them to complete,
and by the amplification factor $A$ (with $A=2$ in pure cell division, and $1\ll A <2$ when $\mu>0$ and there is some cell death during the cell cycle). The beauty of this model is in its simplicity in replacing all of the complicated 
subcellular processes by a function that represents the resulting number of cells produced, and the time $\tau$ required for the cell cycle to complete.
This model can result in very complicated dynamics as explored in \cite{SIADS19}.

\CCLsubsubsection{Hematopoiesis}

\cite{Bernard_2003} used \eqref{eq:Qprime}
to describe HSC dynamics as one component of a larger model describing the regulation of circulating
white blood cell concentrations.
Subsequently many mathematical models have appeared which contain~\eqref{eq:Qprime}, or a variant, within
larger models for the production and regulation of circulating blood cells of 
various types; \cite{Adimy_2006a,Colijn_2005a,Colijn_2005b,Langlois2017}.
Multiple versions of~\eqref{eq:Qprime} have
also been coupled together to model discrete levels of cell maturity~\cite{Adimy_2006c,Qu_2010}.

\cite{craig2016} present a 
population model for white blood cells
that uses a variant of \eqref{eq:Qprime} to model 
the HSC population $Q$. This is coupled with two additional equations for mature neutrophils in the bone marrow reservoir, $N_R$, and circulating neutrophils $N$ to give
\begin{align} \label{eq:Q}
\frac{dQ}{dt} & = -\left(\kappa_N(G(t))+\kappa_\delta + \beta(Q(t))\right)
 + A_Q(t)\beta(Q(t-\tau_Q))Q(t-\tau_Q),\\ \notag
\frac{dN_R}{dt} & =   A_N(t) \kappa_N(G(t-\tauN)) Q(t-\tau_N)\frac{V(G(t))}{V(G(t-\tauNM))}\\ \label{eq:NR}
& \qquad -(\gamma_{N_R}+\ftrans(G(t)))N_R(t),\\ \label{eq:N}
\frac{dN}{dt} & =  \ftrans(G(t))N_R(t)- \gammaN N(t).
\end{align}

This model incorporates several interesting features
not found in \eqref{eq:Qprime}. 
First, the model incorporates a negative feedback loop driven by the principal cytokine granulocyte colony-stimulating factor (G-CSF) that regulates granulopoiesis. Essentially when GCSF concentrations are high, this stimulates the production of neutrophils through the $G$-dependent terms in the equations. On the other hand when neutrophil concentrations are high the production of G-CSF is inhibited; see
\cite{craig2016} for a statement of the differential equations governing G-CSF.

Second, there are many intermediate steps between HSCs and
mature neutrophils, including multipotent progenitor, common myeloid progenitor, granulocyte-macrophage progenitor, myeloblast, promyelocyte, myelocyte, metamyelocyte, band and finally
mature segmented neutrophil. Modelling each of these stages with separate copies of the Burns--Tannock cell cycle model
\eqref{eq:Qprime} would result in a very unwieldy system of ten or more equations. Instead, the model is simplified by omitting the intermediate stages
and introducing an amplification factor $A_N(t)$ defined by
\begin{equation} \label{eq:AN}
A_N(t)
= \exp \left[\int_{t-\tauNM(t)-\tauNP}^{t-\tauNM(t)} \etaNP(G(s)) d s-\gammaNM\tauNM(t)\right].
\end{equation}
This splits the time that it takes for an HSC to give rise to a mature neutrophil into separate proliferation and maturation stages. In the proliferation phase assumed to last a fix time $\tauNP$ cells divide at a rate $\etaNP$ that depends on the concentration of G-CSF, $G$. This is followed by a maturation time of variable length $\tauNM(t)$ during which cells die at a (small) fixed rate
$\tauNM$. Thus mature neutrophils are created at time $t$
from cells that differentiated from HSCs at time
$t-\tauN(t)$ where $\tauN(t)=\tauNP+\tauNM(t)$, the sum of the proliferation and maturation times. 

The maturation time $\tauNM(t)$ is governed by 
a G-CSF-dependent aging rate $V(G(t))$,
\be \label{eq:tauNM}
\int_{t-\tauNM(t)}^{t}V(G(s))ds   =  a.
\ee
Here $a>0$ is constant, representing the maturation
age. During maturation, cells age at nonconstant velocity $V(G(t))$ 
until they reach age $a$. Thus 
\eqref{eq:tauNM} implicitly defines the maturation delay
by $t-\tauNM(t)$ being the time 
that the cell started to mature so that maturation is complete at time $t$. 
In Section~\ref{sec:numdist} we will discuss how to evaluate
\eqref{eq:AN} and \eqref{eq:tauNM} in
numerical computations of the model \eqref{eq:Q}-\eqref{eq:N}.

Delays defined implicitly 
by equations like \eqref{eq:tauNM} are known as \emph{threshold delays}. They are defined by a process (in this case maturation) occurring at a variable speed
that needs to reach a certain threshold (in this case the maturation age $a$).

\cite{gedeon2022operon} extended the Goodwin model \eqref{eq:GwM}-\eqref{eq:GwE} by incorporating state-dependent transcription and translation to obtain
\begin{align} \label{eq:OpM}
	\dfrac{dM\!}{dt}(t) & = \beta_M e^{-\mu \tau_{M}(t)} \dfrac{v_M(E(t))}{v_M(E(t-\tau_{M}(t)))} f(E(t-\tau_M(t))) -\bgamma_M M(t), \\ \label{eq:OpI}
	\dfrac{dI}{dt}(t) & = \beta_I e^{-\mu\tau_I(t)} \dfrac{v_I({M(t)})}{v_I(M(t-\tau_I(t)))} M(t-\tau_I(t)) -\bgamma_I I(t), \\ \label{eq:OpE}
	\dfrac{dE}{dt}(t) & = \beta_E I(t) -\bgamma_E E(t),
\end{align}
supplemented with the two threshold state-dependent delays 
\begin{equation} \label{eq:tauMtauI}
	a_M = \int_{t-\tau_M(t)}^{t} v_M(E(s)) ds \quad \text{and} \quad a_I = \int_{t-\tau_I(t)}^{t} v_I(M(s)) ds.
\end{equation}

The velocity ratio $V(G(t))/V(G(t-\tauNM))$ that 
appears in \eqref{eq:NR}, and the similar 
terms in \eqref{eq:OpM} and \eqref{eq:OpI}
are a feature of properly constructed threshold delay models. 
In Section~\ref{sec:velratio} we will explain how this term arises.

Because the delay defined by \eqref{eq:tauNM} 
depends on the state of the system 
through the quantity $G(t)$ (or more precisely through the function $G_t$ in RFDE notation)  threshold delays are examples of \emph{state-dependent delays}. 

The general DDE \eqref{eq:ddde} 
will be a state-dependent delay DDE, if $\tau$
instead of being a constant delay, is allowed to depend on the state. The simplest way to do this is to make $\tau$ a function of the state-variable at time $t$, so $\tau=\tau(u(t))>0$. 
In Section~\ref{subsec:statedep},
we will explore the dynamics of the DDE \eqref{eq:twostatedep} with two state-dependent delays of this form. 

The DDE \eqref{eq:ddde} 
with $\tau$
defined implicitly by the threshold condition
\be \label{eq:thres}
\int_{t-\tau}^{t} V(u(s)) ds=a,
\ee
for a given non-negative function $V$ and positive constant $a$,
will also be a state-dependent delay DDE.
We remark that the DDE \eqref{eq:ddde} with delay $\tau$ defined by the threshold condition \eqref{eq:thres} is no longer a discrete delay DDE,
as the evolution of the solution depends on the delay $\tau$ which itself
depends integral of $V(u(s))$ for $s\in[t-\tau,t]$. Thus \eqref{eq:ddde},\eqref{eq:thres} is properly regarded as an RFDE, but one where $\udot(t)$ only depends on the values of $u(t)$ and $u(t-\tau)$, but where $\tau$ itself is defined as a function of $u_t$. This kind of halfway house between discrete delays and RFDEs will turn out to make threshold problems more tractable than general RFDEs.

It is also possible to nest delays, so that the delay $\tau$ depends on the state-variable at a delayed time; $\tau=\tau(u(t-\tau_0(u(t))))$, for given functions $\tau$ and $\tau_0$ (see \cite{HKRS22}).

\begin{figure}[t!]
\centering
\begin{overpic}[scale=0.3]{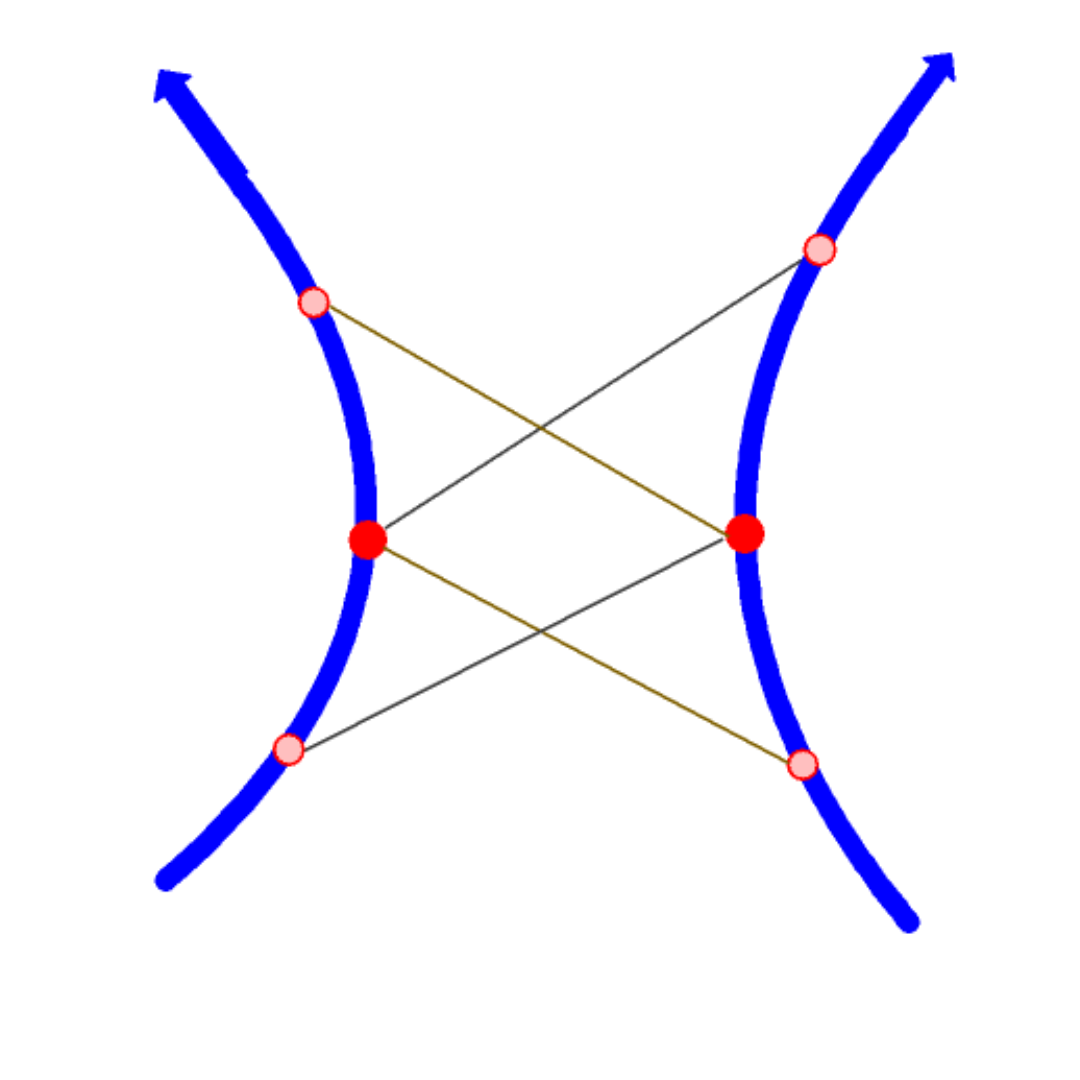}
\put(73,49) {$p(t)$}
\put(78,74) {$p(t\!+\!\tau_p^+)$}
\put(78,28) {$p(t\!-\!\tau_p^-)$} 
\put(1,71) {$e(t\!+\!\tau_e^+)$} 
\put(-1,29) {$e(t\!-\!\tau_e^-)$} 
\put(20,49) {$e(t)$} 
\put(-25,5){\vector(0,1){20}}
\put(-25,5){\vector(1,0){20}}
\put(-30,27){\textrm{time}}
\put(-3,3){space}
\end{overpic}
\vspace*{-1mm}
\caption{Schematic of \cite{WF45,WF49} electrodynamics. A proton $p(t)$ and electron $e(t)$ only feel forces (propagating at the speed of light $c$) from the other charge in the past and future at the points where the world lines and light cones intersect. Thus
delays and advances are implicitly defined by
$|p(t)-e(t\pm\tau_e^\pm)|=c\tau_e^\pm$ and  
$|e(t)-p(t\pm\tau_p^\pm)|=c\tau_p^\pm.$}
\label{fig_WF}
\end{figure}
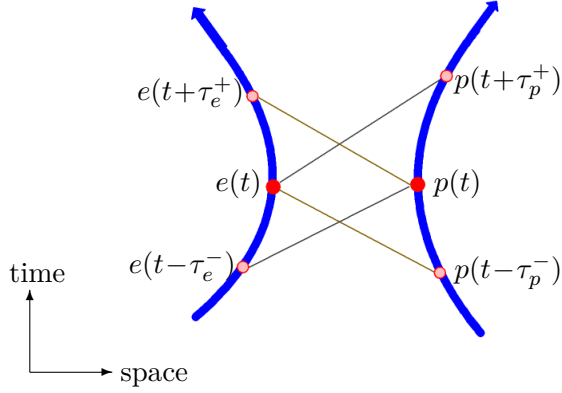

Even more exotic state-dependent delays arise. The electrodynamics of \cite{WF45,WF49}, as illustrated in Figure~\ref{fig_WF},
replaces field theory with forces acting at a distance and propagating at the speed of light $c$.
A proton $p(t)$ will feel a force from an electron  in the past $e(t-\tau_e^-)$, where the delay $\tau_e^-$ is defined implicitly by the distance between $p(t)$ and $e(t-\tau_e^-)$ being
equal to the distance propagated by the force in that time:
$$|e(t-\tau_e^-)-p(t)|=c\tau_e^-.$$
The electron $e(t)$ similarly feels a force from the proton in the past
$p(t-\tau_p^-)$. Only including retarded forces breaks the time symmetry of physics, so Wheeler and Feynman worked with advanced forces as well.
Essentially each particle only feels forces from other particles at the points where the other particles' world-line intersects the forward and backward light cone of the original particle. Relativistic forces are needed for momentum to be conserved, and the whole set up even for just two particles results in a nasty system of neutral implicitly state-dependent advanced-retarded DDEs. Despite being nearly intractable, the problem continues to elicit interest, and a number of special cases have been tackled; see \cite{Schild63,Feynman65,DLGHP10,DLHR12,DeLuca16}.

\CCLsubsection{Why model with state-dependent delays?}
\label{subsec:statedep}

\cite{HDMM12,CHK17} studied the DDE
\be \label{eq:twostatedep} 
\udot(t)
=-\gamma u(t) - \sum_{j=1}^2\kappa_j u(t-a_j-c_j u(t)), \qquad u(t)\in\R.
\ee
It has two state-dependent delays
$\tau_j(u(t))=a_j+c_ju(t)$ 
which become constant delays when $c_1=c_2=0$, and are 
linearly state-dependent otherwise.

Equation~\eqref{eq:twostatedep} does not model any particular process or application, but is a wonderful vehicle
for studying the effects of state-dependent delays, because the 
model \eqref{eq:twostatedep} is completely linear
apart from the state-dependent delays. Consequently, any interesting nonlinear dynamics that are observed must be driven by the state-dependency of the delays.

\begin{figure}[t!]
\centering
\includegraphics[scale=0.55]{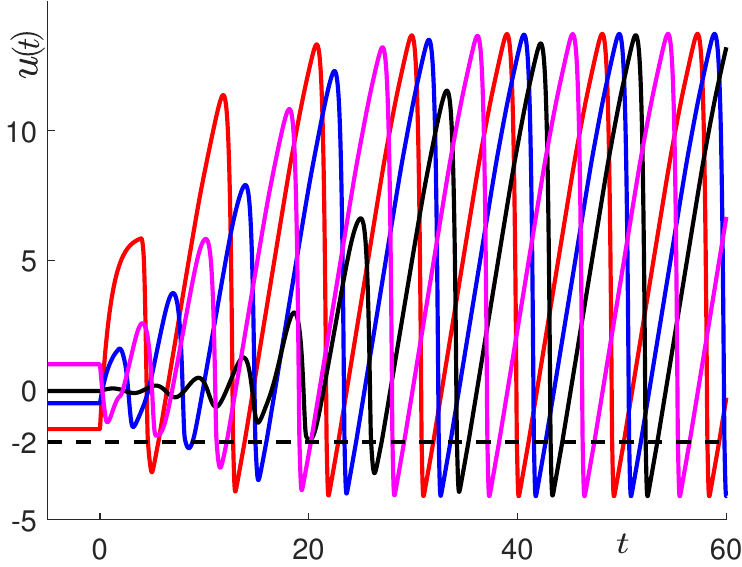}
\caption{Solutions of the IVP for
\eqref{eq:twostatedepmod} with constant initial functions $\phi(t)=u_0$ for several different values of $u_0$
with other parameters $1=a_1=\gamma<\kappa_2=\kappa_1=a_2=2$, $c_1=0.5$, $c_2=0.4$.
The dashed line at $u=-2$ indicates where $u(t)=-a_1/c_1$ and $\tau_1=0$. (\textcopyright~American Institute of Mathematical Sciences (AIMS); Reproduced from \cite{HDMM12} with permission.)}
\label{fig:advex}
\end{figure}
 
When $c_1=c_2=0$, equation \eqref{eq:twostatedep} defines a boring 
linear constant delay DDE, so lets consider the case of strictly positive
parameters $\gamma$, $\kappa_j$, $a_j$ and $c_j$.
Note first that $t-\tau_j=t-a_j-c_ju(t)<t$ if and only if $u(t)>-a_j/c_j$, so the delays will become advances if $u(t)<-a_j/c_j$. Consider, for a moment, the modified problem
\be \label{eq:twostatedepmod}
\udot(t) = -\gamma u(t) - \kappa_1 u(t_1) - \kappa_2 u(t_2), \quad t_j=\min\{t,t-a_j-c_ju(t)\}
\ee
where we have adjusted the delay terms to prevent them from becoming advances.
Solutions of this DDE for different initial functions are presented in Figure~\ref{fig:advex}, and all appear to converge to the same stable limit cycle, but with different phases. However, the solution
enters the region below the dashed line where 
$t<t-a_1-c_1u(t)$, and so are not solutions of the original problem
\eqref{eq:twostatedep}. We remark that the \cite{matlab} state-dependent DDE solver \texttt{ddesd} automatically solves \eqref{eq:twostatedep} by converting it to \eqref{eq:twostatedepmod}, so when using \texttt{ddesd} it is necessary to ensure that the delays really are delays along the whole solution, to avoid obtaining solutions like those in Figure~\ref{fig:advex} for \eqref{eq:twostatedepmod} when you really meant to solve \eqref{eq:twostatedep}.

To ensure that solutions to
\eqref{eq:twostatedep} do not terminate because 
a delay becomes advanced, it is sufficient (for this DDE) to place 
a constraint on the parameter values. Without loss of generality, order the delays so that $0>-a_1/c_1\geq -a_2/c_2$. Then if $u(t)>-a_1/c_1$
it follows that $t-a_j-c_ju(t)<t$ for both $j=1$ and $j=2$. Thus the problem of ensuring that the delays are not advanced is reduced 
to ensuring that $u(t)$ is bounded below by
$-a_1/c_1$. To do this note that from
\eqref{eq:twostatedep} if $u(t)=-a_1/c_1$
then
$$\udot(t)  =\gamma\frac{a_1}{c_1} 
    -\kappa_1 u(t)
   -\kappa_2 u(t-\tau_2) 
=(\gamma+\kappa_1)\frac{a_1}{c_1} 
   -\kappa_2 u(t-\tau_2) >0  
$$
if $u(t-\tau_2)<
\frac{a_1}{\gamma c_1}(\kappa_1+\kappa_2)$ and $\gamma>\kappa_2$. This leads to the following result:

\begin{theorem}[Existence, Boundedness and Uniqueness \cite{HDMM12}] \label{thm:twodelbd}
If $\gamma>\kappa_2$ and $\phi(t)\in(-\frac{a_1}{c_1},\frac{a_1}{\gamma c_1}(\kappa_1+\kappa_2))$ for all $t\in[-\tau_0,0]$ where
$\tau_0=\max_j\{a_j+(\kappa_1+\kappa_2)c_ja_1/(\gamma c_1)\}$ then the IVP \eqref{eq:twostatedep} with $u(t)=\phi(t)$ for $t\in[-\tau_0,0]$ has a unique solution which satisfies
\[u(t)\in\Bigl(-\frac{a_1}{c_1},\frac{a_1}{\gamma c_1}(\kappa_1+\kappa_2)\Bigr)
\qquad\forall t\geq0.\eqno{\qed}\]
\end{theorem}

\begin{figure}[t!]
\centering
\hspace*{-1em}\includegraphics[scale=0.96]{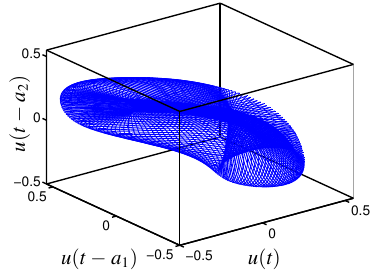}\includegraphics[scale=0.71]{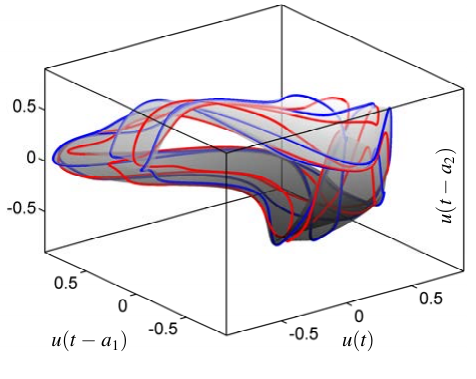}
\caption{Solutions of \eqref{eq:twostatedep} 
projected onto $(u(t),u(t-a_1),u(t-a_2))$-space. 
Parameter values are
$\gamma=4.75>\kappa_2=3.0$, $a_1=1.3$, $a_2=6.0$ and $c_1 = c_2 =1.0$.
(a) With $\kappa_1=4.44$ a quasi-periodic torus is observed. 
(b) For $\kappa_1=6.93$, a 
1:4 phase-locked torus-like object (grey),
with embedded stable (blue) and saddle (red)
periodic orbits. (\textcopyright SIAM; Reproduced from \cite{CHK17} with permission.)}
\label{fig:tori}
\end{figure}

Figure~\ref{fig:tori} shows that \eqref{eq:twostatedep} generates stable and saddle-stability periodic orbits
as well as quasi-periodic and phase-locked invariant tori. This is quite remarkable, as 
with $c_1=c_2=0$ the DDE is linear, for which solutions typically decay (when all the characteristic values have negative real part) or blow up
(when a characteristic value has positive real part).
It is the nonlinearity in \eqref{eq:twostatedep} that results in all of the interesting dynamics, and that nonlinearity arises purely from the state-dependency of the delays.

Approximating state-dependent delays by constant delays would suppress dynamics generated by the state-dependent delays. This is particularly dramatic in the case of \eqref{eq:twostatedep} where there is no other source of nonlinearity. Interesting dynamics may remain when a system is still
nonlinear with constant delays, but without studying the state-dependent delay case it would be impossible to know if important dynamical features had been lost.
Only considering constant delays
may result in oversimplified models that fail to capture the actual dynamics of the underlying system. 
Thus, just as it is important to reflect on whether delays are needed or can be neglected when modelling dynamical processes,  it is also important
to consider when constant delays suffice,
and when state-dependent delays are required.

\CCLsubsubsection{Existence and Uniqueness with Discrete State-Dependent Delays}

For constant discrete delay DDEs existence
and uniqueness of solutions follows directly from the ODE theory. Equation~\eqref{eq:ddde} with initial function defined by \eqref{eq:rfdeic} can be solved for $t\in[t_0,t_0+\tau]$ 
as a nonautonomous ODE since
$u(t-\tau)=\phi(t-\tau)$ is a known function. Having solved to $t_0+n\tau$,
the method can be repeated to solve to $t_0+(n+1)\tau$, hence this technique is known as the \emph{method of steps}. 

\sloppy{Existence and uniqueness is more problematical for state-dependent DDE
problems.} We already saw above that for \eqref{eq:twostatedep} solutions may terminate because a delay becomes an advance. Nonuniqueness of solutions is also possible, as the following example of \cite{Winston70} shows.
Let
\begin{gather} \label{eq:Winstonu}
\udot(t) = -u(t-|u(t)|), \quad u(t)=\phi(t) \;\; t\leq0,\\  \label{eq:Winstonphi}
\phi(t) = \left\{\begin{array}{cl}
-1 & t\leq-1,\\
\tfrac32(t+1)^{1/3}-1 & t\in(-1,-7/8],\\
\tfrac{10}{7}t+1 & t\in(-7/8,0].
\end{array}\right.
\end{gather}
This problem is easily seen to have two solutions given by
\begin{gather*}
u(t)=\left\{\begin{array}{ll}
1+t, & t\geq0\\
1+t-t^{3/2}, & t\in[0,\tfrac14].
\end{array}\right.
\end{gather*}
Notice that for the first solution $t-|u(t)|=-1$, 
while for the second solution $t-|u(t)|=-1+t^{3/2}\in(-1,-7/8]$. 
Thus, the two solutions access different segments of the 
initial function $\phi(t)$, and this leads to the non-uniqueness.

\cite{Driver63} resolved the existence and uniqueness theory for discrete state-dependent DDEs. 
Consider the differential equation for $u(t)\in\R^d$:
\begin{align} \label{eq:driverdde}
   \dot u(t) & = f\bigl( t,u(t),u(\alpha_1(t,u(t))),\ldots, u(\alpha_N(t,u(t))) \bigr), \quad t\geqslant t_0,  \\ \label{eq:driverphi}
   u(t) & = \varphi(t), \hspace*{14em} t\in[t_0-\tau_0,t_0], 
\end{align}
where $\alpha_j(t,u(t))=t-\tau_j(t,u(t))$
is such that 
\be \label{eq:deldels}
t_0-\tau_0\leq\alpha_j(t,u(t))\leq t.
\ee

The following theorems paraphrase the results of \cite{Driver63}, see the original paper for full details.

\begin{theorem}[Extended Existence Theorem \cite{Driver63}] \label{thm:Driverexist}
Let $f$, $\alpha_j$ and $\phi$ be continuous,
and $\alpha_j$ satisfy \eqref{eq:deldels}, then
there exists a solution $u(t)$ 
of \eqref{eq:driverdde},\eqref{eq:driverphi}
for 
$t\in[t_0,T)$ with either
$$T=\infty, \qquad \textrm{or} \qquad \limsup_{t\to T}\Big(\max\{\|u(t)\|,\|u(\alpha_j(t))\|\}\Big)=\infty. \eqno{\qed}$$ 
\end{theorem}

\begin{theorem}[Uniqueness Theorem \cite{Driver63}] \label{thm:Driveruniq}
If $f$, $\alpha_j$ and $\phi$ are locally Lipschitz
and $\alpha_j$ satisfies \eqref{eq:deldels}, 
the solution $u(t)$ of \eqref{eq:driverdde}, \eqref{eq:driverphi}
is unique and has continuous dependence on 
$\phi$, $f$ and $\alpha_j$.\hfill\qed
\end{theorem}

Theorem~\ref{thm:twodelbd} is proved by applying Theorem~\ref{thm:Driverexist} and~\ref{thm:Driveruniq}, once it is established that the solutions are bounded and \eqref{eq:deldels} holds.
The example \eqref{eq:Winstonu}, \eqref{eq:Winstonphi} satisfies the conditions of Theorem~\ref{thm:Driverexist}, but Theorem~\ref{thm:Driveruniq} does not apply because 
$\phi$ is not Lipschitz. Indeed $\dot\phi(t)$ becomes unbounded as $t\to-1$ with $t\in(-1,-7/8)$.

Notice that the solution $u(t)=1+t$ of 
\eqref{eq:Winstonu}, \eqref{eq:Winstonphi} satisfies
$t-\tau=t-|u(t)|=-1$ for all $t\geq0$. In our experience,
the nasty examples of state-dependent delay DDEs 
behaving badly typically involve 
$\frac{\phantom{t}d}{dt}(t-\tau)\leq0$, as in this example. Ensuring that 
$\alpha_j(t,u(t))=t-\tau_j(t,u(t))$ is an increasing function of $t$ on solutions 
resolves many issues.

\CCLsection{Threshold Delays}
\label{sec:thresdels}

Threshold delays arise in systems where a transport process occurs over a fixed distance at a variable speed. Such delays are also common in biological systems, for example, when maturation rates depend on nutrient availability. 

We have already seen examples of threshold delays in \eqref{eq:tauNM} and \eqref{eq:tauMtauI}, and introduced a threshold delay \eqref{eq:thres}
for the general DDE \eqref{eq:ddde}. In this section we consider 
a state-dependent delay $\tau$ implicitly defined by the threshold condition
\begin{equation} \label{eq:thre_scalarDE}
	a 
    = \int_{t-\tau}^{t} V(u(s)) ds
    =\int^t_{t-\tau (u_t)}\hspace*{-1.5em}V(u(s))ds
    =\int^0_{-\tau (u_t)}\hspace*{-1.5em}V(u_t(\theta))d\theta
\end{equation}
where $a>0$ is a fixed constant (the `distance') and $V$ is a given function (the 'speed') of the state variable $u(t)$.

So far we have used the first of the three forms of \eqref{eq:thre_scalarDE}. Since the delay $\tau$ depends on the solution $u$ for all values of $t\in[t-\tau,t]$, it is properly a function of $u_t$ (defined in \eqref{eq:ut}), and then depending on how much RFDE notation you want to use,
the second and third forms in \eqref{eq:thre_scalarDE} can be used.

\CCLsubsection{A scalar DDE with threshold delay}
\label{subsec:scalthresdde}

Consider the
scalar threshold DDE model 
\begin{equation}\label{eq:scalarDE}
	\udot(t)  = \beta e^{-\mu\tau(u_t)}{\frac{V(u(t )))}{V(u(t-\tau(u_t)))}}g(u(t -\tau (u_t ))) - \gamma u (t)
\end{equation}
with the delay given by \eqref{eq:thre_scalarDE}, and $u_t(\theta)$
defined by \eqref{eq:ut}
for $\theta\in[-\tau_{\textit{max}},0]$.

\cite{gedeon2024dynamics} study this model as a reduction of the operon model \eqref{eq:OpM}-\eqref{eq:tauMtauI},
and the special case with 
$g(u)=u$ is considered in 
\cite{smith1993reduction}. This equation models a population with a constant decay/death rate $\gamma$, a production/birth rate $\beta g(u)$,
and maturation/processing time $\tau$. The parameter $\mu$ can either represent a death rate during maturation or, as in the case 
of the operon model, effective dilution due to cell growth.

In the special case where $V(u(t))=V$ is constant, the delay also becomes constant and is given by $\tau=a/V$. 

To ensure well-posedness, it is assumed that the non-constant velocity $V$ is bounded and bounded away from zero, that is
\begin{equation} \label{ineq:vel_bdd}
     0<V_{\textit{min}}=\liminf_{u\geq0}V(u)
     \leq\limsup_{u\geq0}V(u)=V_{\textit{max}}
     <\infty.
\end{equation}
Then since $0<V_{\textit{min}}\leq V(u(t))\leq
V_{\textit{max}}$ it follows that the process described by \eqref{eq:thre_scalarDE}
always proceeds to completion without stalling. That is, 
the threshold delay $\tau=\tau(u_t)$ is also bounded and bounded away from zero with
\begin{equation} \label{ineq:thre_bdd}
	0<\tau_{\textit{min}}=\frac{a}{V_{\textit{max}}}\leq \tau(u_t)\leq 
	\tau_{\textit{max}}=\frac{a}{V_{\textit{min}}}<\infty.
\end{equation}

Differentiating the threshold condition \eqref{eq:thre_scalarDE} with respect to $t$ using the Leibniz rule, we get
\begin{equation} \label{eq:Leibdiff}
	0=V(u(t))-V(u(t-\tau(u_t)))\frac{\phantom{t}d}{dt}(t-\tau(u_t)). 
\end{equation} 
Then $t-\tau(u_t)$ is monotonically increasing in $t$ along solutions since
\begin{equation} \label{ineq:tminustau_up}
	0 < \frac{V_{\textit{min}}}{V_{\textit{max}}}\leq\frac{V(u(t))}{V(u(t-\tau(u_t)))}
	= \frac{\phantom{t}d}{dt}(t-\tau(u_t))  = 1-\frac{\phantom{t}d}{dt}\tau(u_t).
\end{equation}

We can also use \eqref{ineq:tminustau_up} to replace the threshold condition \eqref{eq:thre_scalarDE} to obtain a system of discrete delay DDEs composed of
\eqref{eq:scalarDE} and 
\begin{equation} \label{eq:taudot} 
 \frac{\phantom{t}d}{dt}\tau
 = 1- \frac{V(u(t))}{V(u(t-\tau))}.
\end{equation}
Given an initial function $\phi$ such that $u_{t_0}=\phi$ for \eqref{eq:scalarDE} then
to solve the same IVP with \eqref{eq:taudot}
instead of the threshold condition \eqref{eq:thre_scalarDE} 
it is necessary to impose the initial condition for $\tau(t_0)$ that 
\begin{equation} \label{eq:thre_scalarDEIC}
\int^{t_0}_{t_0-\tau(t_0)}\hspace*{-1.5em}
 V(u(s)) ds=a.
\end{equation}
Formulated in this way as a discrete delay DDE,
Theorems~\ref{thm:Driverexist} and 
Theorems~\ref{thm:Driveruniq} can be applied to obtain existence and uniqueness of solutions. 

It will be important later, when considering stability and numerical solutions, to note that the two formulations \eqref{eq:scalarDE},\eqref{eq:thre_scalarDE} and  \eqref{eq:scalarDE},\eqref{eq:taudot}
are not actually equivalent. Changing the parameter $a$ in \eqref{eq:scalarDE},\eqref{eq:thre_scalarDE} only alters the initial condition \eqref{eq:thre_scalarDEIC} and not the DDE system
in \eqref{eq:scalarDE},\eqref{eq:taudot}.
Thus, the first problem for all different values of the parameter $a>0$, is embedded in a single copy of the second problem. 
This, of course, arises because the value of the parameter $a$ is lost when the
threshold condition \eqref{eq:thre_scalarDE} is differentiated.

\CCLsubsection{Time Rescaling to Constant Delay} 
\label{subsec:rescaletime}

\cite{smith1993reduction,smith91}
demonstrated that a threshold delay can be transformed into a constant delay through an appropriate time rescaling. This transformation allows standard results from the theory of RFDEs
to be applied.

To illustrate this method using the scalar model \eqref{eq:scalarDE},\eqref{eq:thre_scalarDE}, let $\tilde{t}$ be a new trajectory-dependent time variable defined by
\begin{equation} \label{eq:tildet}
	\tilde{t}=\frac1a\int^t_{0}V(u(s))ds. 
\end{equation}
Then by \eqref{eq:thre_scalarDE} and \eqref{eq:tildet},
\begin{equation} \label{eq:tildet-1} \tilde{t}-1=\frac1a\int^t_{0}V(u(s))ds-\frac1a\int^t_{t-\tau(u_t)}V(u(s))ds=\frac1a\int^{t-\tau(u_t)}_{0}V(u(s))ds.
\end{equation}
Thus, under this change of variable, the original time t is mapped to $\tilde{t}$ via \eqref{eq:tildet} and the delayed time $t-\tau(u_t)$ is mapped to $\tilde{t}-1$ as shown in  \eqref{eq:tildet-1}. 

Consequently, defining $U(\tilde{t})=u(t)$, the system \eqref{eq:scalarDE},\eqref{eq:thre_scalarDE} transforms under the time rescaling into the constant delay DDE
\begin{equation} \label{eq:dUdtildet}
	\frac{dU}{d\tilde{t}}  = a\beta e^{-\mu\tau_0(U_{\tilde{t}})}\frac{g(U(\tilde{t}-1))}{V(U(\tilde{t}-1))}- a\gamma \frac{U(\tilde{t})}{V(U(\tilde{t}))}
\end{equation}
where the history function $U_{\tilde{t}}\in \mathbb{C}([-1,0], \mathbb{R}^+)$ is defined by $U_{\tilde{t}}(s)=U(\tilde{t}+s)$ and the delay term $\tau_0: \mathbb{C}([-1,0], \mathbb{R}^+) \to \mathbb{R}^+$ is defined by
\begin{equation*}
	\tau_0(\phi) = \int_{-1}^{0} V(\phi(s))^{-1} ds. 
\end{equation*}

If $\mu=0$ then the transformation converts the threshold delay problem into a discrete constant delay DDE, which can then be solved using standard methods. However, the time transformation is different for each solution trajectory and may alter the structure of the nonlinearity, potentially complicating the interpretation of the system’s dynamics. 

In cases when $\mu>0$, the transformed equation \eqref{eq:dUdtildet} remains a distributed DDE because of the term $\tau_0(U_{\tilde t})$,
even though the delay in the argument of the state variable becomes constant. 
Several works have used this transformation to study threshold models as distributed constant delay DDEs (see for example \cite{KCP16,TIY18}), but 
often the existence of transformation has been used as an excuse to ignore the threshold delay and treat the delay as if constant.
However, we will show below that a threshold delay system can have very different dynamics than the same DDE
does with a constant discrete delay. 

In systems with two or more delays, 
with at least one threshold delay, 
it is generally impossible to simultaneously transform all delays into constant ones using a common time rescaling. 
Given this limitation, we are motivated to seek alternative analytical and numerical approaches for studying threshold delay DDEs without a time transformation.

\CCLsubsection{The Velocity Ratio In Threshold Delay Models}
\label{sec:velratio}

We have yet to explain the presence of the ratio $V(u(t))/V(u(t-\tau(u_t)))$ in \eqref{eq:scalarDE}, and the similar terms appearing in other examples. We will show here that it is a direct consequence of the threshold delay condition \eqref{eq:thre_scalarDE}.

\begin{figure}[t]
    \centering
    \includegraphics[scale=0.5]{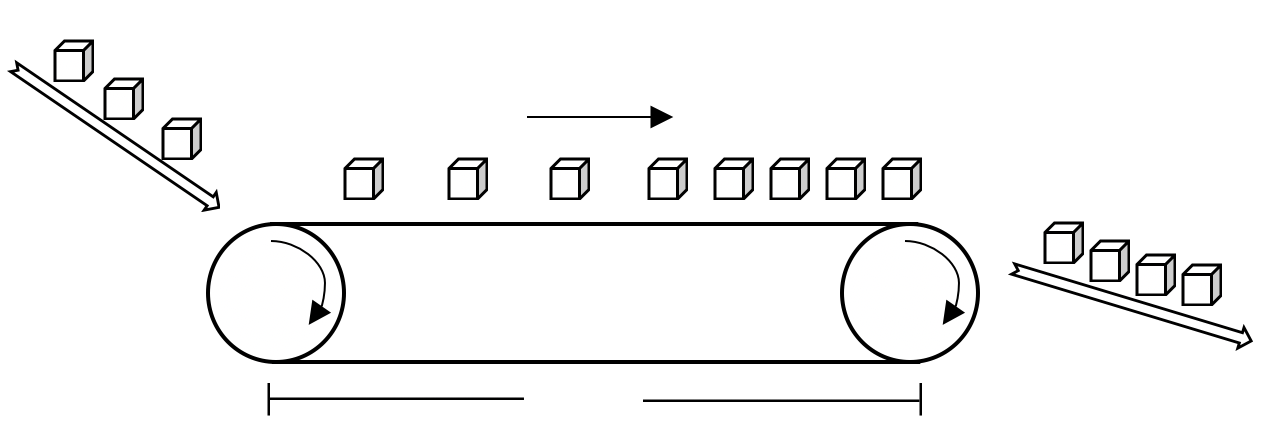}
    \put(-175,7){$a$}
    \put(-182,86){$V(t)$}
\caption{Illustration of a transport process with the conveyor belt analogy. Boxes are placed on the conveyor belt at a constant rate. The boxes towards the right on the conveyor belt are more closely spaced because the speed $V$ of the belt has been increased since they were placed on it.}
\label{fig:conveyor}
\end{figure}

The easiest way to see that this term must be included in the model is to use the conveyor belt analogy illustrated in Figure~\ref{fig:conveyor}.
Consider an airport conveyor belt that moves at a constant speed.
If bags are placed on the belt one every second,
then after the first bag reaches the other end, they will also drop off of the conveyor belt at baggage claim one every second. Notice that the velocity of the conveyor belt and the length of the conveyor belt only affect the delay before the first bag arrives. But once the first bag arrives, the rate that subsequent bags arrive is equal to the rate that they were placed on the conveyor belt.

Now, suppose the conveyor belt is very long, and has adjustable speed. The bags are still placed on the belt at a rate of one a second, but the velocity of the conveyor belt is doubled before the first bag arrives. Once the bags start to arrive, they will do so at a rate of two every second, until such time as all the bags that were placed on the belt when it was moving slowly have arrived.

Clearly, if bags that arrive at time $t$ were placed on the belt at some time $t-\tau$ at rate $R$, then they arrive at time $t$ at rate
\be \label{eq:beltR}
\frac{V(t)}{V(t-\tau)}R,
\ee
where $V(t)$ is the velocity of the belt. 
This is the simplest explanation of the origin of the velocity ratio terms that appear in the threshold delay models. 

We remark that from this argument, the velocity ratio is a consequence of the delay being non-constant, and should appear in any models with non-constant delay. For threshold delays similar to \eqref{eq:thre_scalarDE} the velocity ratio is needed whenever the velocity is non-constant, whether or not the velocity depends on the state of the system. 
 
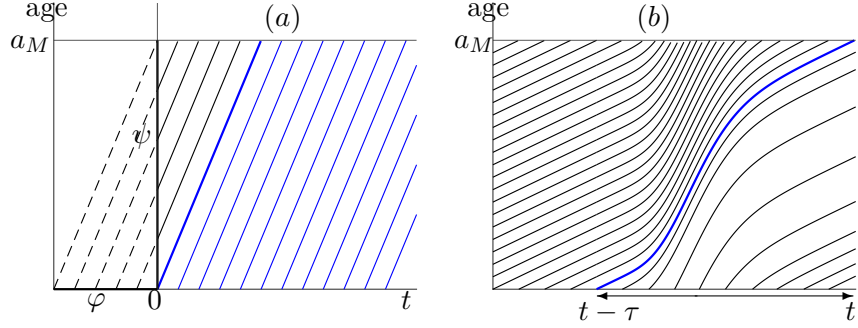
\begin{figure}[tp!] 
    \centering	
    \includegraphics[scale=0.42]{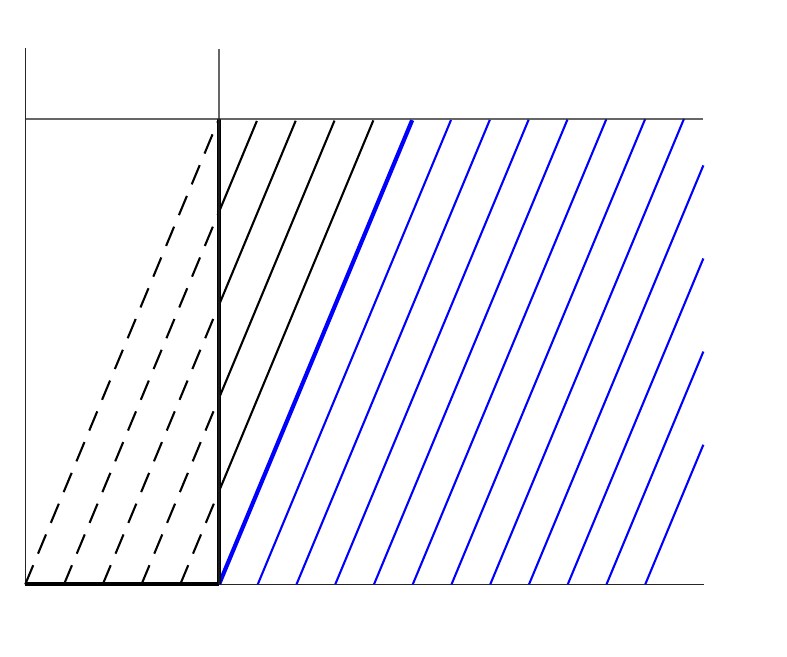}\hspace*{-1em}\includegraphics[scale=0.42]{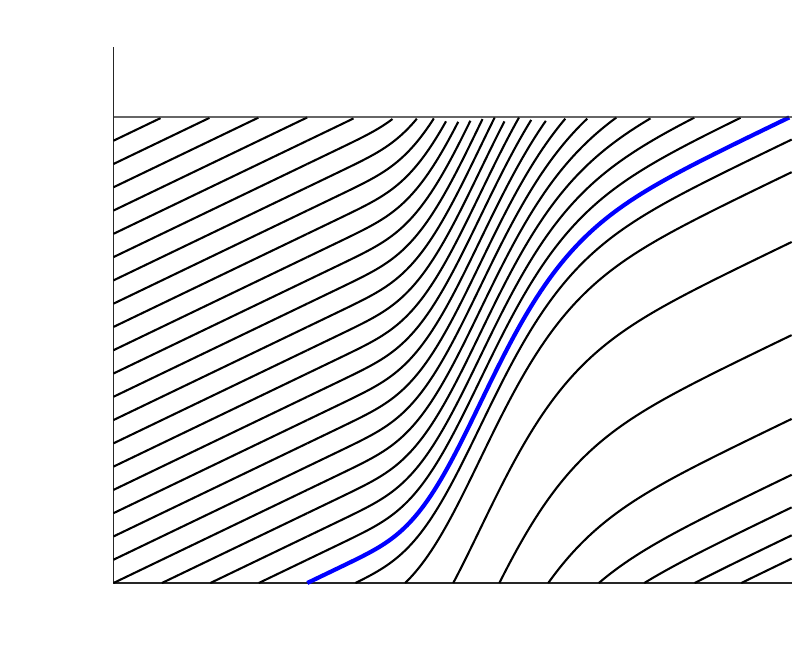}
    \put(-317,118){age}
    \put(-323,106){$a_M$}
    \put(-271,7){$0$}
    \put(-294,8){$\phi$}
    \put(-277,70){$\psi$}
    \put(-175,7){$t$}
    \put(-150,118){age}
    \put(-155,106){$a_M$}
    \put(-108,3){$t-\tau$}
    \put(-7,3){$t$}
    \put(-64,12){\vector(1,0){60}}
    \put(-63,12){\vector(-1,0){38}}
    \put(-227,114){$(a)$}
    \put(-86,114){$(b)$}
   \caption{(a) Illustrating the characteristics of \eqref{eq:aspde} with constant velocity $V$ and the relationship \eqref{eq:psiphi} between the initial function $\phi$ for \eqref{eq:scalarDE}  and the boundary condition $\psi$ for \eqref{eq:aspde}.
    (b) Illustration of maturation with non-constant ageing velocity.} \label{fig:agestr}
\end{figure}

Several authors have used a conveyor 
belt analogy before (see \cite{Bernard2016}), but 
we are not aware of it having been used
to derive \eqref{eq:beltR}.

It is also possible to derive \eqref{eq:scalarDE} as a population model
with the threshold condition \eqref{eq:thre_scalarDE}. 
To do so, consider a population of adults $u$ which reproduce at rate $\beta g(u)$ and die at rate $\gamma$.
We introduce a state-dependent delay $\tau$ for juveniles to mature, and let $A(t)$ be the rate that juveniles mature into adults. Then the adult population dynamics are governed by
\be \label{eq:ueq}
\udot(t)=A(t)-\gamma u(t).
\ee

We model the juvenile population using an age-structured PDE 
\be \label{eq:aspde}
\frac{\partial n}{\partial t}+V(u(t))\frac{\partial n}{\partial a} = -\mu n(t,a), \quad a\in[0,a_M],\,t\geq0,
\ee
where $n(t,a)$ denotes the density of juveniles of age $a$ at time $t$. It is convenient to decouple age from time, and to consider the maturation age $a_M$ to be fixed, while allowing a state-dependent ageing rate $V(u(t))$, with the age being the integral of this velocity $V$.
The form of $V(u)$ will depend on the application and could be decreasing (if adults compete with juveniles for resources) or increasing (in a cooperative society) or a combination of both, but does not matter for our derivation here. The constant $\mu>0$ describes the death rate during maturation, in order to derive the model \eqref{eq:scalarDE}, but one could equally consider $\mu$ with the opposite sign so that there is proliferation during maturation as happens during hematopoiesis; such a scenario with a state-dependent $\mu$ is considered in
\cite{craig2016}.

It is simple to solve along characteristics since
$$-\mu n(t,a)=\frac{\partial n}{\partial t}+V(u(t))\frac{\partial n}{\partial a}=
\frac{\partial n}{\partial t}+\frac{da}{dt}\frac{\partial n}{\partial a}=\frac{dn}{dt}.
$$
Thus
\be \label{eq:ntaM}
n(t,a_M)=e^{-\mu\tau(u_t)}n(t-\tau(u_t),0),
\ee
where individuals that complete the maturation process at time $t$ took time $\tau(u_t)$ to do so. 
Since the ageing rate is $V(u(t))$, the delay $\tau(u_t)$ is defined by \eqref{eq:thre_scalarDE} with $a=a_M$.

The characteristics of \eqref{eq:aspde} have slope
$\frac{da}{dt}=V(u(t))>0$, which is independent of the age. This ensures that there are no shocks or rarefaction fans in the PDE, as all the characteristics are parallel for each fixed $t$. Our approach can be generalised (see for example \cite{GW16}) to age and time-dependent ageing, providing care is taken to ensure that the characteristics are well behaved.

Equation \eqref{eq:aspde} needs to be supplied with appropriate boundary conditions for $n(t,0)$ for $t\geq0$ and for $n(0,a)$ for $a\in[0,a_M]$, and determining the correct conditions is not as trivial as one might think.

We already have all of the elements of the population model
\eqref{eq:scalarDE}, \eqref{eq:thre_scalarDE}, except for the velocity ratio term. To get the velocity ratio term its important to realise, as the conveyor belt analogy illustrates, that $n(t,a_M)$ is a density, and not a flux, and to obtain the correct number of maturing individuals it is necessary to take account of the ageing velocity at the boundaries of the domain. The following arguments are adapted from a more complicated model of blood cell production in \cite{craig2016}.

\begin{figure}
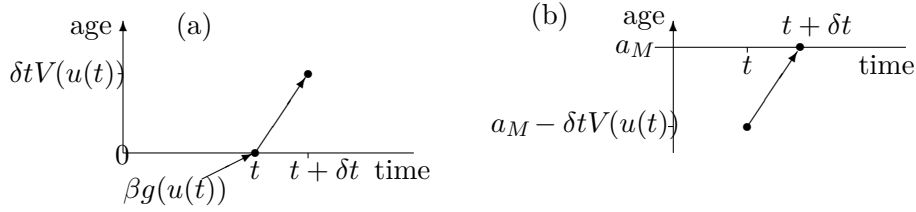

\centering
\mbox{}\hspace*{-9em}
\put(10,0){\line(1,0){110}}
\put(80,-2){\line(0,1){2}}
\put(60,0){\circle*{3}}
\put(80,30){\circle*{3}}
\put(-10,45){age}
\put(30,45){(a)}
\put(10,0){\vector(0,1){50}}
\put(7,-4){$0$}
\put(40,-10){\vector(2,1){19}}
\put(60,0){\vector(2,3){19}}
\put(58,-10){$t$}
\put(73,-10){$t+\delta t$}
\put(105,-10){time}
\put(-33,27){$\delta tV^{\!}(u(t))$}
\put(10,30){\line(-1,0){2}}
\put(10,-17){$\beta g(u(t))$}
\mbox{}\hspace*{17em}
\put(25,40){\line(1,0){95}}
\put(60,38){\line(0,1){2}}
\put(60,10){\circle*{3}}
\put(80,40){\circle*{3}}
\put(13,49){age}
\put(-22,49){(b)}
\put(32,0){\vector(0,1){50}}
\put(10,38){$a_M$}
\put(60,10){\vector(2,3){19}}
\put(58,30){$t$}
\put(73,45){$t+\delta t$}
\put(102,30){time}
\put(-37,8){$a_{M\!}-\delta tV^{\!}(u(t))$}
\put(32,10){\line(-1,0){2}}
    \caption{Schematic of the dynamics at the (a) lower and (b) upper bounds of juvenile ageing that leads to equations 
    \eqref{eq:fluxlb} and \eqref{eq:fluxub}.}
    \label{fig:flux}
\end{figure}

To leading order, juveniles born at time $t$ at rate $\beta g(u(t))$ will have age $\delta t V(u(t))$ at time
$t+\delta t$, as illustrated in Figure~\ref{fig:flux}. Thus we require
\be \label{eq:fluxlb}
\int_{0}^{\delta t V(u(t))}\hspace{-1.5em} n(t+\delta t,a)da-
\int_{t}^{t+\delta t}\hspace{-1.5em}\beta g(u(t))dt=\cO(\delta t^2),
\ee
and hence
\begin{align*} \notag
V(u(t))n(t,0)
&=\lim_{\delta t\to 0}\frac{1}{\delta t}\int_{0}^{\delta t V(u(t))}\hspace{-1.5em} n(t+\delta t,a)da\\
&=\lim_{\delta t\to 0}\frac{1}{\delta t}\int_{t}^{t+\delta t}\hspace{-1.5em}
\beta g(u(t))dt=\beta g(u(t)).
\end{align*}
Thus the correct lower boundary condition for the age-structured PDE is
\be \label{eq:nt0}
n(t,0)=\frac{1}{V(u(t))}\beta g(u(t)).
\ee
Care needs to be taken at the upper boundary too. Combining \eqref{eq:ntaM} and \eqref{eq:nt0} we obtain
\be \label{eq:ntaMsol}
n(t,a_M)=
\frac{1}{V(u(t-\tau(u_t)))}\beta
e^{-\mu\tau(u_t)}g(u(t-\tau(u_t))),
\ee
but this is not the rate $A(t)$ that juveniles mature into adults. 
To leading order, juveniles of age $a\in[a_M-\delta t V(u(t)),a_M]$ will mature during the time interval $[t,t+\delta t]$, as illustrated in Figure~\ref{fig:flux}, 
hence 
\be \label{eq:fluxub}
\int_{a_M-\delta t V(u(t))}^{a_M}\hspace{-1.5em} n(t,a)da-
\int_{t}^{t+\delta t}\hspace{-1em}A(t)dt=\cO(\delta t^2).
\ee
Thus
\begin{align} \notag
A(t)&=\lim_{\delta t\to 0}\frac{1}{\delta t}\int_{t}^{t+\delta t}\hspace{-1em}A(t)dt
=\lim_{\delta t\to 0}\frac{1}{\delta t}\int_{a_M-\delta t V(u(t))}^{a_M}\hspace{-2.5em} n(t,a)da
 =V(u(t))n(t,a_M) \notag \\
& =\frac{V(u(t))}{V(u(t-\tau(u_t)))}\beta
e^{-\mu\tau(u_t)}g(u(t-\tau(u_t))). \label{eq:fluxaM}
\end{align}
Finally, substituting \eqref{eq:fluxaM} into \eqref{eq:ueq} gives the scalar DDE model \eqref{eq:scalarDE} including the velocity ratio term. In \eqref{eq:aspde} we consider a fixed maturation age $a_M$, but 
\cite{cassidy2019equivalences} shows that 
the same ratio arises  
when the maturation condition is a random variable.

We did not deal with the boundary condition 
of \eqref{eq:aspde} at $t=0$ yet. If we let
$n(0,a)=\psi(a)$ then as illustrated in
Figure~\ref{fig:agestr}(a) the DDE
formulation \eqref{eq:scalarDE} would only be valid when $t-\tau(u_t)\geq0$,
otherwise the characteristics hit $t=0$ with a nonzero age and \eqref{eq:ntaM} would not be valid. 

\cite{smith1993reduction} defined a second differential equation for the initial time interval, but it is much more convenient to work with \eqref{eq:scalarDE}
for all $t\geq0$ along with a suitable initial history function \eqref{eq:rfdeic}. This can be done, as suggested in Figure~\ref{fig:agestr}(a),
by extending the characteristics from $t=0$ back to age $a=0$. Then solving along the extended characteristic similar to \eqref{eq:ntaM}
and using \eqref{eq:nt0} we obtain
\begin{align} \notag
\psi(a) &=n(0,a)=
e^{-\mu\tau(a)}n(-\tau(a),0)
=\frac{e^{-\mu\tau(a)}}{V(u(-\tau(a)))}\beta g(u(-\tau(a)))\\
&=\frac{e^{-\mu\tau(a)}}{V(\phi(-\tau(a)))}\beta g(\phi(-\tau(a))). \label{eq:psiphi}
\end{align}
where the delay $\tau(a)$ is defined by \eqref{eq:thre_scalarDE} with $t=0$. Equation \eqref{eq:psiphi} along with \eqref{eq:thre_scalarDE} defines a relation between the initial condition $n(0,a)=\psi(a)$ for the age-structured PDE \eqref{eq:aspde} and the history function $\phi$ for the DDE \eqref{eq:scalarDE} that allows us to apply the DDE for $t\geq0$ with an initial function $\phi$. In reality nobody uses \eqref{eq:psiphi}, as in dynamical systems we are usually interested in the asymptotic behaviour and the initial functions are often chosen quite arbitrarily, often as constant functions.

\CCLsection{Delay Differential Equations as Dynamical Systems}
\label{sec:dynsys}

Consider the DDE IVP
\begin{gather} 
\udot(t)=f(u(t),u(t-\tau)) \label{eq:dde}  \\
u(\theta)=\phi(\theta), \quad \theta\in[-\tau_\textit{max},0],   \label{eq:ddephi}
\end{gather}
where $u(t)\in\R^d$. Here we take the initial time $t_0=0$.
The delay $\tau$ may be constant, discrete-state-dependent or defined by a threshold 
condition \eqref{eq:thre_scalarDE}. 
Informally we sometimes write the discrete-state-dependent delay as $\tau=\tau(u(t))$,
but in the RFDE notation \eqref{eq:ut} this is more properly stated as  $\tau=\tau(u_t(0))$.
Everything in this section also applies to problems with multiple delays, but we
consider \eqref{eq:dde} with one delay for notational simplicity.

We will assume that \eqref{eq:dde} is sufficiently well behaved so that the existence and uniqueness 
results of Theorems~\ref{thm:Driverexist}-\ref{thm:Driveruniq} apply. In particular this requires that
$f=f(u,v)$ is a locally Lipschitz continuous function
$f:\R^d\times\R^d\to\R^d$  of its arguments, and also that 
$\phi$ is Lipschitz. 

The state-dependent delays also need to be well behaved to apply the theorems, in particular
we require $\tau\geq0$. We remark that for specific examples, such as the linear state-dependent delays
considered in Section~\ref{subsec:statedep} this can be ensured either by imposing parameter constraints,
or by modifying the definition of the delays suitably. The threshold delay    
\eqref{eq:thre_scalarDE} can never become advanced, and
the differentiated form \eqref{eq:taudot} allows the theory to be applied. 
It also follows from \eqref{ineq:tminustau_up} that $t-\tau$ is a monotonically increasing function of $t$ along any solution trajectory. 

Crucially to apply dynamical systems theory we require $f$ in \eqref{eq:dde} to be autonomous (not to depend explicitly on $t$), and for solutions to be unique. Uniqueness ensures that solutions cannot cross in phase space,
and this together with Lipschitz continuity of $f$ is what renders phase portraits both useful and beautiful for 
rendering dynamics, whether it be the classical Lorenz attractor or the tori shown in Figure~\ref{fig:tori}.

Recalling that the phase space for a dynamical system is essentially the space of initial conditions, for the
DDE \eqref{eq:dde} with a constant delay $\tau$, we already argued after \eqref{eq:rfdeic} that
the correct space in which to consider the function elements is the space of continuous functions $C=C\bigl([-\tau,0],\mathbb{R}^d\bigr)$, because of the discontinuity in the derivative of $u_t(\theta)$ at 
$\theta=-t$ for $t\in(0,\tau)$. However, with $\phi$ and $f$ both Lipschitz continuous, 
the discontinuity in the derivative of $u_t$ leaves it Lipschitz too, so we should really consider
$C$ to be the space of Lipschitz continuous functions. 

The space $C$ includes all polynomials, so the phase space is immediately infinite dimensional, 
even for scalar problems with $u(t)\in\R$. This means that even scalar DDEs can potentially display chaotic dynamics (which scalar first order ODEs cannot), so DDEs are great for 
generating interesting and rich dynamics. On the other hand, its problematical for the mathematical analysts, as any ODE proofs using the fact that closed and bounded sets in $\R^d$ are compact, will break down for DDEs because that is not true in an infinite dimensional space. As discussed in Section~5.3 of \cite{Smith11}, there is a work around using equicontinuity and the Arzel\`a-Ascoli Theorem, but we will not worry too much about such details here.

Another issue with an infinite dimensional phase space, is how to represent solutions visually. For ODEs in two or three space dimensions its usual to display solutions as parametrised curves in phase space, but for a DDE we will need to project the phase portrait into finite dimensions to display it.
Its not clear how to best do this. When $u(t)$ is analytic, the solution will be equal to its Taylor series, so the most obvious and natural projection would be to project the function $u(t)$ onto the coefficients of its Taylor series evaluated at $t$, and so plot the first few coordinates of $(u(t),\udot(t),\uddot(t),\ldots)$ or 
$(u(t),\udot(t)/1!,\uddot(t)/2!,\ldots)$ as parametrised curves. But this is problematical because the breaking points
mean that generally solutions are not analytic, and even when they are the derivatives may not scale nicely. 

\cite{GlassMackey79} had the idea of plotting $u(t-\tau)$ against $u(t)$ for the Mackey--Glass equation, and the results are so pleasing that nobody seems to have asked whether or why that may be the best projection.  It is now standard to draw $(u(t),u(t-\tau))$ (or sometimes $(u(t),u(t-\tau_1),u(t-\tau_2))$ when there are multiple delays) as a curve in $\R^{2d}$. 
Interestingly, the DDE equation \eqref{eq:dde} can be regarded as a
map from $(u(t),u(t-\tau))$ to $(u(t),\udot(t))$ so these two phase space projections are closely related.

To further complicate our discussion of the phase space $C$, notice that when
the solution is not continuously differentiable at $t=t_0$; 
$$\udot(t_0^-)=\dot{\phi}(0)\ne f(\phi(0),\phi(-\tau))=\udot(t_0^-)$$
further discontinuities in higher derivatives of the solution are generated at subsequent points $t_0+k\tau$ for $k\in\N$, since differentiating \eqref{eq:dde} leads to
\be \label{eq:ds:uddbkpt}
\uddot(t)=\udot(t)\,f_u(u(t),u(t-\tau))+\udot(t-\tau),f_v(u(t),u(t-\tau)),
\ee
where $f_u$ and $f_v$ denote the derivatives of $f(u,v)$ with respect to its first and second arguments.
Thus generally (more precisely when $f_v(u(t_0+\tau),u(t_0))\ne0$), it follows that 
$\ddot{u}(t_0+\tau)$ is discontinuous when $\dot{u}(t_0)$ is discontinuous. Similarly it follows that
$u^{(k+1)}(t)$ may be discontinuous at $t=t_0+k\tau$ for $k\in\N$. These points where some derivative of $u(t)$ is discontinuous are called \emph{breaking points}. The smoothing property associated with them means that 
along any solution trajectory $u(t)$ is $C^{k+1}$ for all $t\geq t_0+k\tau$. If there are two delays $\tau_1$ and $\tau_2$ there will be potential breaking points at  $t=t_0+j\tau_1+k\tau_2$ for $j,k\in\N_0$ (though if $\tau_1$ and $\tau_2$ are rationally related there will be some repetitions in this list).

Additional issues arise when considering state-dependent DDEs as dynamical systems. The breaking points will still be important, but with a discrete state-dependent delay $\tau=\tau(u(t))$ the first breaking point after $t_0$ will be defined by $t=\xi$ where $\xi-\tau(u(\xi))=t_0$ and subsequent breaking points similarly. Thus the location of the breaking points depends on the solution, and must be found together with the solution of the DDE \eqref{eq:dde}.
We will discuss how to do this for numerical solutions in Section~\ref{sec:bps}.

Some care should also be taken with the definition of the phase space when $\tau$ is state-dependent. The existence and uniqueness results in Theorems~\ref{thm:Driverexist}-\ref{thm:Driveruniq}
only require that $t-\tau(u(t))\geq t-\tau_0$ for all $t\geq 0$, but this allows for larger and larger delays as $t$ increases. It is usual to define a maximal delay $\tau_\textit{max}$ and consider  
$C$ to be the space of (Lipschitz) continuous functions $C=C\bigl([-\tau_{\textit{max}},0],\mathbb{R}^d\bigr)$.
In the current work we will consider $\tau_\textit{max}<\infty$, but already \cite{Driver63} considered the possibility of unbounded delays. 

Defining a maximal delay is somewhat unsatisfactory as it will lead us to define function segments $u_t(\theta)=u(t+\theta)$ for $\theta\in[-\tau_\textit{max},0]$ at each point of the solution, even if $\tau(u(t))\ll\tau_\textit{max}$. If $\tau_\textit{max}>\tau(u(t_0))$ it is then possible to define two initial functions $\phi_1$, $\phi_2$ with $\phi_1(t)\ne\phi_2(t)$ for $t<-\tau(u(t_0))$ but
$\phi_1(t)=\phi_2(t)$ for $t\in[-\tau(u(t_0)),0]$, so that the two different initial functions lead to the same solution of \eqref{eq:dde}-\eqref{eq:ddephi}; which is usually not possible in Lipschitz continuous dynamical systems because of results that bound the rate that trajectories can converge towards one another.
When $t-\tau$ is a monotonically increasing function of $t$ along all solution trajectories, as happens for the threshold delay \eqref{eq:thre_scalarDE} when $V_\textit{min}>0$, it is possible to ignore the spurious extra initial data that is not needed. It is a little messy to try to codify this into a definition of $C$, 
which will  thus not do.

\CCLsubsection{Invariant Sets and Bifurcations}
\label{sec:invsets}

To consider \eqref{eq:dde} as a dynamical system, let $u_t$ be the function element
\be \label{eq:utds}
u_t(\theta)=u(t+\theta), \quad \theta\in[-\tau_\textit{max},0],
\ee
where the delay $\tau(u_t)$ is either defined by a discrete state-dependent delay
$\tau=\tau(u_t(0))$ (in the notation of \eqref{eq:ut})
or by a threshold condition \eqref{eq:thre_scalarDE}.

We define the \emph{evolution operator} $S(t)$ for the dynamical system on $C$ by
\be \label{eq:St}
S(t)u_s=u_{s+t}, \quad t\geq0, \; s\geq 0=t_0.
\ee
For a solution of \eqref{eq:dde} for $t\geq0$ for some particular initial function $\phi$, we can think of $u_s$ as the restriction of this solution to a particular time interval $[s-\tau_\textit{max},s]$, and regard $S(t)$ as just translating this time window to $[t+s-\tau_\textit{max},t+s]$, so it can be thought of as just convenient shorthand notation for the solution. However, since it can be applied to any initial function $\phi\in C$,
\be \label{eq:Stphi}
S(t)\phi=S(t)u_0=u_{t}, \quad t\geq0,
\ee
it defines a family of mappings $S(t):C\to C$ for all $t\geq0$, which map an initial function $\phi\in C$ to the solution
element $u_t\in C$ that solves \eqref{eq:dde},\eqref{eq:ddephi}. 

The evolution operator $S(t)$ is easily seen to be associative and commutative;
$S(t_1)S(t_2)=S(t_2)S(t_1)=S(t_1+t_2)$ for all $t_1$, $t_2\geq0$, so $S(t)$ defines a semigroup.
Since $S(0)=I$, the identity map on $C$, $S(t)$ actually defines a commutative monoid, however it does not define a group,
because there is no inverse operator, as we have only defined $S(t)$ for $t\geq0$. As we will illustrate later, solving DDEs backwards in time for general initial functions is very  problematical, and leads to faster than exponential blow up. 

With $S(t):C\to C$ defined for $t\geq0$, the mapping is said to define a \emph{semi-dynamical system}. There is well established theory for semi-dynamical systems, which is useful in applications where backwards in time solutions may not be defined, as 
occurs for DDEs and RFDEs \cite{Smith11}, and also for maps defined by numerical methods \cite{SH96}.

The evolution operator allows us to define invariant sets. A set 
$A\in C$ is \emph{forward invariant} if $S(t)u\in A$ for all $u\in A$ and all $t\geq0$.
The usual definition of backward invariance for a dynamical system with an evolution operator defined for all $t\in\R$
is to require that  $S(t)u\in A$ for all $u\in A$ and all $t\leq0$. However, we can not do that for a semi-dynamical system since $S(t)$ is not defined for $t<0$.  Instead we say that $A$ is \emph{backward invariant} if for all $u\in A$ and all $t\geq0$ there exists $v\in A$ such that $S(t)v=u$. The set $A$ is \emph{invariant} if it is both forward and backward invariant.

If $A$ is an invariant set, then there is a complete orbit passing through every point of $A$; so for $u\in A$, we can define $S(t)u\in A$ for all $t\in\R$.
This then allows for the definition of $\alpha$- and $\omega$-limit sets and attractors.
To show that $\omega$-limit sets are compact and non-empty is not as simple as for the finite-dimensional case of ODEs, but can be done using the Arzel\`a-Ascoli Theorem; see Chapter 5 of \cite{Smith11}, where this is done for RFDEs.

Steady states are the simplest invariant sets, but periodic orbits, invariant tori (recall Figure~\ref{fig:tori}), and chaotic attractors are all of interest. Strictly speaking a steady state under the evolution operator $S(t)$ is a function $\phi\in C$ such that $S(t)\phi=\phi$ for all $t\in\R$, which requires the function $\phi$ be a constant function. It is then usual to identify the constant function $\phi\in C$ with the value $u^*\in\R^d$ that it takes. It is quite simple to solve for the steady states, 
since $\udot=0$ at a steady state, and $u_t$ is a constant function, $ u(t)=u(t-\tau)=u^*$. Then the steady states of \eqref{eq:dde} are given by
\be \label{eq:ddess}
0=f(u^*,u^*).
\ee

We will be interested in the stability of steady states and other invariant sets, and the bifurcations that occur as parameters are varied. We will obtain the stability of steady states by linearizing around then. 
In contrast to partial differential equations for which there is no general principle of linearised stability,
DDEs and RFDEs, although defining infinite-dimensional dynamical systems are better behaved, and do have possess a principle of linearised stability \cite{Smith11}. Consequently a steady state will be stable if 
all the characteristic values $\lambda$ have negative real part, and unstable if one or 
more characteristic values $\lambda$ have positive real part.
Floqu\'et theory generalises this technique to periodic orbits. 

On a periodic orbit of \eqref{eq:dde}, or indeed on any complete orbit contained in a compact invariant set, the smoothing property described in Section~\ref{sec:dynsys} 
ensures that the solution is infinitely differentiable when $f$ is infinitely differentiable.
When the delay is constant,
it then follows from the results of \cite{RN73} that the solution is analytic when $f$ is analytic. 
However, \cite{JMPRN14} showed that if discrete delays are allowed to be time or state dependent then solutions may not be analytic for all $t$. \cite{Krisztin14} on the other hand showed that with a threshold delay defined by \eqref{eq:thre_scalarDE} the bounded solutions of \eqref{eq:dde} are analytic when $f$ and $V$ are both analytic.

Bifurcations occur at changes of stability, so we will concentrate on finding stability changes to detect bifurcations. 
Much of the bifurcation theory has been developed for RFDEs \eqref{eq:rfde}, for which $F$ is not even Lipschitz when the
the delay is state-dependent, and this lack of smoothness has greatly impeded the development of the theory.
Theorems for Hopf bifurcations with state-dependent delays have only recently been proven and can be found in \cite{Eichmann06,HuWu10,Sieber12}.
On the other hand the numerical bifurcation and continuation package 
DDE-Biftool \cite{ddebiftool} has had the ability to detect Hopf bifurcations since before those results were derived, and can also find more complicated co-dimension two bifurcations for which there is still not a complete theory because of the lack of smoothness of the centre manifold. We will explore stability and bifurcations doing our best not to get bogged down in the technical details. 

The linearisation of the RFDE
\be \label{eq:rfdeds}
\udot(t)=F(u_t)
\ee
about the steady state $u_t=u^*\in\R^d$ is 
\be \label{eq:Lds}
\udot(t)= L u_t
\ee
where $L$ is the Fr\'echet derivative of $F$ evaluated at the constant function $u_t=u^*\in\R^d$. If you do not know what this is do not worry; for a state-dependent delay, $F$ is not Lipschitz and so not differentiable and this does not work anyway. 
It was only in the 21$^{st}$ century that a rigorous theory 
was developed for the 
linearisation of state-dependent delay problems. 
However, it involves working in a subspace of $C$ on which all elements $\phi$ are $C^1$ and satisfy $\phi'(0)=F(\phi)$, and essentially differentiating there, before  
extending the result to $C$ \cite{Walther03,HKWW06}. In the hope of reaching readers for whom such analysis arguments are opaque, we will show how to linearise heuristically for both discrete state-dependent delays and threshold delays, and also illustrate the pit holes to avoid falling down.

\CCLsubsection{Linearisation for Discrete Delays}
\label{sec:linearize}

Consider \eqref{eq:dde} as a scalar constant-delay DDE with $u(t)\in\R$ and suppose
$u^*$ satifies \eqref{eq:ddess}  
so $u=u^*$ is a steady state. Let $w(t)=u(t)-u^*$ so
\begin{align*}
\wdot(t)&=\udot(t)=f(u(t),u(t-\tau))=f(u^*+w(t),u^*+w(t-\tau))\\
&=f(u^{*\!},u^*)+f_{u\!}(u^{*\!},u^*)w(t)+f_{v\!}(u^{*\!},u^*)w(t-\tau)+\textit{higher order terms}
\end{align*}
using the Taylor series in two variables for $f(u,v)$. Dropping the higher order terms
(arising from higher derivatives of $f$) we obtain the linearised equation
\begin{align} \notag 
\wdot(t) & =f_u(u^{*\!},u^*)w(t)+f_v(u^{*\!},u^*)w(t-\tau)\\
& =\mu w(t)+\sigma w(t-\tau), \label{eq:ddel1} 
\end{align}
for constants $\mu\in\R$, $\sigma\in\R$.
Positing $u(t)=e^{\lambda t}$ gives a transcendental \emph{characteristic equation}, $\Delta(\lambda)=0$,
where $\Delta$ is the characteristic function
\be \label {eq:charfn1}
\Delta(\lambda)=\lambda-\mu-\sigma e^{-\tau\lambda}.
\ee
Let $\lambda=x+iy$ and take real and imaginary parts of $\Delta(\lambda)=0$ to obtain
\begin{gather}
x-\mu-\sigma e^{-\tau x}\cos(y\tau)=0, \label{eq:charfn1real} \\
y+\sigma e^{-\tau x}\sin(y\tau)=0. \label{eq:charfn1imag}
\end{gather}
There are infinitely many solutions of \eqref{eq:charfn1real},\eqref{eq:charfn1imag} which all lie
on the curve
$$y=\pm\sqrt{\sigma^2e^{-2\tau x}-(x-\mu)^2}.$$
All the characteristic roots satisfy $Re(\lambda)=x<|\mu|+|\sigma|$, and since the roots are isolated there are only finitely many to the right of any vertical line in the complex plane. 
This implies that there are only finitely many characteristic values $\lambda$ with $Re(\lambda)>0$, so 
the linear unstable manifold of the steady state $u^*$ is finite dimensional.
In contrast, the stable manifold is infinite-dimensional, since there are always infinitely many $\lambda$ with $Re(\lambda)<0$. The solution of \eqref{eq:ddel1} is given by
$$w(t)=\sum_{\lambda_j\in\sigma_+}c_j e^{\lambda_j t} + \sum_{\lambda_j\in\sigma_-}c_j e^{\lambda_j t}$$
where $\sigma_-$ denotes the set of characteristic values with $\Re(\lambda)<0$, and 
$\sigma_+$ those with $\Re(\lambda)\geq0$. Since there are infinitely many $\lambda_j\in\sigma_-$ with
$\Re(\lambda)<-M$ typical solutions grow faster than $e^{-Mt}$ backwards in time, for every $M>0$ and so grow faster than any exponential in time. This is why we only defined the evolution operator $S(t)$ for $t\geq0$. Only the special solutions with $c_j=0$ for all $\lambda_j$ with $Re(\lambda_j)<0$ will stay bounded backwards in time; but these are precisely the solutions on the linear centre and stable manifolds of the steady state. 

DDEs posed for $u(t)\in\R^d$ can be linearised similarly.  Consider
\be \label{eq:ddem}
\udot(t)=f(u(t),u(t-\tau_1),\ldots,u(t-\tau_m)),
\ee
where $f(u,v_1,\ldots,v_m):\R^d\times\R^{md}\to\R^d$ satisfies $f(u^*,u^*,\ldots,u^*)=0$,
so $u^*$ is a steady state. Then the linearisation about $u^*$ is
$$\wdot(t)=A_0w(t)+\textstyle\sum_{j=1}^m A_jw(t-\tau_j),$$
where $A_0=f_u$ and $A_j=f_{v_j}$ are $d\times d$ matrices evaluated at the steady-state
(essentially a Jacobian matrix for each delay).
There is a nontrivial solution $u(t)=e^{\lambda t}\underline{v}\in\R^d$ with 
$\Delta(\lambda)\underline{v}=0$ if
\[0=\det(\Delta(\lambda)), \quad\textit{where} \quad \Delta(\lambda)=\lambda I_d-A_0-\textstyle\sum_{j=1}^m A_je^{-\lambda\tau_j}.\]
This characteristic equation again has infinitely many roots, but with 
only finitely many $\lambda$ with $Re(\lambda)>M$ for any $M\in\R$.

\CCLsubsubsection{Linearisation with State-Dependent Delays}
Equation \eqref{eq:ddess} also defines the steady states of
the state-dependent DDE
\be \label{eq:ddesd}
\udot(t)=f(u(t),u(t-\tau(u_t))).
\ee
The original heuristic method to linearise about a steady state $u^*$
is to freeze the delay $\tau(u_t)$ at its steady state value $\tau(u^*)$ and then linearise the resulting constant delay DDE
\be \label{eq:ddefr}
\udot(t)=f(u(t),u(t-\tau(u^*))),
\ee
as described above. For example, 
freezing the delays of the state-dependent delay DDE \eqref{eq:twostatedep} 
at their steady state value $u=0$ results in
\be \label{eq:twostatedepfrozen}
\udot (t) = - \gamma u(t)- \sum_{i=1}^2\kappa_i u(t-a_i). 
\ee
Since \eqref{eq:twostatedepfrozen} is already linear in this case, this completes the linearisation.

An analytic approach using the semiflow in an appropriate Banach space developed in \cite{Walther03,HKWW06}
(see also \cite{cooke1996problem, gyori2007exponential}) leads to the same linearisation 
when the delay $\tau(u_t)$ is a discrete delay, for example when $\tau(u_t))=\tau(u_t(0)))=\tau(u(t))$, 
but shows that freezing the delay leads to an incorrect linearisation for distributed delays, such as threshold delays. Importantly, the work of \cite{Walther03,HKWW06}
establishes the \emph{principle of linearised stability} for DDEs with state-dependent delays.

\CCLsubsubsection{Expansion for State-Dependent Delays}

In order to study bifurcations for differential equations it is necessary to study the nonlinear dynamics on the centre manifold. For the state-dependent DDE \eqref{eq:twostatedep} where the only nonlinearity is generated by the state-dependency of the delays this will necessitate an expansion of the delay terms. However, expanding the delay terms is problematical. For a constant delay $\tau$ many authors have attempted to remove the delay from \eqref{eq:dde} or \eqref{eq:ddel1} using 
\be \label{eq:utauexp}
u(t-\tau)=u(t)-\tau \udot(t) +\frac12 \tau^2 \uddot(t)+\cO(\tau^3),
\ee
and truncating after a suitable number of terms.
This is obviously doomed to failure when $\tau$ is large, as the $\cO(\tau^3)$ terms are already large. Perhaps more surprisingly even when $\tau\ll1$ replacing $u(t-\tau)$ by the first two or three terms on the right-hand side of \eqref{eq:utauexp} can change the stability of the steady state when it changes the order of the differential equation (see Example 6 on page 254 of \cite{Driver77}).
On the other hand, it is shown in Section 4.4 of \cite{Smith11} 
that small delays are harmless in linear DDEs; neglecting them does not change the stability.

\cite{CHK17}  study the state-dependent delay DDE \eqref{eq:twostatedep} 
by expanding the delay term $u(t-\tau(u(t))$ about
$u(t-\tau(u^*))$ instead of $u(t)$, and using the DDE to remove higher derivative terms, to conserve the order of the differential equation.  We summarise the approach here.

For \eqref{eq:twostatedep} there are two state-dependent delays 
\be \label{eq:2ssdels}
\tau_j(u_t)=\tau_j(u_t(0))=\tau(u(t))=a_j+c_j u(t), \quad j=1,2,
\ee
and the steady state of \eqref{eq:twostatedep} is $u^*=0$. Thus at the steady state the delays are given by
$\tau_j(u^*)=\tau_j(0)=a_j$ for $j=1,2$. Now consider
\be \label{eq:2ssexp1}
u(t -a_i - cu(t)) =  u(t-a_i) + \dot{u}(t-a_i)(-cu(t))
+ \frac{1}{2}\ddot{u}(t-a_i)(-cu(t))^2 + h.o.t.,
\ee
where $h.o.t.$ denotes `higher order terms.'
However, this expansion introduces higher derivative terms, and we use the original DDE to remove these.
From \eqref{eq:twostatedep} we have
\begin{align} \notag
\udot(t-a_i)&=-\gamma u(t-a_i)-\sum_{j=1}^2\kappa_j u(t-a_i-a_j-c_j u(t-a_i))\\ \notag
& = -\gamma u(t-a_i)\\ \notag
& \phantom{=} \; -\sum_{j=1}^2 \kappa_j\Big[u(t-a_i-a_j)-c_j u(t-a_i)\udot(t-a_i-a_j)+h.o.t.\Big]\\
& = -\gamma u(t-a_i)-\sum_{j=1}^2 \kappa_ju(t-a_i-a_j)+h.o.t. \label{eq:2ssutai}
\end{align}

To deal with the $\ddot{u}(t-a_i)$ term in \eqref{eq:2ssexp1} differentiate \eqref{eq:twostatedep}
to obtain
$$\ddot{u}(t)=-\gamma \udot(t)-[1-c_j\udot(t)]\sum_{j=1}^2\udot(t-a_j-c_j u(t)).$$
Thus
\begin{align*}
\ddot{u}(t-a_i)&=-\gamma \udot(t-a_i)-[1-c_j\udot(t-a_i)]\sum_{j=1}^2\udot(t-a_i-a_j-c_j u(t-a_i)).
\end{align*}
But $\udot(t-a_i-a_j-c_j u(t-a_i))=\udot(t-a_i-a_j)+h.o.t.$ and both $\udot(t-a_i)$ and 
$\udot(t-a_i-a_j)$ can be evaluated using \eqref{eq:2ssutai}.

Substituting these expression into \eqref{eq:2ssexp1} then substituting this into \eqref{eq:twostatedep}
we obtain
\begin{align} \notag
\udot &(t) = - \gamma u(t)- \textstyle{\sum_{i=1}^2} \kappa_i u(t-a_i) \quad\textrm{(linear)}\\ \notag
&-\textstyle{\sum_{i=1}^2} \kappa_i cu(t)\Bigl[\gamma u(t-a_i)+\textstyle{\sum_{j=1}^2}\kappa_j  u(t - a_i - a_j )\Bigr]
\quad\textrm{(quadratic)}\\ \notag
& \scriptstyle -\sum_{i,j=1}^2\kappa_i\kappa_j c^2u(t)u(t-a_i)\Bigl[\gamma u(t-a_i-a_j)+\sum_{m=1}^2\kappa_m
u(t-a_i-a_j-a_m)\Bigr]\\ \notag
& \scriptstyle -\frac12(cu(t))^2\sum_{i=1}^2\kappa_i\Bigl[\gamma^2u(t-a_i)+2\gamma\sum_{j=1}^2\!\kappa_ju(t-a_i-a_j)
+\sum_{j,m=1}^2\!\!\kappa_j\kappa_mu(t-a_i-a_j-a_m)\!\Bigr]\\
& +\cO(4). \label{eq:2ssconst}
\end{align}
Equation \eqref{eq:2ssconst} defines an expansion in constant delays of \eqref{eq:twostatedep} about the steady state $u=0$. The linear terms in the first line of \eqref{eq:2ssconst} are identical
to the `frozen delay' linearisation \eqref{eq:twostatedepfrozen}. However, this technique allows us to also derive the quadratic and higher order terms in an expansion with only constant delays at the cost of introducing extra delays; to be precise there will be up to $n(n+3)/2$ delays at $n^{th}$ order (so 2 delays in the linear terms, 5 including quadratic terms, 9 including cubic and so on). 

Of course, these expansions are only valid for solutions which are sufficiently differentiable, which is assumed to be the case on the stable and unstable manifolds near to the steady state. \cite{CHK17} present these expansions in more detail using truncated series with integral remainder terms, and used them to compute the normal forms of codimension-two Hopf-Hopf bifurcations. Numerical computations of the full nonlinear system \eqref{eq:twostatedep} are used to validate the results. \cite{Sieber17} provided a partial justification for this approach and also generalised the approach for general discrete state-dependent delay problems, to compute normal forms for a multitude of bifurcations using expansions of the state-dependent delays as constant delays, and this is also now implemented in DDE-Biftool (see \cite{ddebiftool})
for normal form computations of bifurcations for DDEs with state-dependent delays.

\CCLsubsection{Linearisation of Threshold Delays}

Let $u^*$ be a steady-state of the system \eqref{eq:scalarDE},\eqref{eq:thre_scalarDE}. 
Evaluating the threshold condition \eqref{eq:thre_scalarDE},
the delay $\tau$ at steady state is determined by
\begin{equation} \label{eq:tauss} 
	\tau = \tau(u^*) =\frac{a}{V(u^*)}. 
\end{equation}
Since $V(u(t))=V(u(t-\tau)=V(u^*)$ at steady state, the velocity ratio in \eqref{eq:scalarDE} is equal to 1 at any steady state,
which is then defined by the algebraic equation  $h(u^*)=0$ where
\begin{equation} \label{eq:hofu}
	h(u)=\beta e^{-\mu\tau(u)}g(u) - \gamma u,
\end{equation}
with the delay $\tau(u)$ given by \eqref{eq:tauss}. 

Following \cite{gedeon2024dynamics},
suppose that $0\leq g(u)\leq g_{\textit{max}}<\infty$. 
Then $h(0)=\beta e^{-\mu\tau(0)}g(0) \geq 0$ and $h(u)\leq \beta g_{\textit{max}} -\gamma u$,
which implies that $h(u)<0$ for $u>\beta g_{\textit{max}}$. 
Therefore by continuity of $h$ there exists at least one steady state $u^*\in[0,\beta g_{\textit{max}}/\gamma]$.

As for the case of discrete delays, we will present a heuristic expansion to perform the linearisation about steady states for threshold delays. For brevity we will not consider higher order terms, but
just as for discrete delays the expansions could be performed to higher order. 
We will present two elementary approaches here. In the first, we perform a truncated expansion of the solution to obtain the linear dynamics, and in the second we first expand the nonlinear functions and then truncate to a linear system, whose solution supplies the linear dynamics.

\CCLsubsubsection{Linearisation by Perturbation from Linear Dynamics}

We illustrate this technique using the scalar model \eqref{eq:scalarDE},\eqref{eq:thre_scalarDE}. 
We begin by assuming that $\|u(t)-u^*\|\ll1$ and
\begin{equation} \label{eq:u_ptb}
	u(t)-u^*=\cE e^{\lambda t}+\cO(\cE^2).
\end{equation}
Here $\cE$ is a small parameter, so \eqref{eq:u_ptb} represents a small perturbation from linear dynamics.
Then, expanding $V(u)$ around $u^*$ we obtain
\begin{align*}
	V(u(s))& = V(u^*)+(u(s)-u^*)V'(u^*)+\cO(\|u(s)-u^*\|^2)\\
	&  = V(u^*)+\cE V'(u^*)e^{\lambda s}+\cO(\cE^2).
\end{align*}
Substituting this into the threshold condition \eqref{eq:thre_scalarDE} and integrating gives
\begin{align*}
	\tau^*V(u^*)  =a &= \int_{t-\tau(t)}^{t} V(u(s)) ds \\
	&= \int_{t-\tau(t)}^{t} V(u^*)+\cE V'(u^*)e^{\lambda s}+\cO(\cE^2)ds\\
    &= \tau(t)V(u^*)+\cE V'(u^*)\frac{e^{\lambda t}}{\lambda}\Big[1-e^{-\lambda\tau(t)}\Big]+\cO(\cE^2),
\end{align*}
which leads to
\begin{equation} \label{eq:tau_linapp}
	\tau(t)=\tau^*-\cE \frac{V'(u^*)}{V(u^*)}\frac{e^{\lambda t}}{\lambda}\Big[1-e^{-\lambda\tau^*}\Big]+\cO(\cE^2).
\end{equation}
Using \eqref{eq:tau_linapp} we obtain
\begin{gather} \label{eq:expo_linapp}
	e^{-\mu\tau(t)}=e^{-\mu\tau^*}\Big[
	1+\cE\mu\frac{V'(u^*)}{V(u^*)}\frac{e^{\lambda t}}{\lambda}(1-e^{-\lambda\tau^*})
	\Big]+\cO(\cE^2), \\
 \label{eq:velratio_linapp}
	\frac{V(u(t)))}{V(u(t-\tau(t)))}=
	1+\cE\frac{V'(u^*)}{V(u^*)}e^{\lambda t}(1-e^{-\lambda\tau^*})+\cO(\cE^2),
\end{gather}
and similarly 
\begin{align} \notag 
	g(u(t-\tau(t))) &  = g(u^*+\cE e^{\lambda(t-\tau(t))}+\cO(\cE^2))\\
& = g(u^*)+\cE g'(u^*) e^{\lambda(t-\tau^*)} +\cO(\cE^2). \label{eq:g_linapp}
\end{align}
Substituting \eqref{eq:u_ptb}-\eqref{eq:g_linapp} into \eqref{eq:scalarDE} and simplifying using the condition $h(u^*)=0$ yields
\begin{align*}
	\cE\lambda e^{\lambda t} & =\frac{d}{dt}(u^*+\cE e^{\lambda t})
	= \beta e^{-\mu\tau(t)}\frac{V(u(t )))}{V(u(t-\tau(t)))}g(u(t -\tau (t))) - \gamma u (t)\\
	&= \cE \beta e^{-\mu\tau^*}\Big[g(u^*)\frac{V'(u^*)}{V(u^*)}(1-e^{-\lambda\tau^*})e^{\lambda t}(\frac{\mu}{\lambda}+1)+g'(u^*) e^{\lambda(t-\tau^*)} \Big]\\
& \qquad\quad-\gamma\cE e^{\lambda t}+\cO(\cE^2).
\end{align*}
Cancelling $e^{\lambda t}$ and dropping the higher order terms, we obtain that the characteristic equation for \eqref{eq:scalarDE}-\eqref{eq:thre_scalarDE} is $\Delta_{\textit{thres}}(\lambda)=0$ where
\begin{equation} \label{eq:char_scalarDE}
\Delta_{\textit{thres}}(\lambda)=	\lambda+\gamma-\beta e^{-\mu\tau^{*\!}}\Big[g(u^*)\frac{V'^{\!}(u^*)}{V^{\!}(u^*)}(1-e^{-\lambda\tau^*})(1+\frac{\mu}{\lambda})+g'^{\!}(u^*)e^{-\lambda\tau^*}\Big].
\end{equation}
This derivation of \eqref{eq:char_scalarDE} assumes linear dynamics to get the characteristic equation directly, without first linearising the system \eqref{eq:scalarDE},\eqref{eq:thre_scalarDE}.

\CCLsubsubsection{Linearisation by Expansion}

An alternative derivation of \eqref{eq:char_scalarDE}, is to first linearise the system \eqref{eq:scalarDE},\eqref{eq:thre_scalarDE}, and then determine the dynamics of the linearised system.
Begin by defining new variables $\tilde{u}(t)$ and $\tilde{\tau}(t)$ by 
\begin{equation} \label{eq:tildeu_tildetau}
	\tilde{u}(t) = u(t)-u^*, \quad \tilde{\tau}(t)=\tau(t)-\tau^*. 
\end{equation}
Consider $(u(t), \tau(t))$ close to the steady state $(u^*, \tau^*)$ so that $\|\tilde{u}\|, \|\tilde{\tau}\|\ll1$. 

The behavior of $\tilde{\tau}(t)$ is found by linearizing the threshold condition \eqref{eq:thre_scalarDE}
\begin{align*}
	a &= \int_{-\tau(t)}^{0} \hspace*{-1em}V(u(t+s))ds 
                 = \int_{-\tau^*-\tilde{\tau}(t)}^{0} \hspace*{-2em}V(u^{*\!}+\tilde{u}(t+s)) ds \\
	&=\int_{-\tau^*-\tilde{\tau}(t)}^{0} \hspace*{-2em} V(u^*)+V'(u^*)\tilde{u}(t+s) + \cO(\|\tilde{u}(t+s)\|^2) ds \\
	&= V(u^*)(\tau^{*\!}+\tilde{\tau}(t))+V'(u^*)\!\int_{-\tau^*-\tilde{\tau}(t)}^{0}\hspace*{-1em}\tilde{u}(t+s) ds + \cO(\|\tilde{u}_t\|^2).
\end{align*}
But $a=V(u^*)\tau^*$ so
\begin{align*}	
0
&= V^{\!}(u^*)\tilde{\tau}(t) + V'^{\!}(u^*)\!\!\int_{-\tau^*}^{0}\hspace*{-1em}\tilde{u}(t+s) ds + \cO(\|\tilde{u}_t\|\|\tilde{\tau}(t)\|)+ \cO(\|\tilde{u}_t\|^2).
\end{align*}
Thus 
\begin{equation} \label{eq:tautilde}
	\tilde{\tau}(t) = -\frac{V'(u^*)}{V(u^*)}\int_{-\tau^*}^{0}\tilde{u}(t+s) ds + h.o.t.
\end{equation} 

Expanding \eqref{eq:thre_scalarDE} using \eqref{eq:tildeu_tildetau} we obtain 
\begin{align} \label{eq:dtildeudt}
	\frac{d\tilde{u}(t)}{dt} &= \beta e^{-\mu\tau^*}\Big[g(u^*)\frac{V'(u^*)}{V(u^*)}(\tilde{u}(t)-\tilde{u}(t-\tau^*))\nonumber\\
	&\quad\quad -g(u^*)\mu\tilde{\tau}(t)+g'(u^*)\tilde{u}(t-\tau^*)\Big]-\gamma\tilde{u}(t)+h.o.t.
\end{align}

Now combining \eqref{eq:dtildeudt} and \eqref{eq:tautilde} and neglecting the higher order terms, we get the linearised system
\begin{align}
	\frac{d\tilde{u}(t)}{dt} & = \beta e^{-\mu\tau^*}\Big[g(u^*)\frac{V'(u^*)}{V(u^*)}(\tilde{u}(t)-\tilde{u}(t-\tau^*)) \nonumber\\
	& \quad\quad -g(u^*)\mu\tilde{\tau}(t)+g'(u^*)\tilde{u}(t-\tau^*)\Big]-\gamma\tilde{u}(t), \label{eq:dtildeudt_lin} \\
	\tilde{\tau}(t) &= -\frac{V'(u^*)}{V(u^*)}\int_{-\tau^*}^{0}\tilde{u}(t+s) ds.  \label{eq:tautilde_lin}
\end{align}

To solve the linear system  \eqref{eq:dtildeudt_lin}-\eqref{eq:tautilde_lin} make the ansatz that
the solution $(\tilde{u}(t), \tilde{\tau}(t))$ can be expressed as
\begin{equation} \label{eq:tildeu_tildetau_exp}
	\begin{pmatrix}
		\tilde{u}(t) \\ \tilde{\tau}(t)
	\end{pmatrix} = e^{\lambda t} \begin{pmatrix}
		\epsilon_1 \\ \epsilon_2
\end{pmatrix}. 
\end{equation}
Substituting \eqref{eq:tildeu_tildetau_exp} into \eqref{eq:tautilde_lin} yields
\begin{equation} \label{eq:eps2}
	\epsilon_2 = -\frac{V'(u^*)}{V(u^*)}\frac{1}{\lambda}(1-e^{-\lambda\tau^*})\epsilon_1.
\end{equation}
Then substituting \eqref{eq:tildeu_tildetau_exp} into \eqref{eq:dtildeudt_lin} and applying \eqref{eq:eps2}, we get the characteristic equation is again
$\Delta_{\textit{thres}}(\lambda)=0$ where $\Delta_{\textit{thres}}(\lambda)$ is defined by
\eqref{eq:char_scalarDE}.
The same characteristic equation is also derived using the Banach space approach in \cite{gedeon2024dynamics}.

We remark that this characteristic equation cannot be derived by freezing the delay. Freezing the delay essentially treats \eqref{eq:scalarDE} as a constant discrete delay DDE, which would have an exponential polynomial as its characteristic function, but $\Delta_{\textit{thres}}$ is not an exponential polynomial because of the $\mu/\lambda$ term (which arises from the integration over the time interval in the derivation of \eqref{eq:expo_linapp}).

\CCLsubsubsection{Linearisation with differentiation of the threshold condition}

In Section~\ref{subsec:scalthresdde}
we proposed to replace the threshold condition \eqref{eq:thre_scalarDE} by its differentiated form 
\eqref{eq:taudot} to convert the threshold delay system \eqref{eq:scalarDE},\eqref{eq:thre_scalarDE}
into the discrete delay system \eqref{eq:thre_scalarDE},\eqref{eq:taudot}. 
To linearise \eqref{eq:taudot} around a steady state, substitute \eqref{eq:tildeu_tildetau} into \eqref{eq:taudot} to obtain
\begin{equation} \label{eq:dtildetaudt}
	\frac{d\tilde{\tau}(t)}{dt} = -\frac{V'(u^*)}{V(u^*)}(\tilde{u}(t)-\tilde{u}(t-\tau^*))+h.o.t. 
\end{equation}
The linearisation
of \eqref{eq:thre_scalarDE},\eqref{eq:taudot} is then defined by
the linear system \eqref{eq:dtildeudt},\eqref{eq:dtildetaudt}. Making the solution ansatz
\eqref{eq:tildeu_tildetau_exp}, we obtain
\begin{equation} \label{eq:diffthreschareq}
	\begin{pmatrix}
		A_{11} & -\beta e^{-\mu\tau^*}g(u^*)\mu \\
		\frac{V'^{\!}(u^*)}{V(u^*)}(1-e^{-\lambda\tau^*}) & \lambda
	\end{pmatrix} \! \begin{pmatrix}
	\epsilon_1 \\ \epsilon_2
	\end{pmatrix}=\begin{pmatrix}
		0 \\ 0
	\end{pmatrix}
\end{equation}
where 
$$A_{11}=
\lambda+\gamma-\beta e^{-\mu\tau^*\!}\Big(g(u^*)\frac{V'^{\!}(u^*)}{V(u^*)}(1-e^{-\lambda\tau^*})+g'^{\!}(u^*)e^{-\lambda\tau^*\!}\Big).$$
It follows that the characteristic equation to \eqref{eq:dtildeudt},\eqref{eq:dtildetaudt} is given by $\Delta(\lambda)=0$, where $\Delta(\lambda)=\lambda\Delta_{\textit{thres}}(\lambda)$, where
$\Delta_{\textit{thres}}(\lambda)$, defined in \eqref{eq:char_scalarDE}, is the characteristic function of the undifferentiated threshold problem
\eqref{eq:scalarDE},\eqref{eq:thre_scalarDE}. 

Thus we see that differentiating the threshold condition preserves all of the characteristic values of the original system, but introduces one extra spurious characteristic value $\lambda=0$. That this arises directly from the differentiation can be seen by multiplying out the second line of \eqref{eq:diffthreschareq}, which gives precisely \eqref{eq:eps2} multiplied by $\lambda$.

\CCLsection{Numerical Techniques for DDEs}

\CCLsubsection{IVPs with Discrete State-Dependent Delays \label{sec:bps}}

There are two main issues for numerically solving the IVP 
\begin{align} \label{eq:num:dde}
\udot(t)&=f(t,u(t),u(t-\tau(t,u(t)))), \quad  t\geq t_0,\\
u_{t_0} &= \phi \in C. \label{eq:num:phi}
\end{align}
These are the need to compute $u(t-\tau)$ between mesh points during the numerical computation, and also maintaining the global approximation order of the method while coping with the loss of smoothness of the solution at breaking points. One-step Runge-Kutta (RK) methods are preferred, as multistep methods are derived from interpolating over several contiguous steps and will suffer an order reduction due to the breaking points. 

When all of the delays are discrete and bounded away from zero the method of steps (introduced above equation~\eqref{eq:Winstonu}) can be used to solve the DDE as a sequence of nonautonomous ODE problems. 
To evaluate delayed terms $u(t-\tau)$ where $t-\tau\leq t_n$ while computing the solution on the time-step 
$t\in[t_n,t_{n+1}]$, so called Continuous Runge-Kutta methods (CRKs) are used.  Originally formulated for ODEs, these methods define a piecewise smooth interpolant function, $\uh(t)$, of the same global order as the ODE method, 
so that the numerical solution is defined for all $t\in[t_0,t_n]$. There is an excellent description of CRK methods for constant delay DDEs in \cite{BZ03}. 
CRK methods generally require that the step-size $h_n=t_{n+1}-t_n$ satisfies $h_n<\tau$, otherwise \emph{overlapping} occurs
where a value of $u(t-\tau)$ is required from the current as yet incomplete step. In this situation CRK methods become implicit (see \cite{BZ03}). 

\CCLsubsubsection{Explicit Functional Continuous Runge-Kutta (FCRK) Methods}

These
were introduced by \cite{Tavernini71}, while general FCRK convergence theory has been developed more recently (see \cite{MTV05,BZMGActa09}). 
FCRKs solve the overlapping problem by defining an interpolant for each stage of the method
based only on the already computed stages, so FCRKs remain explicit, even in the case of overlapping. 
Suppose
$n$ steps of the solution have already been computed, then at the next step we need to solve
\eqref{eq:num:dde} for $t\in[t_n,t_{n+1}]$ with
 \be \label{eq:mainnthint} 
  u(t) = \left\{
  \begin{aligned}
    &\phi(t), &&\textrm{if}\;\;t\leq t_0,\\
    & \uh(t),&&\textrm{if}\;\;t\in[t_0,t_n].
  \end{aligned} \right.
\ee

The general $s$-stage explicit FCRK method for obtaining the approximation $\uh(t)$ to $u(t)$ for $t\in(t_n,t_{n+1}]$ is defined by
\begin{equation}\left.
  \begin{gathered}
    \uh(t_{n}+\theta h_n) = \uh(t_n) + h_n\sum_{i=1}^sb_i(\theta)k_{n,i},\quad \theta=\frac{t-t_n}{h_n}\in(0,1],\\
      U_{n,i} = \uh(t_n) + h_n\sum_{j=1}^{i-1}a_{ij}k_{n,j}, \\
    k_{n,i} = f\big(t_n + c_ih_n, U_{n,i}, \eta_{n,i}(t-\tau(t_n + c_ih_n,  U_{n,i}))\big).
  \end{gathered}\right\} \label{eq:fcrk}
\end{equation}
The stage interpolants $\eta_{n,i}(t)$ are defined by \eqref{eq:mainnthint} when $t\leq t_n$,
and when $t_n<t$ (overlapping) by
\begin{equation} \label{eq:fcrketa}
  \eta_{n,i}(t) = \uh(t_n) + h_n\sum_{j=1}^{i-1} a_{ij}(\theta)k_{n,j}.
\end{equation}
The interpolant $\uh(t_n+\theta h_n)$ in \eqref{eq:fcrk} allows us to extend $\uh$ in \eqref{eq:mainnthint} 
to the interval $t\in[t_0,t_{n+1}]$, while equation~\eqref{eq:fcrketa} defines explicit interpolants
which are only used in the computation of the $i$-th stage value $k_{n,i}$ if 
there is overlapping (if $t-\tau(t_n + c_ih_n,  U_{n,i})>t_n$).

We will only consider methods with $a_{ij}=a_{ij}(c_i)$, and then the coefficients of particular methods can be represented by the Butcher tableau
\begin{equation} \label{tbl:fcrk}
\begin{array}{c|ccccc}
     0     \\
     c_2    & a_{21}(\theta)  \\
     c_3    & a_{31}(\theta) & a_{32}(\theta)  \\
     \vdots & \vdots & \vdots & \ddots   \\
     c_s    & a_{s1}(\theta) & a_{s2}(\theta) & \cdots & a_{s-1,s}(\theta)  \\[0.3mm] \hline \\[-4mm]
              & b_1(\theta) & b_2(\theta) & b_3(\theta) & \cdots & b_s(\theta).
\end{array}
\end{equation}
We will use four explicit methods. FCRK1 and FCRK2 are $s$-stage methods of global (uniform) order $s$
proposed by \cite{Tavernini71} defined by 
\begin{equation} \label{tbl.fcrk12}
  \begin{array}{c|c}
    0        &    \\
     \hline & \\[-3.5mm]
      & \theta
  \end{array}
\qquad
  \begin{array}{c|cc}
    0        & &   \\
    1        & \theta & \\
     \hline && \\[-3.5mm]
      &  \theta-\frac12\theta^2 & \frac12\theta^2 \\[1mm]
  \end{array}
\end{equation}
FCRK3 and FCRK4 are methods of global order $3$ and $4$ with respectively $4$ and $7$ stages,
which were first derived in
%
\cite{MTV05}, where their coefficients are stated. See \cite{BZMGActa09} for a very detailed discussion of how the global order is derived from the properties of the interpolants and the underlying RK method.
See \cite{Alexey19} for explicit FCRK methods of global order 3 and 4 with 3 and 6 new
stages per step respectively. 

\CCLsubsubsection{Breaking Points}

We introduced breaking points in Section~\ref{sec:ddes} and discussed them further in
Section~\ref{sec:dynsys}. Now we need to distinguish their severity.

\begin{definition}[\cite{BZ03}]   \label{def:bkpt}
A \textit{breaking point} is said to be of order $k$ if all derivatives of the solution up to the $k$-th exist and the $k$-th is Lipschitz continuous.
\end{definition}

With this definition, $t_0$ is generically a breaking point of order $0$.
If the initial function $\phi$ has jump discontinuities (or if the initial condition is set such that $u(t_0)\ne\phi(0)$) they are breaking points of order $-1$. These breaking points are not an issue 
as they are not in the computational interval for $t\geq t_0$. However, 
if $t=\xi_i$ is a breaking point then from 
\eqref{eq:num:dde}, additional breaking points are created at any $t=\xi_j>\xi_i$ such that
\be \label{eq:xij}
\xi_j-\tau(\xi_j,u(\xi_j))=\xi_i.
\ee
For simplicity we assume that $f$ is smooth (otherwise $f$ itself may generate additional breaking points),
then as already illustrated in \eqref{eq:ds:uddbkpt}, the smoothing property of the DDE
means that a breaking point $\xi_j$ defined by \eqref{eq:xij} will have order (at least) one more than the order of $\xi_i$.

Reduced differentiability at breaking points can pollute the local and global error of numerical solvers.
Since order $p$ IVP solvers achieve their local error $\cO(h^{p+1})$ by approximating the Taylor series of the solution over each time-step, breaking points of order $p$ and higher, will not affect the global order of the numerical computation. From the smoothing property, if there is a maximum delay $\tau(t,u(t))\leq \tau_{\textit{max}}$ then for a method of order $p$ we will only be concerned 
with breaking points of order $k<p$ which will occur for $t\in(t_0,t_0+p\tau_{\textit{max}})$, and so only finitely many breaking points need to be tracked and approximated. As for how accurately they need to be approximated: 

\begin{theorem} \label{thm:approx}
Let $\xi(h)$ be the nearest mesh point $t_n$ to a breaking point $\xi_j$ of order $k$. Then a necessary condition for an FCRK method to have global order $p>k$ is that
\be \label{eq:hpk+1}
\xi(h)-\xi_j = \cO(h^{p/{(k+1)}}).
\ee
\end{theorem}

\begin{proof}
In the absence of a breaking point the FCRK method (and indeed any one-step method) approximates the Taylor series of the smooth solution for $t\in[t_n,t_{n+1}]$ to obtain its local accuracy.
When there is a breaking point of order $k$ at $\xi_j\in(t_n,t_{n+1})$ the Taylor series of the exact solution
$u_-(t)=\sum_{j=0}^\infty\frac{1}{j!}f^{(n)}(\xi_j)(t-\xi_j)^n$ for $t\leq \xi_j$ will differ from
the Taylor series of the exact solution $u_+(t)$ for $t\geq\xi_j$ in the terms of order $(t-\xi_j)^{k+1}$ and above. For the smooth numerical solution defined by the FCRK method on $t\in[t_n,t_{n+1}]$ to have global order $p$, it must approximate both $u_-(t)$ and $u_+(t)$ to order $\cO(h^p)$ accuracy, which is only possible if 
$\sup_{t\in[t_n,t_{n+1}]}\|u_-(t)-u_+(t)\|=\cO(h^p)$ which leads directly to \eqref{eq:hpk+1}.
\end{proof}

\begin{table}[t!]
\begin{center}
\begin{tabular}{|c|c|c|c|c|c|c|}\hline
Breaking point & \multicolumn{6}{c|}{Order of method $p$} \\
Order $k$ & 1 & 2 & 3 & 4 & 5 & 6 \\ \hline
0 &  (1) & 2 & 3 & 4 & 5 & 6 \\
1 & - & (1) & 1.5 & 2 & 2.5 & 3 \\
2 & - & - & (1) & 4/3 & 5/3 & 2 \\
3 & - & - & - & (1) & 1.25 & 1.5 \\
4 & - & - & - & - & (1) & 1.2 \\ \hline
\end{tabular}
\caption{The smallest value of $r$ to achieve global order $p$ when breaking points of order $k$ are approximated with accuracy $\cO(h^r)$ in \eqref{eq:hpk+1}.}  \label{tab:Ohr}  
\end{center} 
\end{table} 

Table~\ref{tab:Ohr} states the lowest accuracy $\cO(h^r)$ to which breaking points of order $k$ need to be approximated to achieve global order $p$. We already saw that points of order$\geq p$ are immaterial, and now we see that breaking points of order $p-1$ do not need to be approximated either. More precisely, they need to be approximated to accuracy $\cO(h)$, which the mesh already does.     
Order $0$ breaking points need to be approximated very accurately; to $\cO(h^p)$ to obtain order $p$. But higher order breaking points do not need to be found quite so accurately; for example finding the order 1 and 2 breaking points to just $\cO(h^2)$ accuracy is necessary for fourth order convergence. 

We will present an explicit (non-iterative) method for detecting breaking points to the accuracy required by 
\eqref{eq:hpk+1} to achieve global order $p$ for FCRK methods. Lacking space here for a full convergence proof, we will use two test problems with known exact solutions to illustrate the performance of the different FCRK methods with and without breaking point detection.

\begin{example}[\cite{FN84}] \label{ex:test1}
Consider the DDE IVP
\begin{equation} \label{eq:test1}
\udot(t) = \frac{u(u(t) - \sqrt{2}+1)}{2\sqrt{t}}, \; t\in[1,5], \qquad u(t) = 1,\; t \leq 1,
\end{equation}
which has solution
\begin{equation} \label{eq:test1sol}
u(t) = \left\{
  \begin{aligned}
    &\sqrt{t},  && t\in[1,2], \\
    &\frac{t}4 + \frac12 +\left(1 - \frac1{\sqrt{2}}\right)\sqrt{t}, && t\in[2,5.0294372515248].
  \end{aligned}
  \right.
\end{equation}
The solution $u(t)$ of \eqref{eq:test1} does not link smoothly with the history function at the initial point $t_0 = 1$, thus there is a breaking point of order $0$ at $\xi_1=1$. From \eqref{eq:xij} this leads
to a breaking point $\xi_2=2$ of order $1$ (where $\alpha(\xi_2,u(\xi_2))=u(\xi_2)-\sqrt{2}+1=\xi_1$). The next breaking point $\xi_3>5$ has order $2$. We will solve this IVP on the interval $t\in[t_0,t_{\!f}]=[1,5]$ which contains the single breaking point $\xi_2=2$ of order $1$.
\end{example}

\begin{example}  \label{ex:test2}
Consider the DDE
\begin{equation}  \label{eq:test2}
\udot(t) = -u(t) - u(t - 1 - u(t)), \qquad
u(t)  = \phi(t)= \left\{
    \begin{array}{cl}
      0,\; & t < -1,\\
      1, & t\in[-1,0],
    \end{array} \right.
\end{equation}
which has the solution
\begin{equation} \label{eq:test2sol}
u(t) =
\left\{
  \begin{aligned}
    &\phantom{-}\rme^{-t},  && t\in[0,\W(1)], \\
    &-1 + \rme^{-t}\bigl(1+\rme^{\W(1)}\bigr), && t\in[\W(1),\W(1+\rme^{\W(1)})].
  \end{aligned}
  \right.
\end{equation}
Here $\W(t)$ is Lambert $W$-function, which is defined as a solution of $\W(t)\rme^{\W(t)}=t$
(see \cite{Corless1996} for a detailed treatment).
The DDE \eqref{eq:test2} is a special case of \eqref{eq:twostatedep}.
We consider \eqref{eq:test2} as an IVP with
a discontinuity in the history function at $\xi_1 = -1$, which defines a breaking point of order $-1$.

As in the previous example, the solution is not smoothly linked with the history function at $t_0 = 0$,
so there is a second breaking point $\xi_2=0$, which is of order $0$. The discontinuity in $\phi$ gives rise to 
a second 0-order breaking point
$\xi_3 = \W(1)\approx 0.567143$. We solve the IVP for $t\in[t_0,t_{\!f}]=[0,1]$, for which $\xi_3$ is the only breaking point
in the computation interval.
\end{example}

Example~\ref{ex:test2} falls outside of the existence and uniqueness theory of Theorems~\ref{thm:Driverexist}-\ref{thm:Driveruniq} because $\phi$ is not continuous. Problems with piecewise smooth initial functions $\phi$ arise often in applications. For example in hematopoietic models, a blood donation, or an intravenous bolus drug administration, can be modelled as an instantaneous change in the relevant variables. Equation~\eqref{eq:test2} nevertheless has the unique solution \eqref{eq:test2sol}, despite the discontinuity in $\phi$. There has been a lot of study of
piecewise smooth systems in recent years, but mainly focused on ODEs. Clearly, the only issue will be when 
$t=\xi_j$ solves \eqref{eq:xij}. In that case, provided
\be \label{eq:nosliding}
\frac{d}{dt}\big(t-\tau(t,u^\pm(t))\big)\Big|_{t=\xi_j}>0
\ee
for both the solutions $u^-(t)$ for $t\leq \xi_j$ and $u^+(t)$ for $t\geq \xi_j$ the solution can be uniquely continued through this point. The solution \eqref{eq:test2sol} is easily verified to satisfy this condition.
We are not aware of a theoretical result establishing \eqref{eq:nosliding} rigorously, but elements of the discussion can be found in \cite{BG09}.

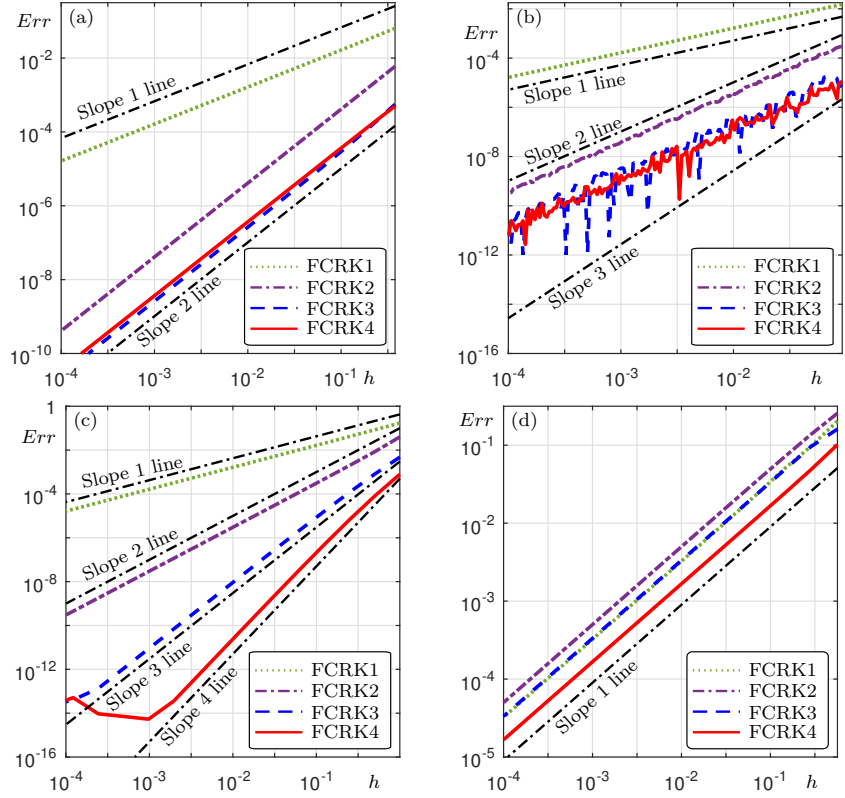
\begin{figure}[tp!]
\centering
\scalebox{0.9}{\includegraphics{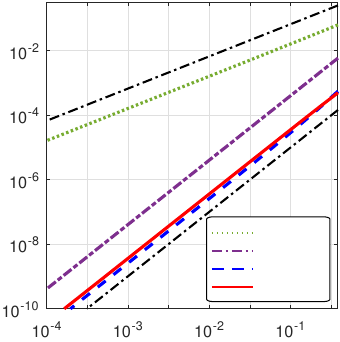}
\put(-38,50){\scriptsize FCRK1}
\put(-38,41){\scriptsize FCRK2}
\put(-38,32){\scriptsize FCRK3}
\put(-38,24){\scriptsize FCRK4}
\put(-135,114){\rotatebox{20}{\scriptsize Slope 1 line}}
\put(-110,18){\rotatebox{40}{\scriptsize Slope 2 line}}
\put(-14,3){\scriptsize$h$}
\put(-160,153){\scriptsize$Err$}
\put(-138,155){\scriptsize (a)}}\hspace*{2em}\scalebox{0.9}{\includegraphics{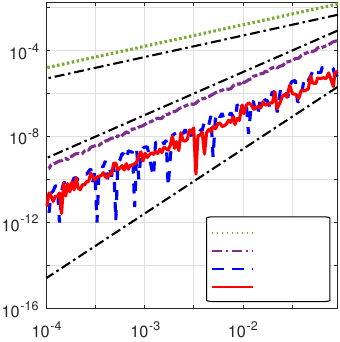}
\put(-38,50){\scriptsize FCRK1}
\put(-38,41){\scriptsize FCRK2}
\put(-38,32){\scriptsize FCRK3}
\put(-38,24){\scriptsize FCRK4}
\put(-135,121){\rotatebox{11}{\scriptsize Slope 1 line}}
\put(-135,95){\rotatebox{22}{\scriptsize Slope 2 line}}
\put(-125,34){\rotatebox{33}{\scriptsize Slope 3 line}}
\put(-14,3){\scriptsize$h$}
\put(-160,148){\scriptsize$Err$}
\put(-138,155){\scriptsize (b)}}

\scalebox{0.9}{\includegraphics{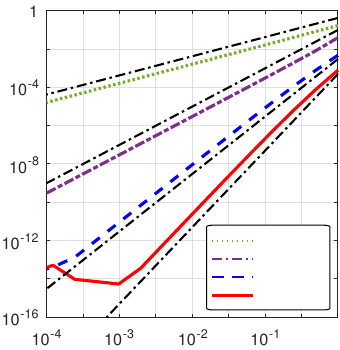}
\put(-38,50){\scriptsize FCRK1}
\put(-38,41){\scriptsize FCRK2}
\put(-38,32){\scriptsize FCRK3}
\put(-38,24){\scriptsize FCRK4}
\put(-135,129){\rotatebox{12}{\scriptsize Slope 1 line}}
\put(-125,35){\rotatebox{38}{\scriptsize Slope 3 line}}
\put(-135,90){\rotatebox{25}{\scriptsize Slope 2 line}}
\put(-102,20){\rotatebox{47}{\scriptsize Slope 4 line}}
\put(-14,3){\scriptsize $h$}
\put(-160,148){\scriptsize $Err$}
\put(-138,155){\scriptsize (c)}}\hspace*{2em}\scalebox{0.9}{\includegraphics{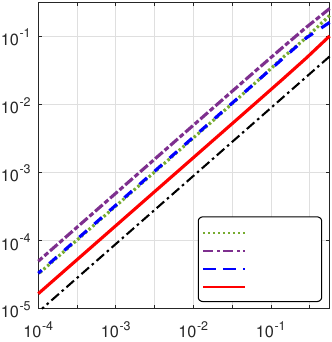}
\put(-38,50){\scriptsize FCRK1}
\put(-38,41){\scriptsize FCRK2}
\put(-38,32){\scriptsize FCRK3}
\put(-38,24){\scriptsize FCRK4}
\put(-120,25){\rotatebox{42}{\scriptsize Slope 1 line}}
\put(-14,3){\scriptsize $h$}
\put(-160,155){\scriptsize $Err$}
\put(-138,155){\scriptsize (d)}}
\caption{The error
$Err$ \eqref{eq:Err} for the four methods FCRK1-4 without breaking point detection and step-sizes chosen using \eqref{eq:hlambda}. 
For Problem \eqref{eq:test1} on $t\in[1,5]$ with (a) $\lambda=1/2$, (b) 
$\lambda$ random  and uniformly distributed, (c) $\lambda=0$. (d) For problem \eqref{eq:test2} on $t\in[0,1]$
with $\lambda=0$.
Reference lines of different slopes are included to indicate the global order of the methods.} \label{fig:testnone}
\end{figure}

Figure~\ref{fig:testnone} illustrates the behaviour of the four methods FCRK1-4 applied to the 
test problems without any breaking point detection. 
To ensure the first breaking point $\xi>t_0$ is not accidentally 
or randomly included in the computational mesh we integrate with constant step-sizes $h$ defined by
\be \label{eq:hlambda}
\xi-t_0=(N+\lambda)h
\ee
for integer $N$ and constant $\lambda\in[0,1)$ so that the breaking point always falls at the same fraction $\lambda$ in the $(N+1)$-st step, for any choice of the integer $N$. The error between
the numerical and exact solutions as functions is measured using the $L^\infty$-norm
\be \label{eq:Err}
Err=\sup_{t\in[t_0,t_{\!f}]}\|u(t)-\uh(t)\|.
\ee

For the test problem \eqref{eq:test1}, in Figure~\ref{fig:testnone}(a) step-sizes are chosen so that $\lambda=1/2$ in 
\eqref{eq:hlambda} and the breaking point always falls exactly in the middle of a step.
The figure shows that when an order $1$ breaking point is not approximated in the computational mesh, all higher order FCRK methods are reduced to second order. 
Noting that with step-size $h$ all breaking points are approximated to at least $\cO(h)$ accuracy,
this behaviour is fully consistent with Theorem~\ref{thm:approx}.

Figure~\ref{fig:testnone}(c) shows that when the order $1$ breaking point is included in the mesh for 
the test problem \eqref{eq:test1}, the FCRK$p$ method recovers it full order $p$. A more realistic scenario,
without breaking point detection when the exact solution is not known, is that the breaking point will occur  
at a random point in the $[t_n,t_{n+1}]$. This is simulated in Figure~\ref{fig:testnone}(b) by choosing 
uniformly distributed $\lambda$ in \eqref{eq:hlambda}, which again shows that higher order methods are 
reduced to second order. 
 
Figure~\ref{fig:testnone}(d) shows the behaviour of the FCRK methods applied to the test problem \eqref{eq:test2} with an order $0$ breaking point. When the breaking point is not included in the mesh, all the FCRK methods are reduced to first order, as expected from Theorem~\ref{thm:approx}. However, Figure~\ref{fig:testnone}(d) actually shows the case where the exact breaking points are included in the mesh ($\lambda=0$),
where we see that the methods still all behave as order $1$.




Several issues are at play in the example of Figure~\ref{fig:testnone}(d).
Firstly, because breaking point locations depend on the solution, the breaking points $\xi_j^h$ of the numerical solution $\uh(t)$ and the breaking points $\xi_j$ of the exact solution $u(t)$ will not correspond exactly, and it is  $\xi_j^h$, and not $\xi_j$ that should be included in the computational mesh. 
Secondly even if $\xi_j^h-\xi_j=\cO(h^p)$, for an order $0$ breaking point $\xi^h_j$,
there will be an $\cO(1)$ difference between $f(t,\uh(t),t-\tau(t,\uh(t)))$ for $t$ values close to $\xi^h_j$ such that $t-\tau(t,\uh(t))$ falls either side of the previous order $-1$ breaking point. In this case an arbitrary small error in determining the breaking point can generate a $\cO(1)$ error in an evaluation of $f$.
Thus, it is not sufficient just to detect breaking points to accuracy \eqref{eq:hpk+1}, but how these breaking points are handled in the evaluation of the numerical solution will be crucial for maintaining the order of the method. In particular, it will be necessary to ensure that function values are never evaluated from the wrong side of a previous breaking point.

While several authors have tackled breaking points for state-dependent delays, none have 
supplied a comprehensive algorithm for explicit numerical solvers that both preserves the 
global order of the method, and keeps the method explicit. 
\cite{BZ03} while stating many theoretical results for CRKs 
simply assume that all breaking points are known exactly.

\cite{FN84} present a detailed study of breaking points, but do not present an algorithm for finding them.
They suggest to detect if a new breaking point $\xi_j^h$ appears after making a step by checking whether $t-\tau(t,\uh(t))$ crosses a previously found breaking point $\xi_i^h$.
If that is the case, then the interpolant $\uh$ of the solution determines where the new breaking point is.
However, as illustrated in Figure~\ref{fig:bp3} this provides a poor approximation, and risks leaving us in the same scenario as in Figure~\ref{fig:testnone}(d).

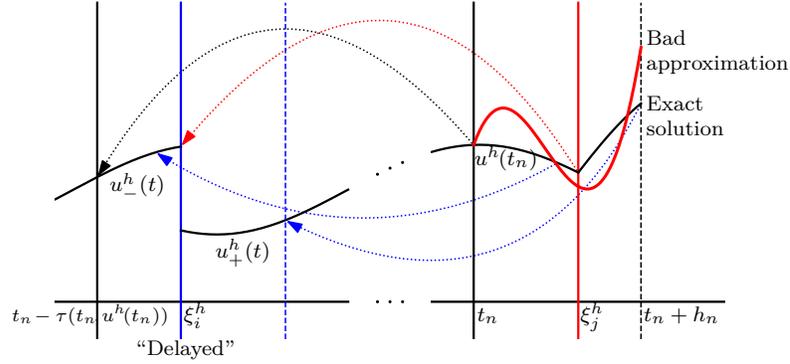
\begin{figure}[t!]
\centering
\includegraphics[scale=0.565]{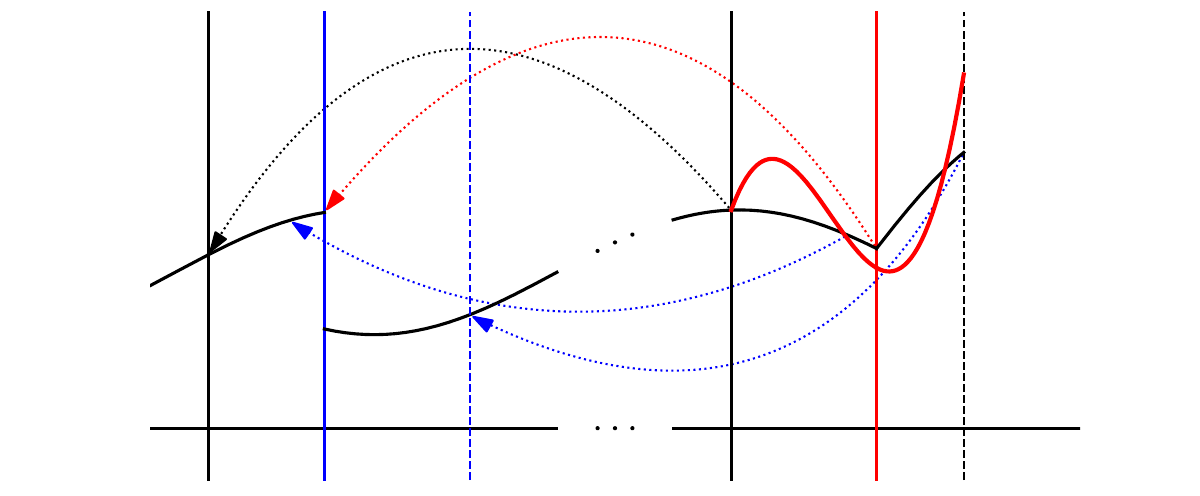}
    \put (-126,11){\scriptsize $t_n$}
    \put (-63,11) {\scriptsize $t_n+h_n$}
    \put (-127,70) {\scriptsize $\uh(t_n)$}
    \put (-302,11) {\tiny $t_{n\!}-\tau(t_{n\!},^{\!}u^{h\!}(t_n))$}
    \put (-237,11) {\scriptsize $\xi_i^h$}
    \put (-87,11) {\scriptsize $\xi_j^h$}
    \put (-265,60) {\scriptsize $\uh_-(t)$}
    \put (-225,35) {\scriptsize $\uh_+(t)$}
    \put (-255,-2) {\scriptsize ``Delayed''}
    \put (-62,86) {\parbox{2cm}{\scriptsize Exact\\  solution}}
    \put (-62,110) {\parbox{2cm}{\scriptsize Bad\\ approximation}}
\caption{An example of a bad approximation over a breaking point of order $0$.} \label{fig:bp3} 
\end{figure}

The problem with this approach is that $\uh(t)$ defines a smooth approximation to the solution $u(t)$ of \eqref{eq:num:dde} for $t\in[t_n,t_{n+1}]$, but $u(t)$ is not itself smooth because of the breaking point 
at $\xi_j^h\in(t_n,t_{n+1})$. The breaking point $\xi_j^h$ is defined from $\xi_i^h$ via \eqref{eq:xij}
and for $t\in[t_n,t_{n+1}]$ parts of the delay term $u(t-\tau(t,u(t)))$ will evaluate to 
$\uhm(t)$ and other parts to $\uhp(t)$, either side of the previous braking point $\xi_i^h$ (as illustrated in Figure~\ref{fig:bp3}).  
The term $u(t-\tau(t,u(t)))$ then lacks sufficient smoothness for the numerical solution to have its full order over the step $[t_n,t_{n+1}]$, and so it is not possible to approximate $\xi_j$ to the full accuracy of the method.

\cite{GH08} do successfully use such reduced order approximations to find breaking points for implicit stiff CRK methods. Since the stage-values of the implicit method need to be found iteratively, they add 
$t_{n+1}-\alpha(t_{n+1},\uh(t_{n+1}))=\xi_i^h$ as an extra equation along with the $s$ stage equations and solve all $s+1$ equations together to find the stage values $k_{n,j}$ along with the breaking point $t_{n+1}$, which is now conveniently at the end of the step. Since their CRK methods are already implicit, the extra cost of this approach is negligible for them.

The approach of \cite{GH08} makes all methods implicit, and so is 
unsuitable for explicit FCRK methods. Here we present
an efficient
method to find breaking points to the accuracy required by \eqref{eq:hpk+1} while 
keeping the method explicit 
(so no implicit equations to solve iteratively) and achieve global order $p$ for FCRK methods. A preliminary version of this algorithm was first reported in \cite{EH15}.

\CCLsubsubsection{Breaking Point Detection Details} 

Rather than giving a formal algorithm statement, we will explain the concepts and steps involved.
Consider \eqref{eq:num:dde} for $t\in[t_n,t_{n+1}]$ when a single breaking point $\xi_i^h$
appears in the delayed term $u(t-\tau)$ during the time-step. Hence we need to solve 
\begin{equation} \label{eq:bp}
\left.\begin{aligned}
\udot(t)&=f(t,u(t),u(t-\tau)), &&  t\in[t_n, t_{n+1}],\\
u(t_n)&=\uh(t_n),\\
u(t-\tau)&=\uhm(t-\tau), && t-\tau \leq \xi_i^h,\\
u(t-\tau)&=\uhp(t-\tau), && t-\tau \geq \xi_i^h.
\end{aligned}\right\}
\end{equation}
As already illustrated in Figure~\ref{fig:bp3}, the lack of smoothness of $u(t-\tau)$ will cause a loss of accuracy for the IVP solver over the step $t\in[t_n,t_{n+1}]$, preventing 
an accurate approximation of the new breaking point $\xi_j^h$. 

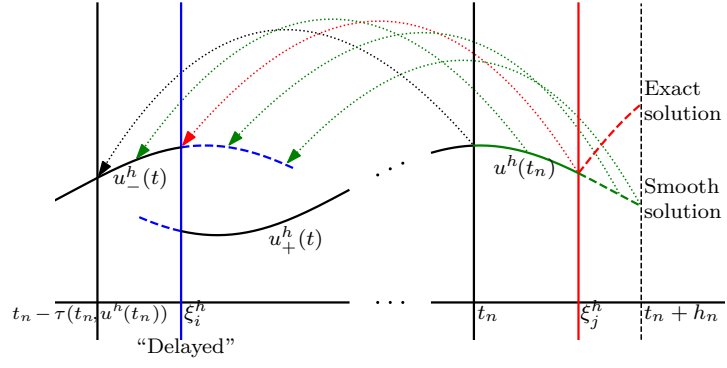
\begin{figure}[t!]
\centering
\includegraphics[scale=0.565]{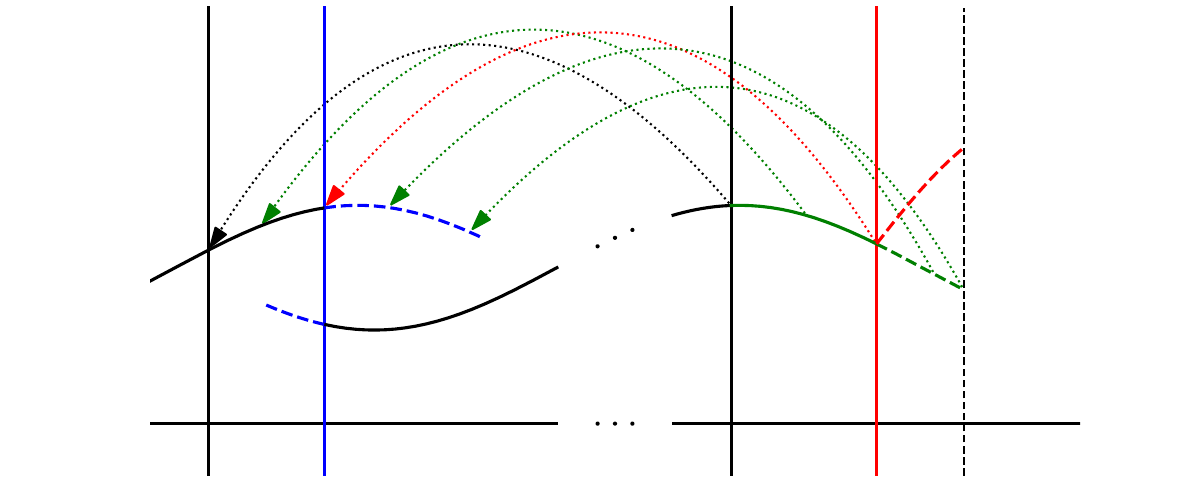}
    \put (-126,11) {\scriptsize $t_n$}
    \put (-63,11) {\scriptsize $t_n+h_n$}
    \put (-120,66) {\scriptsize $\uh(t_n)$}
    \put (-302,11) {\tiny $t_{n\!}-^{\!}\tau(t_{n\!},^{\!}\uh(t_n))$}
    \put (-237,11) {\scriptsize $\xi_i^h$}
    \put (-87,11) {\scriptsize $\xi_j^h$}
    \put (-264,62) {\scriptsize $\uh_-(t)$}
    \put (-205,39) {\scriptsize $\uh_+(t)$}
    \put (-255,-2) {\scriptsize ``Delayed''}
    \put (-63,91) {\parbox{2cm}{\scriptsize Exact\\ solution}}
    \put (-63,54) {\parbox{2cm}{\scriptsize Smooth\\ solution}}
\caption{Accurate breaking point computation in the modified smooth problem \eqref{eq:bpsmooth}.} \label{fig:bp5} 
\end{figure}

To resolve this issue, our algorithm considers the modified problem
\begin{equation} \label{eq:bpsmooth}
\left.\begin{aligned}
\udot(t)&=f(t,u(t),u(t-\tau)), &&  t\in[t_n, t_{n+1}],\\
u(t_n)&=\uh(t_n),\\
u(t-\tau)&=\uhm(t-\tau),
\end{aligned}\right\}
\end{equation}
shown in Figure~\ref{fig:bp5}. In \eqref{eq:bpsmooth} we analytically continue $\uhm(t-\tau)$ across the breaking point $\xi_i^h$ to 
create a modified problem without a breaking point.
The new problem \eqref{eq:bpsmooth} is smooth over the step $t\in[t_n,t_{n+1}]$ and has
an identical solution to \eqref{eq:bp} for $t\in[t_n,\xi_j^h]$. Thus we solve \eqref{eq:bpsmooth} 
for $t\in[t_n,t_{n+1}]$, and use this solution to approximate the new breaking point $\xi_j^h$ to the full accuracy of the FCRK method. Once $\xi_j^h$ is found, 
the solution of \eqref{eq:bpsmooth} for $t\in[t_n,\xi_j]$ can be used to define $\uh(t)$ for $t\in[t_n,\xi_j^h],$
and we proceed to the next step defining $t_{n+1}:=\xi_j^h$.

To turn this idea into an efficient explicit algorithm we still need to be able to use the FCRK solution of 
\eqref{eq:bpsmooth} to solve \eqref{eq:xij} to find $\xi_j^h$ to the accuracy required by \eqref{eq:hpk+1}.
To do this we compute every step $t\in[t_n,t_{n+1}]$ of the solution of \eqref{eq:num:dde} by first solving
\eqref{eq:bpsmooth}, using the smooth delayed term $\uhm(t-\tau)$. If none of the previous
breaking points $\xi_i^h$ satisfy
\be \label{eq:bkptneeded}
(t_{n+1}-\tau(t_{n+1},\uh(t_{n+1}))-\xi_i^h)(t_{n}-\tau(t_{n},\uh(t_{n}))-\xi_i^h) \leq 0
\ee
then $\uh(t-\tau)$ is sufficiently smooth on that step, and we go on to the next step.
But if \eqref{eq:bkptneeded} is satisfied
for some previous breaking point $\xi_i^h$ of order$\leq p-3$, then there will be a new breaking point 
$\xi_j^h\in(t_n,t_{n+1}]$ of order$\leq p-2$ that needs to be approximated. To do this define
\be \label{eq:alphat}
\alpha(t)=t-\tau(t,\uh(t)), \quad t\in[t_n,t_{n+1}].
\ee
Then $\xi_j^h$ is defined by 
\be \label{eq:bpalph}
\alpha(\xi_j^h)=\xi_i^h. 
\ee
To solve \eqref{eq:bpalph} efficiently 
we replace $\alpha(t)$ by a quadratic polynomial.
Let $\alpha_n=\alpha(t_n)$, $\alpha_{n+1}=\alpha(t_{n+1})$, 
$\alpha_{n+\frac12}=\alpha(t_{n+\frac12})$, where $t_{n+\frac12}=\frac12(t_n+t_{n+1})=t_n+\frac12h_n$, and define the finite differences
\begin{gather} \label{eq:Delta}
\Delta\alpha =\dfrac{1}{h_n}(\alpha_{n+1}-\alpha_n)=\alpha'(t_{n+\frac12})+\cO(h_n^2),\\ \label{eq:Delta2}
\Delta^{\!2}\alpha =\dfrac{1}{h_n^2}(\alpha_{n+1}-2\alpha_{n+\frac12}+\alpha_n)=\alpha''(t_{n+\frac12})+\cO(h_n^2).
\end{gather}
Then let $p(\theta)$ be the quadratic polynomial
\be \label{eq:pquad}
p(\theta)=\alpha_n+h_n\theta\Delta\alpha-2\theta(1-\theta)h_n^2\Delta^{\!2}\alpha
\ee
which interpolates $\alpha(t_n+\theta h_n)$ at $\theta=0$, $\tfrac12$ and $1$. 
Define $\theta_j^p$ by 
\be \label{eq:bpp}
p(\theta_j^p)=\xi_i^h
\ee
so that 
$\xi_j^p=t_n+\theta_j^p h_n$ is an approximation to $\xi_j^h$. 
If $\Delta^{\!2}\alpha=0$ then it follows from \eqref{eq:pquad}
and \eqref{eq:bpp} that
\be \label{eq:xiplin}
\xi_j^p=t_n+\theta_j^p h_n=t_n+\frac{\xi_i^h-\alpha_n}{\Delta\alpha},
\ee
otherwise taking the negative square root in the quadratic formula
\begin{align} \label{eq:xipquad}
\xi_j^p
&=t_n
+ \frac{2h_n\Delta^2\alpha-\Delta\alpha}{4\Delta^2\alpha}
\left(1-\left(1+\frac{8(\xi_i^h-\alpha_n)\Delta^2\alpha}{(2h_n\Delta^2\alpha-\Delta\alpha)^2}\right)^\frac12\right)
\\ \label{eq:xipquadexp}
&=t_n+\frac{\xi_i^h-\alpha_n}{\Delta\alpha-2h_n\Delta^2\alpha}
-\frac{2(\xi_i^h-\alpha_n)^2\Delta^2\alpha}{(\Delta\alpha-2h_n\Delta^2\alpha)^3}
+\cO(h_n^3).
\end{align}
Equation \eqref{eq:xipquadexp} follows from \eqref{eq:xipquad} using the binomial series, noting that 
$\xi_i^h\in[\alpha_n,\alpha_{n+1}]$ and \eqref{eq:Delta} implies that $\xi_i^h-\alpha_n=\cO(h_n)$.
The negative square root is required in \eqref{eq:xipquad} to obtain the correct answer, which we can see from 
\eqref{eq:xipquadexp} is a perturbation of the linear approximation \eqref{eq:xiplin}. (The positive square root would result in $\xi_j^p-t_n\to-\alpha'(t_{n+\frac12})/2\alpha''(t_{n+\frac12})$ as $h_n\to0$, which is not $\cO(h_n)$.)

By standard approximation results for quadratic interpolating polynomials, when $\xi_j^h$ is a simple root
of \eqref{eq:bpalph} we have
\be \label{eq:xiapprox}
|\xi_j^p-\xi_j|\leq |\xi_j^p-\xi_j^h|+|\xi_j^h-\xi_j|=\cO(h^{\min(p,3)}).
\ee
Equation \eqref{eq:xiapprox} holds whether we use \eqref{eq:xipquad} or 
the truncated form in \eqref{eq:xipquadexp} to evaluate $\xi_j^p$. Equation \eqref{eq:xipquadexp} may be preferable to implement, as it avoids the cancellation errors that arise in
\eqref{eq:xipquad} when $h_n$ is very small. 

Table~\ref{tab:Ohr} shows that $\xi_j^p$ defined by \eqref{eq:xipquad} or \eqref{eq:xipquadexp},
gives an approximation of the required accuracy for all methods of order $p\leq6$ for breaking points of order $k\geq1$. However, it is not sufficiently accurate
for a breaking point of order $k=0$
to obtain order $p\geq4$ convergence. For the FRCK4 method with a breaking point of order $k=0$
we use $t_n$ and $\xi_j^p$ as the base points to do one step of Secant correction to  
$\xi_j^p$ to obtain $\xi_j^s$ defined by
\be \label{eq:sec1}
\xi_j^s=\textrm{secant}(\xi_j^p,t_n)=\xi_j^p-(\alpha(\xi_j^p)-\xi_i)\frac{\xi_j^p-t_n}{\alpha(\xi_j^p)-\alpha_n}.
\ee
From the standard Secant method error analysis (see e.g.~Section 3.2.3 of \cite{IK66}) we obtain
\be \label{eq:secerr1}
|\xi_j^s-\xi_j^h|=\cO(|\xi_j^p-\xi_j^h|\times|\xi_j^h-t_n|)=\cO(h^3\times h)=\cO(h^4),
\ee
so $\xi_j^s$ has the needed accuracy for the FCRK4 method to obtain global order $4$ with an order $0$ breaking point.  If fifth or higher order FCRK methods are implemented, a second secant step could be used 
to replace $\xi_j^s$ by 
$\xi_j^{ss}=\textrm{secant}(\xi_j^s,\xi_j^p)$ for order $0$ breaking points, for which
\be \label{eq:secerr2}
|\xi_j^{ss}-\xi_j^h|=\cO(|\xi_j^s-\xi_j^h|\times|\xi_j^p-\xi_j^h|)=\cO(h^4\times h^3)=\cO(h^7).
\ee

\CCLsubsubsection{Breaking Point Detection Algorithm} 

To implement the algorithm for an FCRK$p$ method, we begin with a list of the breaking points $t=\xi_i\leq t_0$ arising from the initial function and their orders for all breaking points of order $k\leq p-3$ (not forgetting to include $t_0$ itself). This list divides $[t_0-\tau_{\textit{max}},t_0]$ into subintervals  on which
the solution $u(t)$ is smooth. We assume that there are finitely many such breaking points so that these subintervals all have non-zero length. At the initial time the delay $t_0-\tau(t_0,u(t_0))$ falls into one such interval $[\xi_i,\xi_{i+1})$. 

To advance the solution through one step $[t_n,t_{n+1}]$ we first solve the problem \eqref{eq:bpsmooth}, with 
the smooth delayed term $\uhm(t-\tau)$ defined by taking the function $\uh(t-\tau)$ defined for $t-\tau\in[\xi_i,\xi_{i+1})$ and analytically extending it to the real line. Since $\uh$ is typically polynomial or piecewise polynomial this extension is trivially obtained by just evaluating the polynomial expression outside the domain on which it is a valid approximation of the solution.

Once \eqref{eq:bpsmooth} is solved for $t\in[t_n,t_{n+1}]$, evaluate the right-hand side of \eqref{eq:bkptneeded} for each previous delay $\xi_i$ in the list. If \eqref{eq:bkptneeded} is not satisfied for any $\xi_i$, then there is no breaking point for this step, and the step is accepted. If \eqref{eq:bkptneeded} is satisfied
for one or more breaking points,
determine the new (approximated) breaking point $\xi_j\in(t_n,t_n+1]$ using either \eqref{eq:xipquad} or 
\eqref{eq:xipquadexp} (or using \eqref{eq:sec1} for FCRK4 if $\xi_i$ is an order $-1$ breaking point).
If there is more than one such $\xi_j\in(t_n,t_n+1]$ then discard all but the first one. 
We then redefine $t_{n+1}$ to be equal to $\xi_j$, and 
the already computed solution of \eqref{eq:bpsmooth} is truncated to define the solution $u^h(t)$ for $t\in[t_n,\xi_j]=[t_n,t_{n+1}]$.

Before moving on to the next step, if $\xi_i$ was a breaking point of order $k$ then the new breaking point $\xi_j$ is assigned order $k+1$ (in reality it might have even higher order, but its order cannot be less than $k+1$), and if $k+1\leq p-3$ it is added to the list of previous breaking points. Finally, for the next step we update the interval in which $t-\tau$ falls to be the other side of the breaking point $\xi_i$ that triggered
the new breaking point $\xi_j$. This is illustrated in Figure~\ref{fig:bp5}, where the function $\uhp(t)$ and its analytic extension will be used to compute the solution on the next step. 

For problems for which $t-\tau(t,u(t))$ is a monotonically increasing function of $t$, the breaking point $\xi_i$
will not be encountered again, and can be deleted from the list. Otherwise it remains in the list, but because of the numerical approximation errors may be encountered again at the beginning of the next step in which case it is ignored.

Once $t>t_0+(p-1)\tau_\textit{max}$ there can be no more breaking points that affect the order of the method, and breaking point detection can be turned off. Problems with multiple delays are handled similarly, and just require more bookkeeping.

\CCLsubsubsection{Performance of FCRK methods with breaking points}

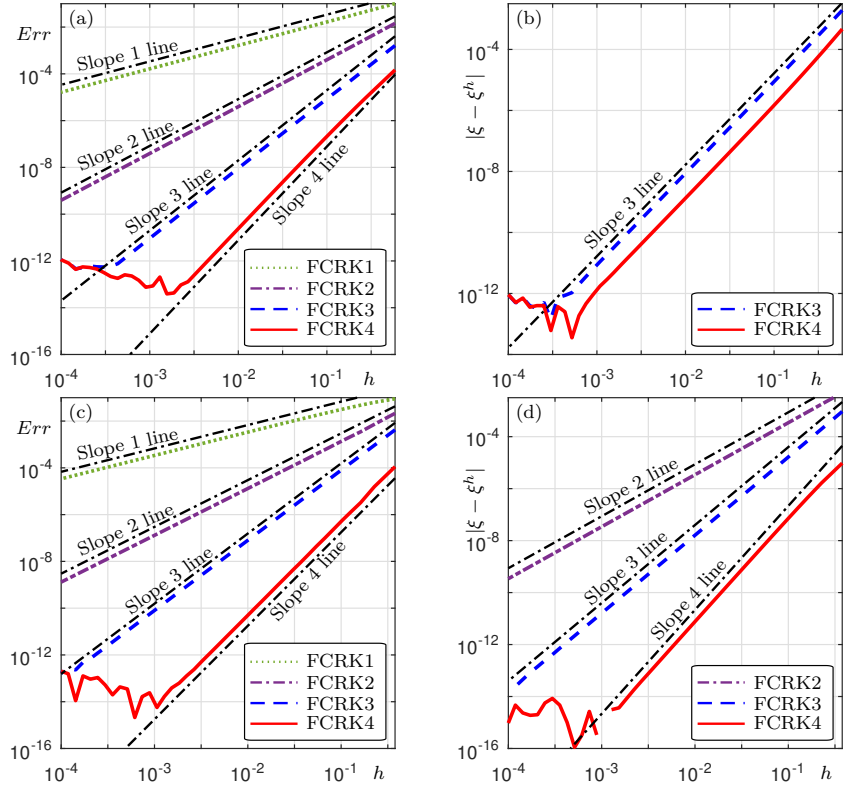
\begin{figure}[t!]
\centering
\scalebox{0.9}{\includegraphics{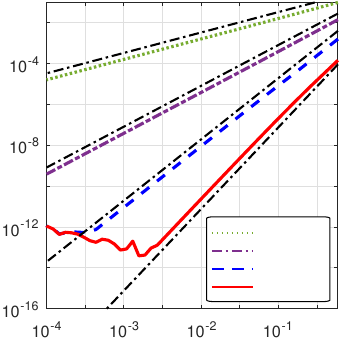}
\put(-38,50){\scriptsize FCRK1}
\put(-38,41){\scriptsize FCRK2}
\put(-38,32){\scriptsize FCRK3}
\put(-38,24){\scriptsize FCRK4}
\put(-135,135){\rotatebox{12}{\scriptsize Slope 1 line}}
\put(-135,93){\rotatebox{24}{\scriptsize Slope 2 line}}
\put(-115,66){\rotatebox{36}{\scriptsize Slope 3 line}}
\put(-53,70){\rotatebox{48}{\scriptsize Slope 4 line}}
\put(-14,3){\scriptsize $h$}
\put(-160,148){\scriptsize $Err$}
\put(-138,155){\scriptsize (a)}}\hspace*{2em}\scalebox{0.9}{\includegraphics[scale=1]{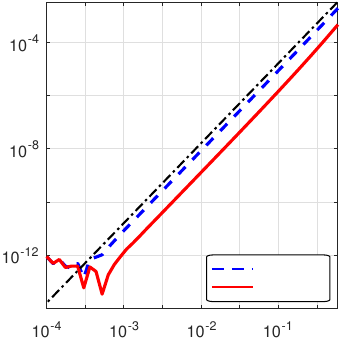}
\put(-38,32){\scriptsize FCRK3}
\put(-38,24){\scriptsize FCRK4}
\put(-110,60){\scriptsize \rotatebox{45}{Slope 3 line}}
\put(-14,3){\scriptsize $h$}
\put(-160,106){\rotatebox{90}{\scriptsize $|\xi-\xi^h|$}}
\put(-138,155){\scriptsize (b)}}

\scalebox{0.9}{\includegraphics[scale=1]{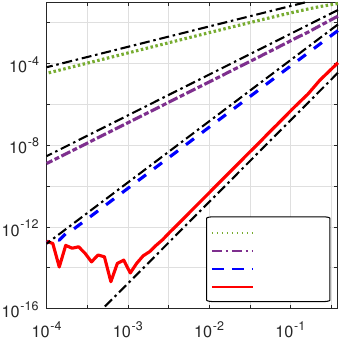}
\put(-38,50){\scriptsize FCRK1}
\put(-38,41){\scriptsize FCRK2}
\put(-38,32){\scriptsize FCRK3}
\put(-38,24){\scriptsize FCRK4}
\put(-135,137){\rotatebox{12}{\scriptsize Slope 1 line}}
\put(-135,97){\rotatebox{24}{\scriptsize Slope 2 line}}
\put(-115,72){\rotatebox{36}{\scriptsize Slope 3 line}}
\put(-54,69){\rotatebox{48}{\scriptsize Slope 4 line}}
\put(-10,3){\scriptsize $h$}
\put(-160,148){\scriptsize $Err$}
\put(-138,155){\scriptsize (c)}}\hspace*{2em}\scalebox{0.9}{\includegraphics[scale=1]{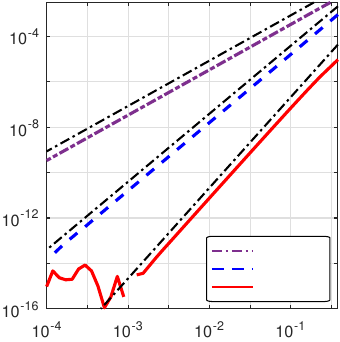}
\put(-38,41){\scriptsize FCRK2}
\put(-38,32){\scriptsize FCRK3}
\put(-38,24){\scriptsize FCRK4}
\put(-110,114){\rotatebox{28}{\scriptsize Slope 2 line}}
\put(-110,76){\rotatebox{41}{\scriptsize Slope 3 line}}
\put(-82,61){\rotatebox{48}{\scriptsize Slope 4 line}}
\put(-10,3){\scriptsize $h$}
\put(-160,106){\rotatebox{90}{\scriptsize $|\xi-\xi^h|$}}
\put(-138,155){\scriptsize (d)}}
\caption{The error
$Err$ \eqref{eq:Err} for the four methods FCRK1-4
(a) for Problem \eqref{eq:test1} with breaking point detection applied for FCRK3-4 and 
(c) for Problem \eqref{eq:test2} with breaking point detection applied for FCRK2-4, with one secant step for FCRK4.
(b) The error in the computed order $1$ breaking point for FCRK3-4 for Problem \eqref{eq:test1}
and (d) the error in the computed order $0$ breaking point for FCRK2-4 for Problem \eqref{eq:test1}.} \label{fig:testEH12}
\end{figure}

In Figure~\ref{fig:testEH12} we present the results of applying our algorithm to the test problems
\eqref{eq:test1} and \eqref{eq:test2}. We previously solved these problems without breaking point detection
in Figure~\ref{fig:testnone}.

Figure~\ref{fig:testEH12}(a) and (b) show the 4 FCRK methods applied to \eqref{eq:test1}
over the time interval $t\in[1,5]$, for which there is a breaking point of order $1$ at $\xi=2$. As in the previous tests we define constant step-sizes $h$ using \eqref{eq:hlambda} with $\lambda=1/2$ so that the breaking point is between mesh points. Since the only breaking point is of order $1$, no breaking point detection is applied for FCRK1 or FCRK2, while for FCRK3 and FCRK4 the breaking point is approximated by $\xi_j^p$ defined by \eqref{eq:xipquad}, as described above. Figure~\ref{fig:testEH12}(a) shows that each FCRK$p$ method achieves order $p$ convergence
until the error is smaller than $10^{-12}$ when finite precision arithmetic errors become significant.  
Figure~\ref{fig:testEH12}(b) shows that $\xi_j^p$ given by \eqref{eq:xipquad} defines an $\cO(h^3)$ approximation to the exact breaking point $\xi_j$
of the original problem \eqref{eq:test1}, in accordance with \eqref{eq:xiapprox}. 

Figure~\ref{fig:testEH12}(c) and (d) show the 4 FCRK methods applied to \eqref{eq:test2}
over the time interval $t\in[0,1]$, for which there is a breaking point of order $0$ at $\xi_j\in(0,1)$. 
No breaking point detection is applied for FCRK1, while for FCRK2 and FCRK3 the breaking point is approximated by $\xi_j^p$ defined by \eqref{eq:xipquad}. For FCRK4 the Secant approximation  $\xi_j^s$ defined by \eqref{eq:sec1} 
is used to approximate $\xi_j$. Recall from Figure~\ref{fig:testnone}(d) that without breaking point detection all of the methods displayed only first order convergence, even when the exact breaking point $\xi_j$ was included in the computational mesh. Figures~\ref{fig:testEH12}(c) and (d) show that applying our breaking point detection algorithm, all of the methods retain their full order, and also find the breaking point to the full order of accuracy of the method.

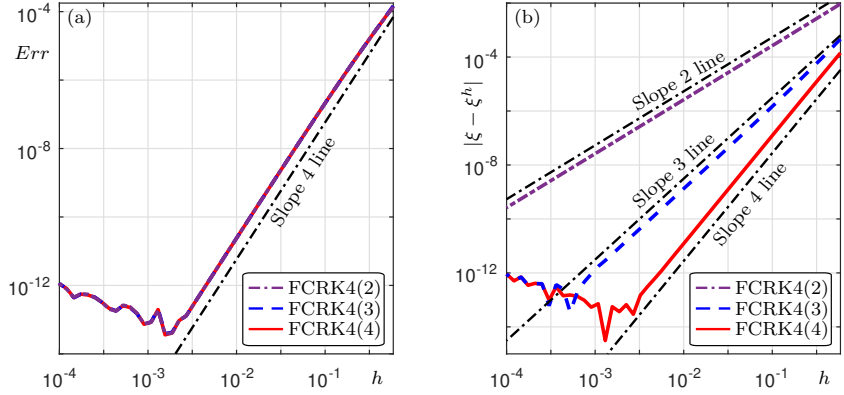
\begin{figure}[t!]
\centering
\scalebox{0.9}{\includegraphics[scale=1]{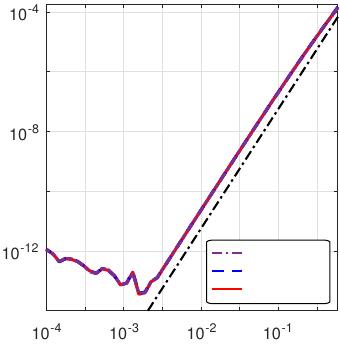}
\put(-45,41){\scriptsize FCRK4(2)}
\put(-45,32){\scriptsize FCRK4(3)}
\put(-45,24){\scriptsize FCRK4(4)}
\put(-53,69){\scriptsize \rotatebox{57}{Slope 4 line}}
\put(-10,3){\scriptsize $h$}
\put(-160,140){\scriptsize $Err$}
\put(-138,155){\scriptsize (a)}}\hspace*{2em}\scalebox{0.9}{\includegraphics[scale=1]{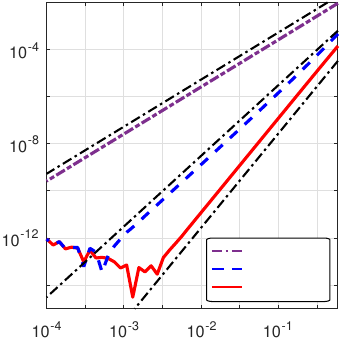}
\put(-45,41){\scriptsize FCRK4(2)}
\put(-45,32){\scriptsize FCRK4(3)}
\put(-45,24){\scriptsize FCRK4(4)}
\put(-89,118){\rotatebox{30}{\scriptsize Slope 2 line}}
\put(-89,77){\rotatebox{42}{\scriptsize Slope 3 line}}
\put(-55,61){\rotatebox{50}{\scriptsize Slope 4 line}}
\put(-10,3){\scriptsize $h$}
\put(-160,103){\rotatebox{90}{\scriptsize $|\xi-\xi^h|$}}
\put(-138,155){\scriptsize (b)}}
\caption{For Problem \eqref{eq:test1}, (a) the error
$Err$ defined by \eqref{eq:Err} for the FCRK4 method with three different orders of breaking point detection, as described in the text, and
(b) the error $|\xi-\xi^h|$ in the computed order $1$ breaking point.} \label{fig:testEH4}
\end{figure}

In Figure~\ref{fig:testEH4} we consider the FCRK4 method applied to the test problem \eqref{eq:test1} again, but with different approximations to the breaking point $\xi_j$. FCRK4(4) approximates $\xi_j$ using 
the Secant approximation $\xi_j^s$, FCRK4(3) uses the 
standard quadratic approximation \eqref{eq:xipquad}, while FCRK4(2)
uses the linear approximation \eqref{eq:xiplin}.
We already saw in Figure~\ref{fig:testnone}(a) that the FCRK4 method is reduced to second order when applied without breaking point detection, in which case $\xi_j$ is only approximated to order $h$ accuracy by the fixed mesh. 

According to Theorem~\ref{thm:approx} it is only necessary 
to find a breaking point of order $k=1$ to accuracy $\cO(h^{p/(k+1)})=\cO(h^2)$ to obtain order $p=4$, so all three implementations should retain their full order. This is 
confirmed in Figure~\ref{fig:testEH4} where the three implementations FCRK4(q) all display fourth
order convergence though they only approximate the breaking point to order $q$ accuracy for $q\in\{2,3,4\}$.
Thus fourth order would be retained by the slightly cheaper algorithm where $\alpha_{n+\frac12}$ is not computed and $\xi_j^p$ is evaluated using
\eqref{eq:xiplin} instead of \eqref{eq:xipquad}, at the cost only of a poorer approximation to $\xi_j$.
On the other hand, if desired, $\xi_j^s$ provides an approximation to $\xi_j$ 
to the full order of the method, at the cost of evaluating both    $\alpha_{n+\frac12}$
and $\alpha(\xi_j^p)$, and without changing the convergence order of the FCRK4 method.


\CCLsubsection{IVPs with Distributed and Threshold Delays}
\label{sec:numdist}

\cite{Maset25} gives a theoretical treatment of numerical methods for RFDEs in another chapter of this volume,
but there are not yet widely available reliable numerical implementations for RFDEs or threshold delays.
Methods to solve DDEs with discrete constant and state-dependent delays are in widespread use, but standard
software packages for solving DDE IVPs
only allow access to the
solution history $u_t(\theta)$ for a finite set of values of $\theta\in[-\tau,0]$ during computation.
This prevents them from being directly applicable for both threshold delay problems such as
\eqref{eq:scalarDE},\eqref{eq:thre_scalarDE} and for DDEs with constant distributed delays.
The hematopoiesis model \eqref{eq:Q}-\eqref{eq:N} is an example of a DDE with 
a distributed delay, as the term $A_N(t)$, defined by \eqref{eq:AN},
involves an integral of a function of the state-variable $G$ over the time interval
$[t-\tauNM(t)-\tauNP,t-\tauNM(t)]$, as well as the threshold delay $\tauNM(t)$ defined by \eqref{eq:tauNM}.

While it is possible to develop solvers for specific problems (see \cite{Langlois2017,CGHvD22}),
it is a lot of work to reinvent and recode all the functionality of a sophisticated IVP solver for each new 
problem. Until a comprehensive numerical RFDE IVP solver is available, it is convenient and time saving to leverage the existing discrete delay software packages to solve problems with distributed delays. 

In Sections~\ref{subsec:scalthresdde} and~\ref{sec:linearize} we already
proposed to replace
\eqref{eq:scalarDE},\eqref{eq:thre_scalarDE} by the system
\eqref{eq:scalarDE},\eqref{eq:taudot} with the additional initial condition
\eqref{eq:thre_scalarDEIC}. This converts the threshold delay problem to a discrete delay DDE,
to which the methods discussed in the previous section can be applied. 

In general solving for $\tau_0$ from \eqref{eq:thre_scalarDEIC} requires evaluating the integral only for the first step of the computation, but even this can be avoided if the initial function $\phi$ is taken to be a constant function, $\phi(t)=\phi$,
in which case \eqref{eq:thre_scalarDEIC} reduces to $\tau_0=a/V(\phi)$. 
After the numerical computation of the IVP solution of \eqref{eq:scalarDE},\eqref{eq:taudot}
it should be verified that the computed solution of \eqref{eq:taudot} satisfies
\eqref{eq:thre_scalarDE} to a reasonable tolerance over the domain of the numerical computation,
which is easy to do, using the numerical solution interpolant $u^h(t)$.
This is necessary, as the value of the constant $a$ only enters through the initial condition \eqref{eq:thre_scalarDEIC},
and numerical errors may accumulate through the computation.

If the threshold condition is not solved to a reasonable accuracy then a penalty term can be added to \eqref{eq:taudot}
to obtain
\begin{equation} \label{eq:taudotpen} 
 \frac{\phantom{t}d}{dt}\tau
 = 1- \frac{V(u(t))}{V(u(t-\tau))}+\gamma\left[a-\int^{t}_{t-\tau}\hspace*{-0.5em}
 V(u(s)) ds\right].
\end{equation}
Even though we have reintroduced an integral term here, it is much easier to numerically evaluate the integral in \eqref{eq:taudotpen} than in \eqref{eq:thre_scalarDE} because $\tau$ is not implicitly defined in \eqref{eq:taudotpen}, but rather is known for any value of $t$ through the solution of \eqref{eq:taudotpen} up to that time point.
\cite{wendy-phd} discusses \eqref{eq:taudotpen} at length, where it is found that for $\gamma>0$ this stabilises the computation
of the threshold delay, but numerical evaluation of the integral in \eqref{eq:taudotpen} introduces an error of $10^{-8}$ or worse in the threshold condition \eqref{eq:thre_scalarDE}. Equation \eqref{eq:taudotpen} was not used by
\cite{gedeon2024dynamics} where it was found that \eqref{eq:taudot} resulted in errors smaller than $10^{-12}$ in the threshold condition \eqref{eq:thre_scalarDE}.

Differentiating a constraint to obtain an extra differential equation is very reminiscent of differential algebraic equations (DAEs). In particular, the system \eqref{eq:scalarDE}-\eqref{eq:thre_scalarDE} can be regarded as an index-1 retarded DAE. 

To convert the model \eqref{eq:Q}-\eqref{eq:N} to a discrete
delay DDE in \cite{craig2016} both \eqref{eq:tauNM} and \eqref{eq:AN} are differentiated
to obtain
\begin{gather} \label{eq:tauNMdiff}
\tfrac{d}{dt}\tauNM(t) =1-\tfrac{\VN(G(t))}{\VN(G(t-\tauNM(t)))}  \\ \notag
\tfrac{d}{dt}A_{N\!}(t)  = A_{N\!}(t)\Bigl[\bigl(1-\tfrac{d}{dt}\tauNM(t)\bigr)\bigl(\etaNP(G(t\!-\!\tauNM(t))) -\etaNP(G(t\!-\!\tau_N(t)))\bigr)\\ \label{eq:ANdiff}
 \hspace{3em}  -\gammaNM\tfrac{d}{dt}\tauNM(t)\Bigr].
\end{gather}
Just as with the simpler model problem \eqref{eq:scalarDE},\eqref{eq:thre_scalarDE}, care needs to be taken to use the correct initial condition for both \eqref{eq:tauNMdiff} and \eqref{eq:ANdiff}. 
Also, after the IVP is solved it should be verified that the numerically computed solutions of \eqref{eq:tauNMdiff} and \eqref{eq:ANdiff}
still satisfy \eqref{eq:tauNM} and \eqref{eq:AN} over the solution interval to an acceptable tolerance, as was found to be the case in \cite{craig2016}.


\CCLsubsection{Numerical Continuation and Bifurcation Detection} 
\label{subsec:num_bvps}

Numerical continuation and bifurcation techniques for ODEs are very mature and well known, and there are many extensions to constant delay DDEs. For discrete state-dependent delays DDE-Biftool (\cite{ddebiftool}) provides a suite of
\cite{matlab} routines for detecting bifurcations and performing numerical continuation of solution branches to compute bifurcation diagrams. This is described extensively by \cite{KS23} in a previous CISM volume, and so will not be repeated here.

As for the DDE IVP solvers, DDE-Biftool is formulated to solve \eqref{eq:ddde} with discrete delays, and so is not directly applicable for threshold delays or more general RFDE problems \eqref{eq:rfde}. 
In order to leverage all of the capabilities of DDE-Biftool without having to write a new software package from scratch, our preferred approach is to alter the formulation of the threshold delay DDE so that DDE-Biftool can be used. Differentiating the threshold condition, as we did to solve IVPs in Section~\ref{sec:numdist} does not work, because the value of the constant $a$ in \eqref{eq:thre_scalarDE} is lost when it is differentiated.

To apply DDE-Biftool to \eqref{eq:scalarDE},\eqref{eq:thre_scalarDE} we introduce extra dummy delays to discretise
the threshold integral.
Recall \eqref{ineq:thre_bdd}, which implies that the delay $\tau=\tau(u_t)$ satisfies
\begin{equation*}
	t-\tau \in [t-\tau_{\textit{max}}, t-\tau_{\textit{min}}] \subset [t-\tau_{\textit{max}}, t] . 
\end{equation*}
We discretise the time interval $[t-\tau_{\textit{max}}, t]$ uniformly with a sequence of  mesh points
\begin{equation*}
	t = x_0 > x_1 > \cdots > x_N = t-\tau_{\textit{max}},
\end{equation*}
and define $N$ constant ``dummy'' delays 
\begin{equation*}
	\tau_j = t-x_j = \frac{j}{N}\tau_{\textit{max}} = \frac{j}{N} \frac{a}{V_{\textit{min}}}.
\end{equation*}
Let $J_h(j)$ denote the cumulative quadrature sum over the first $j$ subintervals
\begin{equation*}
	J_h(j) = \int_{x_j}^{x_0} V(u(s)) ds = \int_{t-\tau_j}^{t} V(u(s)) ds . 
\end{equation*}
To approximate $\tau$ that solves \eqref{eq:thre_scalarDE}, we successively add subintervals to the integral and identify the index $j$ such that 
\begin{equation*}
	a \geq J_h(j) \quad \text{and} \quad a < J_h(j+1), 
\end{equation*}
which implies that $\tau \in [\tau_j, \tau_{j+1})$. 

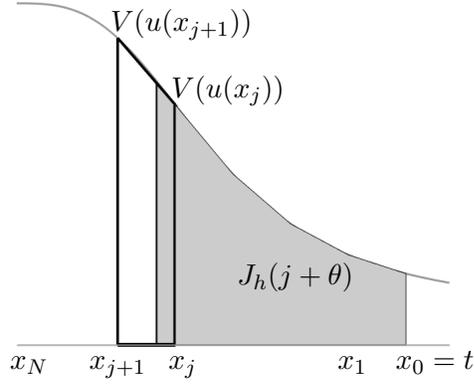
\begin{figure}[tp!]
	\centering
	\includegraphics[scale=0.25]{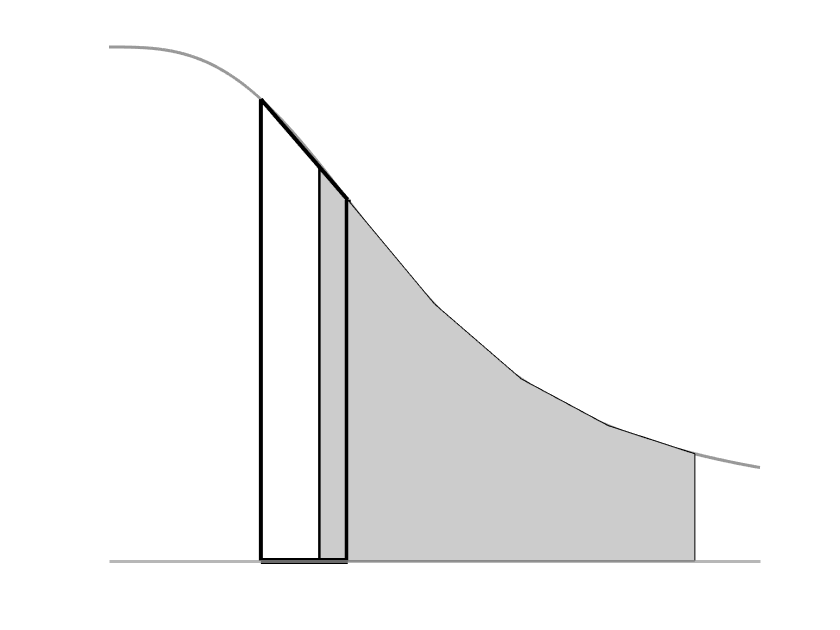}
	\put(-186,8){$x_N$}
	\put(-156,8){$x_{j+1}$}
	\put(-126,8){$x_j$}
	\put(-62,8){$x_1$}
	\put(-40,8){$x_0=t$}
	\put(-100,40){$J_h(j+\theta)$}
	\put(-148,136){$V(u(x_{j+1}))$}
	\put(-125,111){$V(u(x_{j}))$}
	\caption{A schematic representation of the numerical quadrature used to approximate the threshold delay in equation \eqref{eq:thre_scalarDE}.} \label{fig:thresquad}
\end{figure}

It remains to find $\theta\in[0,1)$ so that $J_h(j+\theta)=a$ and 
hence $\tau_j = (j+\theta)a/N V_{\textit{min}}$. To do this we extend the definition of $J_h(j)$ to non-integer values of $j$ as illustrated in Figure~\ref{fig:thresquad}. See \cite{gedeon2022operon,ifac22} for further details.

In this way \eqref{eq:scalarDE},\eqref{eq:thre_scalarDE} is converted into a discrete-delay DDE with $N$ constant delays and one state-dependent delay $\tau$ (whose value depends on the value of $u(t)$ and $u(t-\tau_j)$ for each of the constant delays. DDE-Biftool can be applied directly to this problem, and in Section~\ref{subsec:scalarDEonethres} we will present numerical results obtained using this approach.

\CCLsection{Examples}
\label{sec:examples}

We finish with two extended examples: in Section~\ref{subsec:scalarDEonethres} we present the dynamics of the scalar DDE with threshold delay \eqref{eq:scalarDE}, and 
in Section~\ref{subsec:scalarDEtwosd} we present the dynamics of the scalar DDE with two linearly state-dependent discrete delays  \eqref{eq:twostatedep}.

\CCLsubsection{Scalar Threshold Delay Example}
\label{subsec:scalarDEonethres}

Following \cite{gedeon2024dynamics}, we consider the scalar threshold delay DDE
\eqref{eq:scalarDE},\eqref{eq:thre_scalarDE}
with $g(u)$ and $V(u)$ defined by the Hill functions
\begin{equation} \label{eq:gV}
	g(u) =\dfrac{g^{-}\theta_g^n + g^{+} u^n}{\theta_g^n+u^n}, \quad V(u) =\dfrac{v^-\theta_v^m + v^{+} u^m}{\theta_v^m+u^m}
\end{equation} 
where $g^-, g^+, v^-, v^+, \theta_g, \theta_v > 0$ and $m, n \in \mathbb{N}$. When $v^+=v^-$ so $V(u)$ is a constant function, the system reduces to a constant delay DDE.
\cite{gedeon2024dynamics} study all of the different combinations of increasing, decreasing, or constant
functions $g$ and $V$.
Here we highlight some interesting cases.

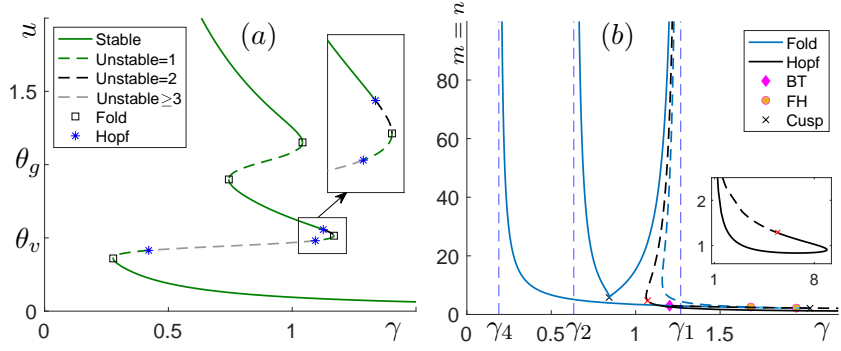
\begin{figure}[tp!]
	\centering	
	\includegraphics[scale=0.43]{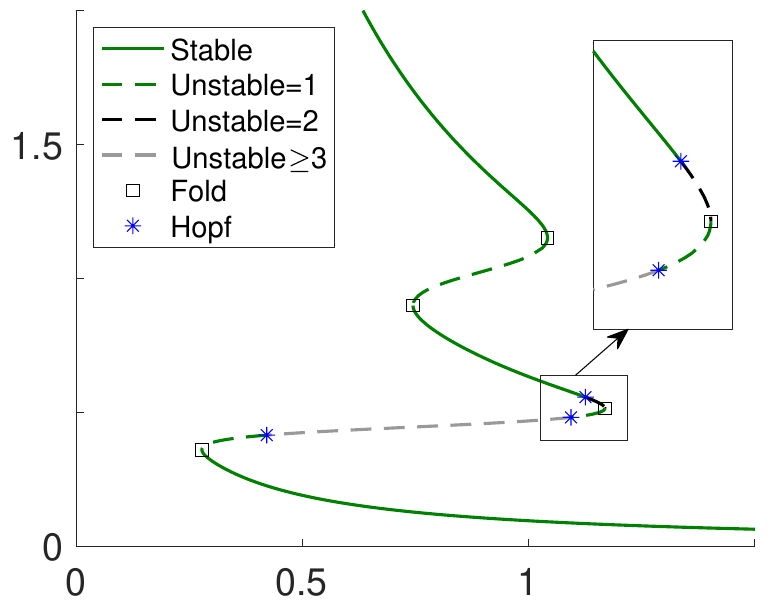}
	\put(-152,117){\rotatebox{90}{$u$}}
	\put(-15,5){$\gamma$}
	\put(-155,68){$\theta_g$}
	\put(-155,38){$\theta_v$}
    \put(-70,115){$(a)$}
\includegraphics[scale=0.43]{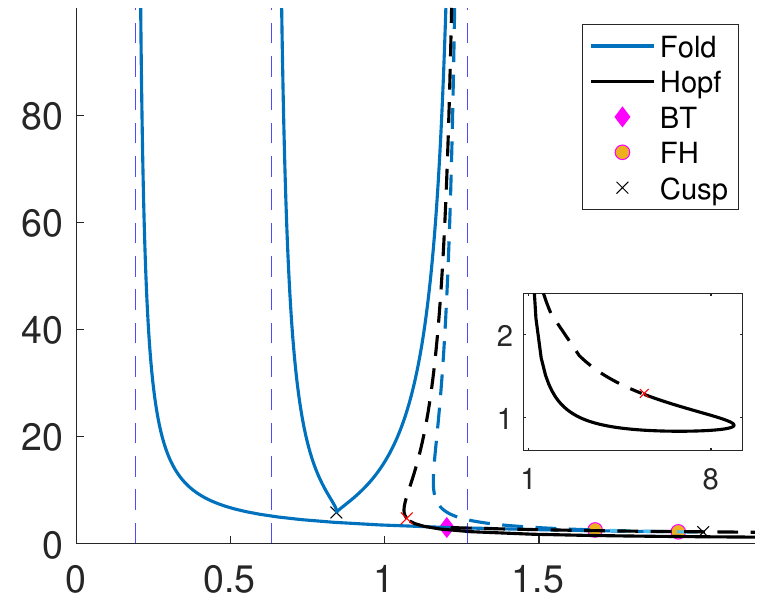}
	\put(-95,115){$(b)$}
	\put(-150,107){\rotatebox{90}{\scriptsize$m=n$}}
	\put(-14,5){$\gamma$}
	\put(-138,5){$\gamma_4$}
	\put(-108,5){$\gamma_2$}
	\put(-69,5){$\gamma_1$}
	\caption{Bifurcation diagram of \eqref{eq:scalarDE}-\eqref{eq:thre_scalarDE} with $g$ and $V$ both increasing and
    parameters $\beta=1.4$, $\mu=0.2$, $g^-=0.5$, $g^+=1$, $\gamma=0.8$, $\theta_v=0.5$, $\theta_g=a=1$, $v^-=0.1$ and $v^+=2$.
    (a) One parameter continuation in $\gamma$ with $m=n=20$. (b)  Two-parameter continuations in $m=n$ and $\gamma$ of the fold and the Hopf bifurcations. (Reproduced from \cite{gedeon2024dynamics} under CC-BY 4.0 license.)}
	\label{fig:gupvup}
\end{figure}

When $g(u)$ and $V(u)$ are both increasing functions, it is possible to have up to five coexisting steady states as illustrated in Figure~\ref{fig:gupvup}(a). In particular, a region of tristability is observed, characterised by three stable steady states separated by two unstable ones. This contrasts sharply with the dynamics of the same operon model with constant state delay, for which at most three steady states are observed when $g$ is an increasing Hill function. If we flip the axes in the
Figure~\ref{fig:gupvup}(a) we see that each value of $u>0$ is a steady state for a unique value of $\gamma$, and that the 
values of $\theta_g$ and $\theta_v$ organise the locations of the fold bifurcations.

Furthermore, Figure~\ref{fig:gupvup}(b) shows the $\gamma-m$ two-parameter continuation of the fold and Hopf bifurcations along with codimension-2 bifurcations. This  
displays other dynamics in the threshold model that are not seen in the corresponding constant delay model including the loss of stability at a Hopf bifurcation
(see also Figure~\ref{fig:gupvup}(a) inset) and the codimension-two cusp, Bautin, Bogdanov-Takens (BT), and fold-Hopf bifurcations. The values $\gamma_j$ indicated in Figure~\ref{fig:gupvup}(b) come from an analysis of the dynamics in the limiting case as $m$ and/or $n\to\infty$ and the Hill function becomes a step function.
Of course, $m$ and $n$ in \eqref{eq:gV} appear in different functions representing different processes, so there is no reason to suppose they are equal other than mathematical convenience. \cite{gedeon2024dynamics} show that different behaviour is observed in the limiting case if one of $m$ and $n$ is increased faster than the other.

This example highlights the interesting dynamics that can be driven by state-dependent threshold delays.
It also serves as a warning, that care needs to be taken when considering  
the rescaling presented in Section~\ref{subsec:rescaletime}. Even though it may be possible to rescale time so that the delay becomes constant, the dynamics will not in general be the same as if the delay in the threshold delay problem is just replaced by a constant delay.  


\begin{figure}[tp!]
	\centering	
	\hspace*{-1em}\includegraphics[scale=0.42]{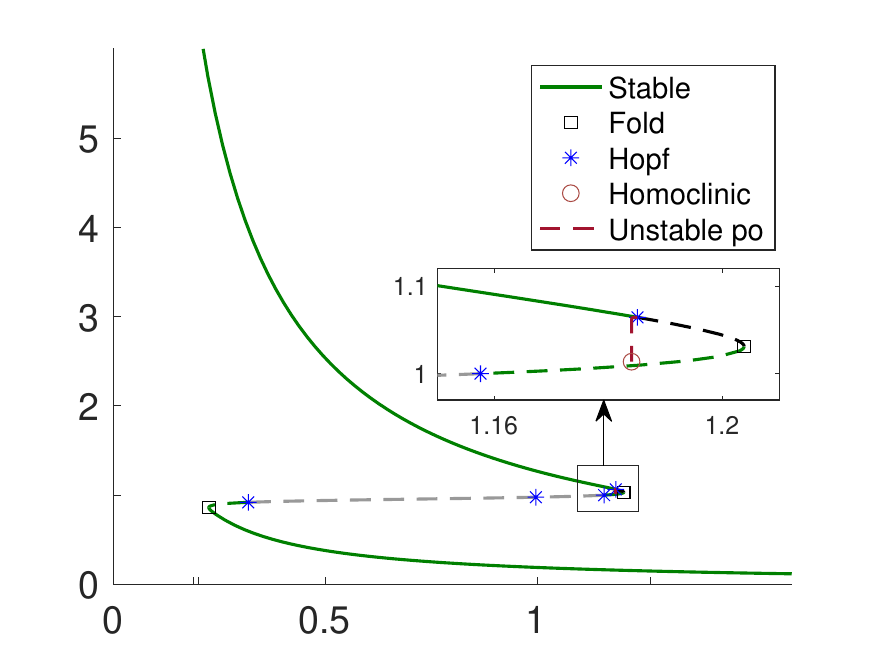}\hspace*{-1em}\includegraphics[scale=0.42]{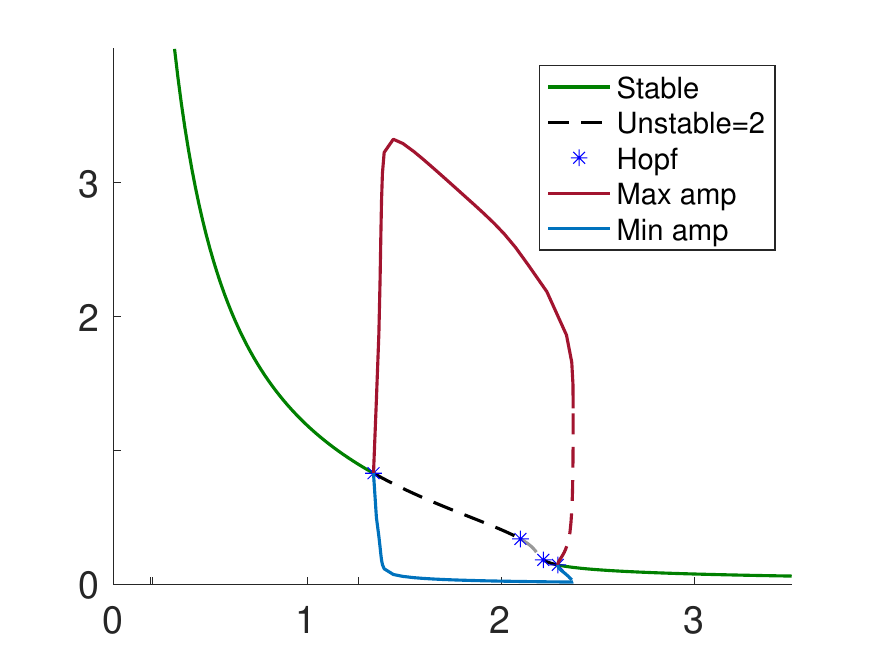}
	\put(-327,117){\rotatebox{90}{$u$}}
	\put(-330,30){$\theta_v$}
	\put(-308,7){$\gamma_4$}
	\put(-217,7){$\gamma_3$}
	\put(-191,7){$\gamma$}
	\put(-280,115){$(a)$}
	\put(-110,115){$(b)$}
	\put(-163,39){$\theta_v$}
	\put(-160,117){\rotatebox{90}{$u$}}
	\put(-150,7){$\gamma_4$}
	\put(-108,7){$\gamma_3$}
	\put(-24,7){$\gamma$}	

\hspace*{-1em}\includegraphics[scale=0.42]{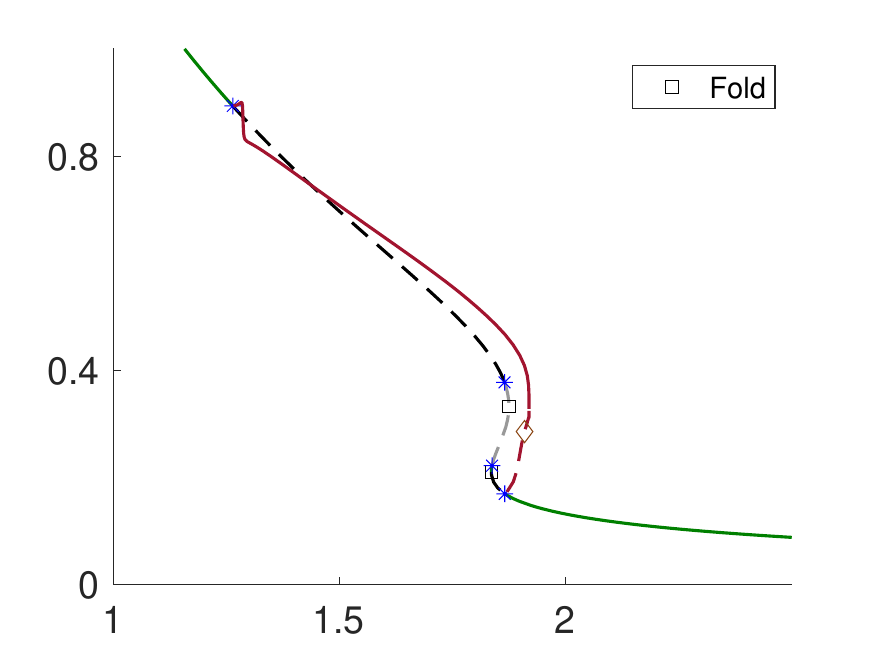}\hspace*{-1em}\includegraphics[scale=0.42]{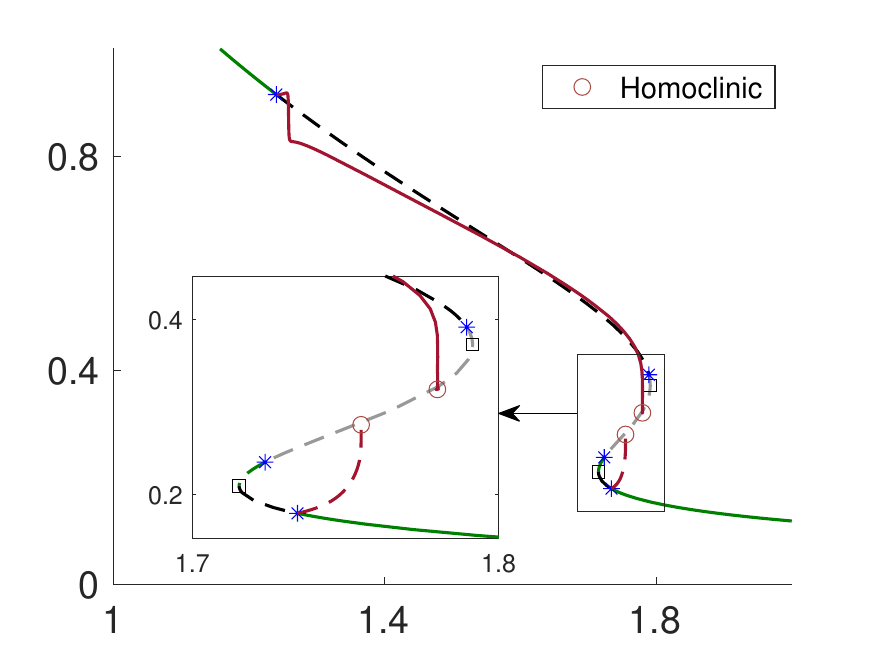}
	\put(-327,117){\rotatebox{90}{$u$}}
	\put(-191,7){$\gamma$}
	\put(-160,117){\rotatebox{90}{$u$}}
	\put(-24,7){$\gamma$}
	\put(-292,106){\small$T_l=4.87$}
	\put(-237,32){\small$T_r=42.33$}
	\put(-234,45){\small$T_{\textit{max}}=95.37$}
	\put(-117,110){\small$T_l=5.24$}
	\put(-60,22){\small$T_r=44.04$}
	\put(-245,115){$(c)$}
	\put(-153,115){$(d)$}

	\caption{One parameter bifurcation diagrams of \eqref{eq:scalarDE}-\eqref{eq:thre_scalarDE} with 
     $g$ constant, and $V$ an increasing function defined by \eqref{eq:gV}
    with (a)  $m=50$ (b) $m=2$ (c) $m=2.25$ and (d) $m=2.35$. Remaining parameters: $\beta=1.4$, $\mu=0.2$, $g^-=g^+=1$, $\theta_v=1$, $a=1$, $v^-=0.1$ and $v^+=2$. In (b) periodic orbits are represented by the maximum and minimum values of $u$ along the periodic solution profile, and in the other panels by 
    the $L^2$-norm \eqref{eq:2norm}. (Reproduced from \cite{gedeon2024dynamics} under CC-BY 4.0 license.)}
	\label{fig:gamma4ltmultgamma3gt} 
\end{figure}
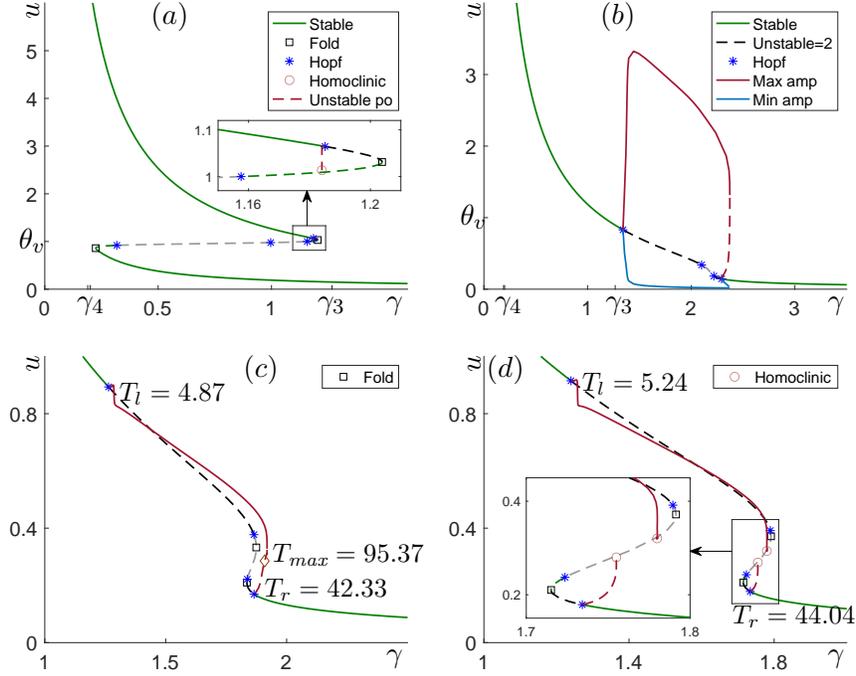

We next simplify the model and set $g_-=g_+$ so that $g$ is a constant function.
If we also set $v_-=v_+$ then $V$ would be a constant function too, from which it follows that the delay 
$\tau$ is constant, and \eqref{eq:scalarDE},\eqref{eq:thre_scalarDE} reduces to the simple linear scalar ODE
$\udot(t)  = \beta e^{-\mu\tau}g - \gamma u(t)$ with a unique globally asymptotically stable steady state
$u= \beta e^{-\mu\tau}g/\gamma$. Instead, we consider \eqref{eq:scalarDE},\eqref{eq:thre_scalarDE}
with $g$ is a constant function, and $V$ an increasing function, resulting in another DDE (like \eqref{eq:twostatedep})
where the only nonlinearity in the system arises from the state-dependency of the delay.

Figure~\ref{fig:gamma4ltmultgamma3gt} shows four one-parameter bifurcation diagrams in $\gamma$, each one for a different value of the steepness $m$ of the function $V$ in \eqref{eq:gV}, As seen in Figure~\ref{fig:gamma4ltmultgamma3gt}(a) both fold and Hopf bifurcations can occur when $m$ is sufficiently large. This time (since $g$ is a constant function) there are at most three co-existing steady states with $u=\theta_v$ on the unstable branch of steady states between the fold bifurcations.    

Figure~\ref{fig:gamma4ltmultgamma3gt}(b) shows an example for small $m$ where there are no fold bifurcations and a unique steady state for each value of $\gamma>0$, as in the $m=0$ case, but for which the Hopf bifurcations persist
even in the absence of fold bifurcations. The stable periodic orbits generated through a supercritical Hopf bifurcation subsequently collide with unstable periodic orbits arising from a subcritical Hopf bifurcation, leading to a saddle-node bifurcation of periodic orbits. In contrast, Figure~\ref{fig:gamma4ltmultgamma3gt}(a) shows that the unstable periodic orbits emanating from a subcritical Hopf bifurcation on the upper steady-state branch terminate in a homoclinic bifurcation.

In Figure~\ref{fig:gamma4ltmultgamma3gt}(c) and (d) we consider values of $m$ slightly larger than in (b).
Figure~\ref{fig:gamma4ltmultgamma3gt}(c) illustrates the effect of passing through the cusp bifurcation, giving rise to the appearance of two fold bifurcations of steady states. Notably, the maximum period along the branch of periodic orbits increases significantly as $m$ increases and the location at which this maximum occurs approaches the intermediate steady state.
Figure~\ref{fig:gamma4ltmultgamma3gt}(d) reveals the breakup of the periodic orbit branch into two disconnected segments, each of which starts at a Hopf bifurcation and terminates in a homoclinic bifurcation on the middle steady-state branch. 

Some care needs to be taken when representing periodic orbits alongside steady states on a bifurcation diagram.
In Figure~\ref{fig:gamma4ltmultgamma3gt}(a), (c) and (d) we plot the $L^2$-norm of the periodic solution
of period $T$, defined as
\be \label{eq:2norm}
\|u\|_2=\left(\frac1T\int_{t=0}^{T}|u(t)|^2 dt\right)^{\!1/2}.
\ee
In contrast, in Figure~\ref{fig:gamma4ltmultgamma3gt}(b) we
plot both $\max u(t)$ and $\min u(t)$ to represent the periodic orbit, which clearly shows the amplitude of the solution. 
Both representations can be useful on bifurcation diagrams, since all three expressions are equal to the steady state value 
at a Hopf bifurcation, and thus periodic orbits are seen to emanate from a steady states at a Hopf bifurcation. However, the 
$L^2$-norm \eqref{eq:2norm} has the extra property
that $\|u\|_2\to u_s$ as the solution converges to a homoclinic orbit to a saddle steady state $u_s$, as seen
in both Figures~\ref{fig:gupvup} and~\ref{fig:gamma4ltmultgamma3gt}.

The homoclinic bifurcation on the branch emanating from the lower Hopf point
in Figure~\ref{fig:gamma4ltmultgamma3gt}(c) and (d) does not persist for large $m$, but
instead the homoclinic bifurcation and the Hopf bifurcation itself move towards the fold bifurcation and terminate at 
BT-point (similar to the one seen in Figure~\ref{fig:gupvup}).
Bogdanov-Takens bifurcations have recently been analyzed for constant delay DDEs in
\cite{BK_BT2022},
but we are not aware of any systematic study of them in state-dependent DDE problems. 

The proximity of the cusp point to the BT point suggests that the system
may be close to a codimension-three Bogdanov-Takens-cusp (BTC) point. While we are not aware of a systematic study of this bifurcation,  they have been observed in a neuron model in \cite{AlDarabsahCampbell21}, and some of the bifurcation structures found in \cite{gedeon2024dynamics} resemble those in \cite{AlDarabsahCampbell21}.

\CCLsubsection{Two Linearly State-Dependent Delays}
\label{subsec:scalarDEtwosd}

Finally we illustrate the dynamics of the two-linearly state-dependent delay DDE \eqref{eq:twostatedep}. Recall that with $c_1=c_2=0$ this is a linear DDE, so when $c_1\ne0$ or $c_2\ne0$, the only nonlinearities arise from the state-dependency of the delays, which drives all of the interesting dynamics. Following \cite{HDMM12,CHK17}
we will set the parameters as $c_1=c_2=1$, $a_2=6\gg a_1=1.3$ and $\gamma=4.75$. The parameters $\kappa_1>0$ and $\kappa_2>0$
will be used as bifurcation parameters with the restriction that  $\kappa_2\in(0,4.75)$ so that $\kappa_2<\gamma$, which by Theorem~\ref{thm:twodelbd} ensures that the state-dependent delays cannot become advances.

\begin{figure}[tp!]
	\centering	
	\hspace*{-1em}\includegraphics[scale=0.33]{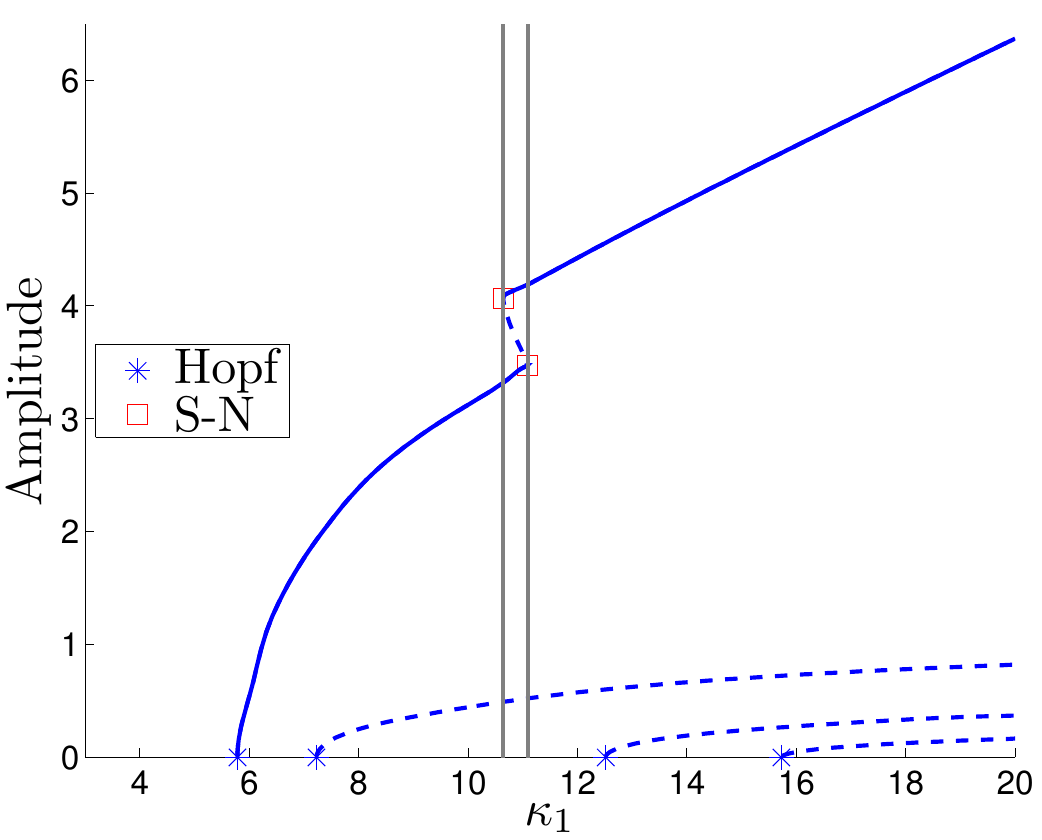}\hspace*{-0em}\includegraphics[scale=0.33]{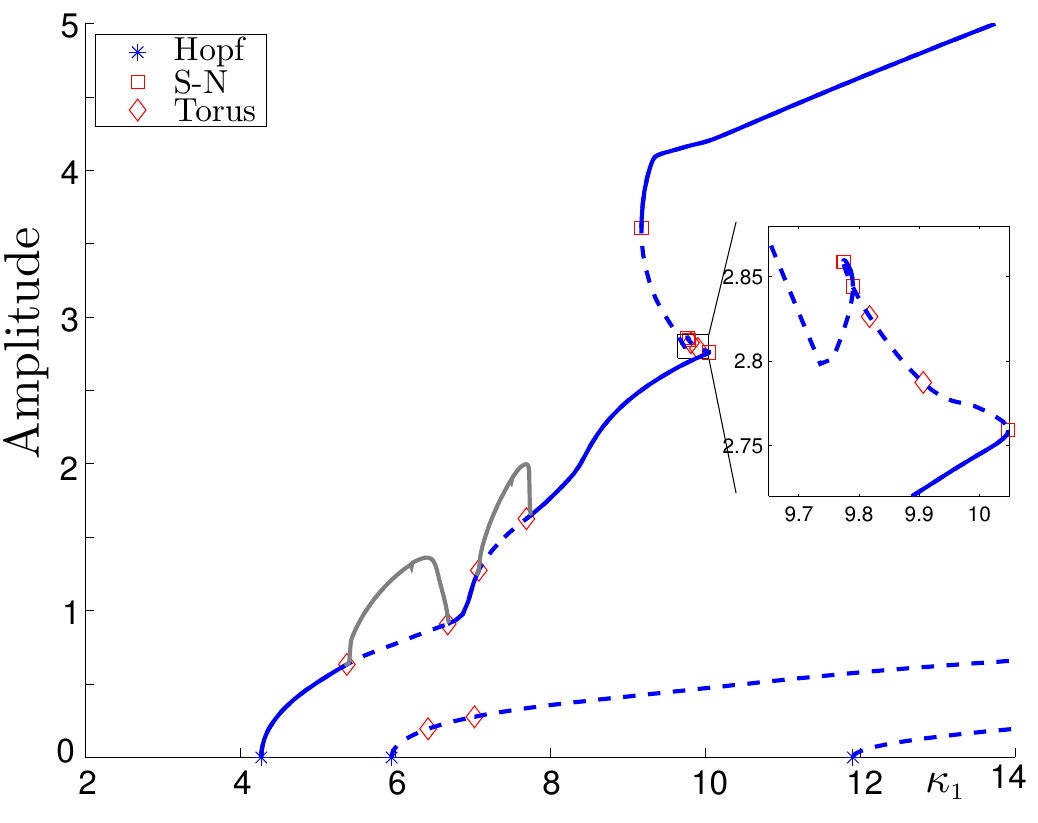}

	\caption{Bifurcation diagrams of \eqref{eq:twostatedep} as $\kappa_1$ is varied showing the amplitude and stability of periodic orbits with (a) $\kappa_2=1$ (b) $\kappa_2=2.3$. In (b) the grey curves denote stable invariant tori. (\textcopyright~AIMS; Reproduced from \cite{HDMM12} with permission.)}
	\label{fig:orianna} 
\end{figure}

Figure~\ref{fig:orianna} displays two bifurcation diagrams for different values of $\kappa_2$, revealing lots of branches of Hopf bifurcations bifurcating from the steady state $u=0$ in supercritical Hopf bifurcations, with only the first Hopf bifurcation resulting in a branch of stable periodic solutions. When $a_2\gg a_1$ there is a pair of fold bifurcations 
on this principal branch resulting in an interval of bistability of periodic orbits. 
When $\kappa_2$ is sufficiently large, as illustrated in Figure~\ref{fig:orianna}(b) there are additional torus bifurcations 
on the principal branch resulting in parameter sets for which there are no stable steady states or periodic orbits. 
Numerical IVP simulations of \eqref{eq:twostatedep} confirm the existence of stable invariant tori between the torus bifurcations. We already illustrated one of these tori in Figure~\ref{fig:tori}. We remark that torus bifurcations are also seen on the second branch of periodic solutions, but we do not have a method for computing unstable tori.

\begin{figure}[tp!]
	\centering	
	\scalebox{0.28}{\includegraphics{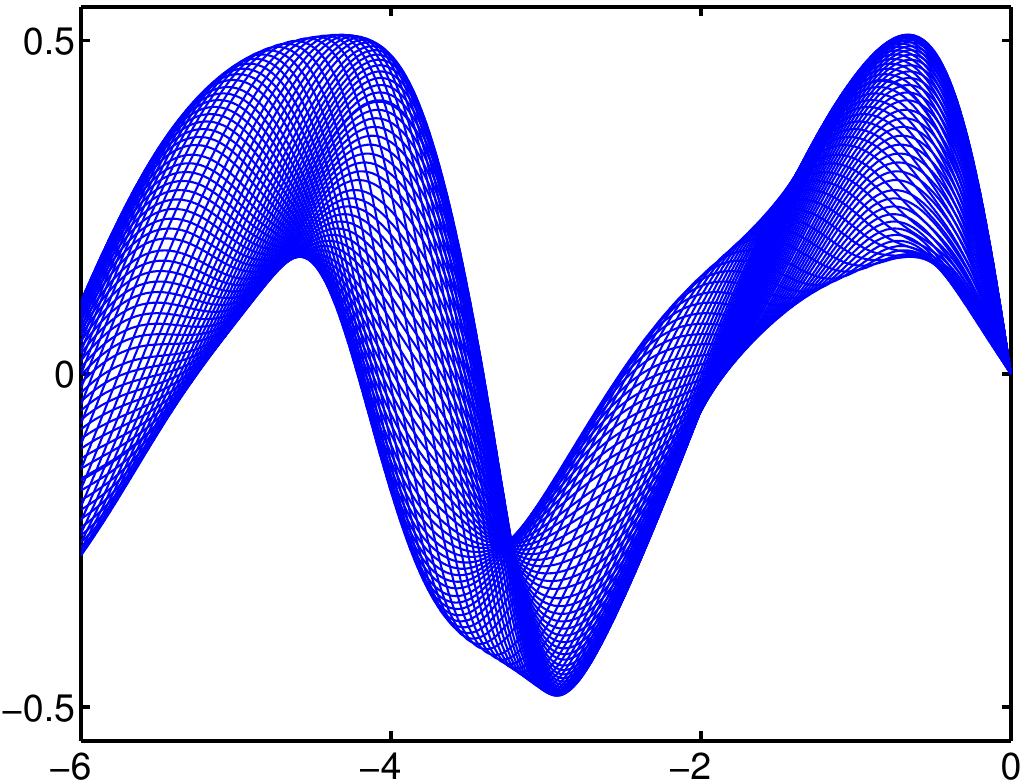}}\scalebox{0.4}{\hspace*{2em}\includegraphics{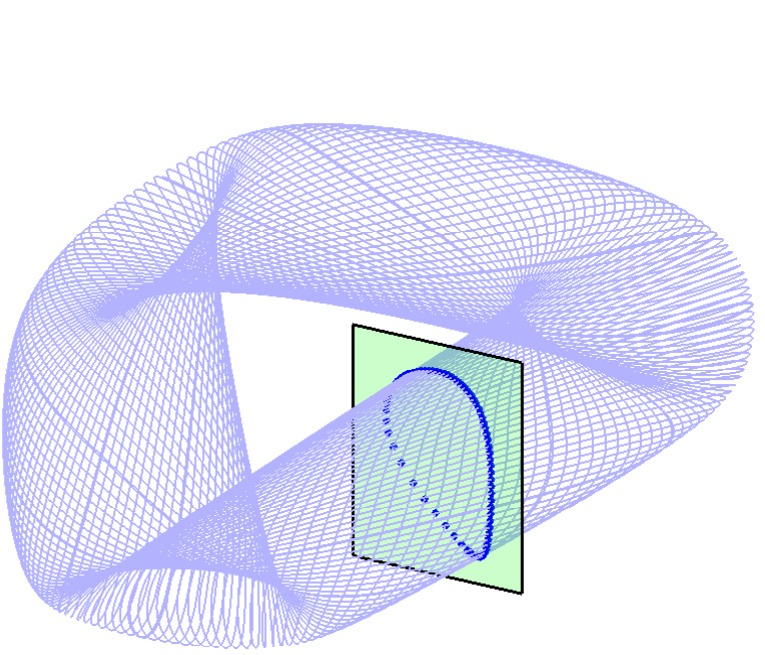}}
\put(-70,0){\small $u_t(0)=u(t)=0$}
\put(-175,-5){\small $t$}
\put(-220,90){\small (a)}
\put(-140,90){\small (b)}
\put(-295,75){\rotatebox[origin=c]{90}{\small $u_t(\theta)$}}
	\caption{Poincar\'e section $u_t(0)=0$ of the quasi periodic torus from Figure~\ref{fig:tori} generated from a single solution of \eqref{eq:twostatedep} with $\kappa_1=4.44$ and $\kappa_2=3$. 
(a) Solution segments $u_t$ on the Poincar\'e section. (b) Projection of the solution onto $(u(t),u(t-a_1),u(t-a_2))$-space (light blue), together with the two-dimensional projection of the Poincar\'e section. (\textcopyright SIAM; Reproduced from \cite{CHK17} with permission.)}
	\label{fig:poincare} 
\end{figure}
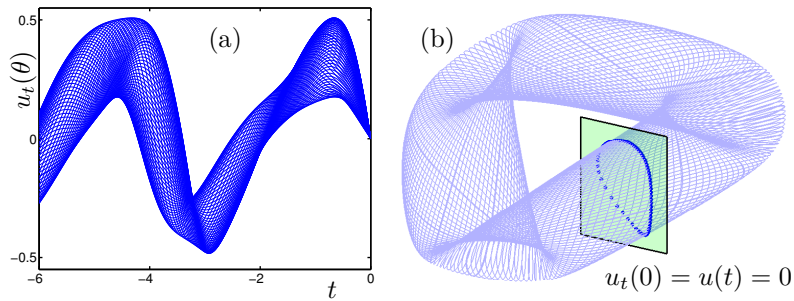

Solutions oscillate about the steady state $u^*=0$, so a natural Poincar\'e section to study the dynamics is
$\bigl\{\phi\in C:\phi(0)=0$ and $\phi'^{\!}(0)<0\bigr\}$. For finite dimensional dynamical systems the Poincar\'e section reduces the dimension by one, but $C$ is infinite dimensional, resulting in an infinite dimensional Poincar\'e section. 
Plotting function segments on the Poincar\'e section with $u_t(0)=0$ is not very
revealing, as seen in Figure~\ref{fig:poincare}(a). 

In a variation on the approach of \cite{GlassMackey79} described in Section~\ref{sec:dynsys}, with state-dependent
delays \cite{CHK17} found it more convenient to visualise solution in three dimensions 
by projecting $u_t$ onto $(u(t),u(t-a_1),u(t-a_2))=(u(t),u(t-\tau_1(u^*)),u(t-\tau_2(u^*)))$ rather than using the 
state-dependent delays $(u(t),u(t-\tau_1(u(t))),u(t-\tau_2(u(t))))$. Applying the Poincar\'e section
$u_t(0)=0$ to $(u(t),u(t-a_1),u(t-a_2))$ is the same as requiring $u(t)=0$, and so we can represent
function segments on the Poincar\'e section in $\R^2$ by $(u_t(-a_1),u_t(-a_2))=(u(t-a_1),u(t-a_2))$,
by plotting $u(t-a_2)$ against $u(t-a_1)$ when $u(t)=0$ (with $\udot(t)<0$). This is illustrated in 
Figures~\ref{fig:poincare}(b) and~\ref{fig:chk:phaselocked}(d).

\begin{figure}[tp!]
	\centering	
	\includegraphics[scale=0.77]{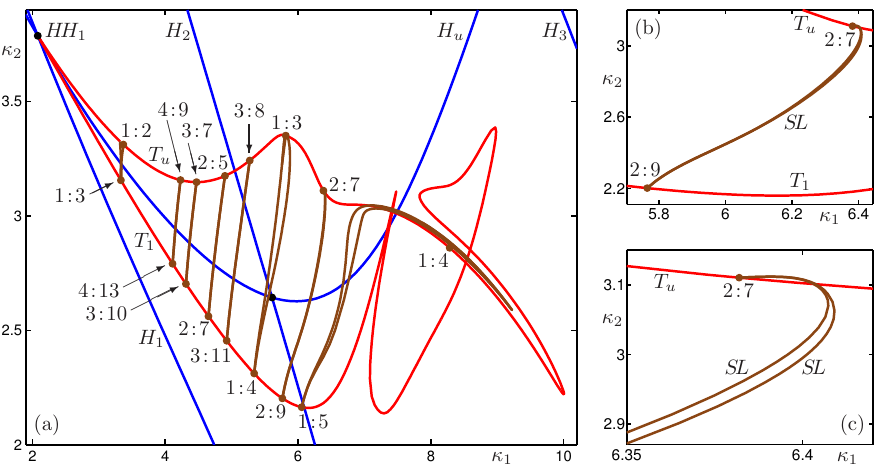}
	\caption{Two-parameter continuation of Hopf (blue), torus (red) and fold (brown) bifurcations in the $(\kappa_1,\kappa_2)$-parameter plane show many resonance tongues of phase locked tori, with successive enlargements if one of the tongues in (b) and (c). (\textcopyright SIAM; Reproduced from \cite{CHK17} with permission.)}
	\label{fig:chk:2param} 
\end{figure}

Figure~\ref{fig:chk:2param} shows a two parameter bifurcation diagram for \eqref{eq:twostatedep}
which reveals multiple double-Hopf bifurcations. Two curves of torus bifurcations emerging from the double-Hopf point $\textit{HH}_1$ give rise to the tori encountered already. Numerous resonance tongues of phase locked tori are also shown. One of the larger resonance tongues is explored in Figure~\ref{fig:chk:phaselocked}.

\begin{figure}[tp!]
	\centering	
	\includegraphics[scale=0.77]{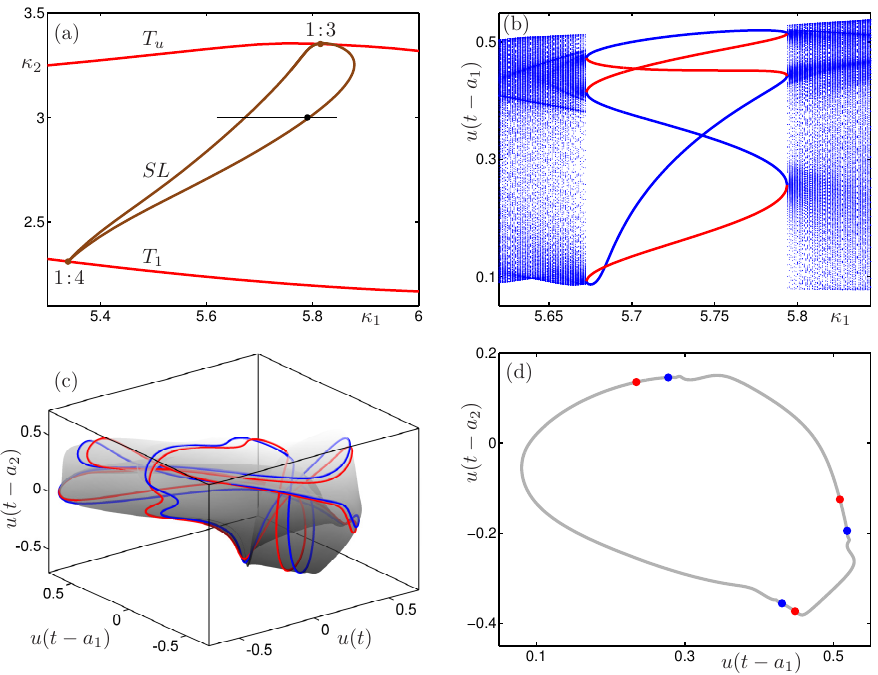}
	\caption{(a) An enlargement of Figure~\ref{fig:chk:2param} showing one resonance tongue. (b) A one parameter continuation
with $\kappa_2=3$ (corresponding to black line in (a) showing values of $u(t-a_1)$ at each intersection of the orbit with the Poincar\'e section. (c \& d) The phase-locked torus when $(\kappa_1,\kappa_2)=(5.79,3)$ (black dot in (a)) with the embedded unstable and stable periodic orbits (c) projected onto $(u(t),u(t-a_1),u(t-a_2))$-space and (d) its Poincar\'e
trace in the $(u(t-a_1),u(t-a_2))$ plane. (\textcopyright SIAM; Reproduced from \cite{CHK17} with permission.)}
	\label{fig:chk:phaselocked} 
\end{figure}

Stable quasi-periodic tori, such as seen earlier in Figure~\ref{fig:poincare}, are easy to compute as solutions of the DDE IVP converge to them and fill the torus. Taking an initial function that is close to the unstable steady state (or better a function close to the unstable periodic orbit) and throwing away the initial transient reveals the torus in that case.
The computation of phase-locked tori, as in Figure~\ref{fig:chk:phaselocked}, is altogether more involved.

To compute the phase locked torus shown in Figure~\ref{fig:chk:phaselocked} we first solve \eqref{eq:twostatedep} as an IVP
(using ddesd) which reveals only the stable periodic orbit which is embedded in the phase locked torus. Then a one parameter 
continuation of the periodic orbit using DDE-Biftool finds both the fold bifurcations at the edge of the resonance tongues and the second unstable periodic orbit embedded in the torus. This is how the red and blue curves in Figure~\ref{fig:chk:phaselocked}(b)-(d) are computed.
A two-parameter continuation of the fold bifurcations reveals the resonance tongue shown in  Figure~\ref{fig:chk:phaselocked}(a).
 To fill out the rest of the torus (the grey surface and curves in 
Figure~\ref{fig:chk:phaselocked}(c)-(d)) we compute the unstable manifold of the unstable periodic orbit. Since the torus itself is stable, this is embedded within the torus and fills the strip between the unstable and stable periodic orbits. See \cite{CHK17} for details of the computations.

\begin{figure}[tp!]
	\centering	
	\hspace*{-0.2em}\includegraphics[scale=0.7]{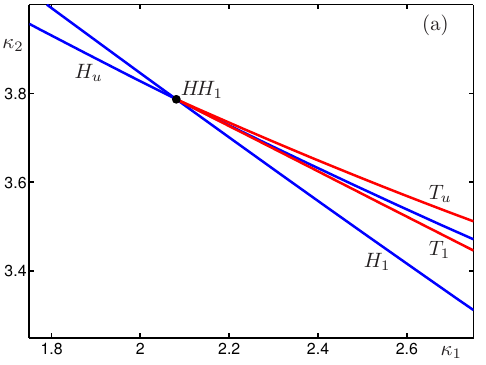}\hspace*{-0em}\includegraphics[scale=0.7]{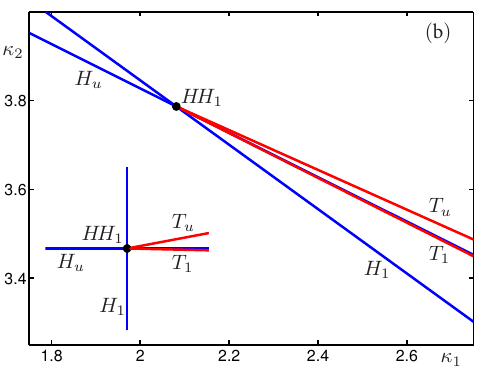}
	\caption{(a) Detail of the bifurcation diagram from Figure~\ref{fig:chk:2param}(a) near to the double-Hopf point $\textit{HH}_1$. (b) Inset: Normal form of the bifurcation in the $(\mu_1,\mu_2)$ normal form coordinates obtained by computing the normal form coefficients at the Hopf-Hopf point. Main panel: The normal form transformed to the original 
$(\kappa_1,\kappa_2)$ parameters. (\textcopyright SIAM; Reproduced from \cite{CHK17} with permission.)}
	\label{fig:chk:nf} 
\end{figure}

In Figure~\ref{fig:chk:nf} we briefly demonstrate the efficacy of the normal form computations of \cite{CHK17} for double-Hopf bifurcations. Figure~\ref{fig:chk:nf}(a) shows the numerically computed bifurcation diagram near to $\textit{HH}_1$ for the full nonlinear state-dependent delay DDE \eqref{eq:twostatedep}. The nearly identical Figure~\ref{fig:chk:nf}(b) is obtained from the normal form computations, first using the techniques of Section~\ref{sec:linearize} to convert \eqref{eq:twostatedep} into the 
9 constant delays DDE \eqref{eq:2ssconst}, then using established techniques for constant delay DDEs
to reduce this to a 4-dimensional ODE on the centre manifold, and finally applying the ODE theory to obtain the normal form.

The computed bifurcation normal form in the inset of Figure~\ref{fig:chk:nf}(b) looks quite different to the original bifurcation in Figure~\ref{fig:chk:nf}(a) because 
the normal form computations introduce new variables so that the Hopf bifurcations occur on the coordinate axes. Equation (A.52) in \cite{CHK17} can be used to map back to the original $(\kappa_1,\kappa_2)$ coordinates. After making this transformation the two panels of Figure~\ref{fig:chk:nf} look remarkably similar, despite the left panel being computed from the full nonlinear dynamics of \eqref{eq:twostatedep} over a range of 
$(\kappa_1,\kappa_2)$ values, while the right panel is computed only from the normal form coefficients of the constant delay DDE \eqref{eq:2ssconst} at the single point  $\textit{HH}_1$! This all works without any rigourous theory for co-dimension-two double Hopf bifurcations in state-dependent delay DDEs (in part because it has not been shown that the centre manifold is sufficiently smooth). Indeed it works so well that we are convinced that there is an analytical proof somewhere in the ether waiting to be discovered or constructed. With that we leave the reader with something to think about.

%
%
%
%
%
%
%
%
%
%
\bibliographystyle{plainnat}

\end{document}